\documentclass[a4paper]{amsart}
\usepackage{amsmath, amsthm, amsfonts, amscd, amssymb, MnSymbol}
\usepackage{xcolor}
\usepackage{enumitem}
\usepackage{tikz-cd}
\usepackage{hyperref}
\usepackage[normalem]{ulem}
\usepackage{leftindex}

\newtheorem{thm}{Theorem}[section]
\newtheorem*{thma}{Theorem A}
\newtheorem*{thmb}{Theorem B}
\newtheorem*{thmc}{Theorem C}
\newtheorem{cor}[thm]{Corollary}

\newtheorem{lem}[thm]{Lemma}
\newtheorem{rmk}[thm]{Remark}
\newtheorem{prop}[thm]{Proposition}

{\theoremstyle{remark}}

\numberwithin{equation}{section}

\usepackage{biblatex}
\newcommand{\Ker}{\operatorname{Ker}} 
\newcommand{\Img}{\operatorname{Im}}
 
\newcommand{\id}{\operatorname{id}}

\newcommand{\Span}{\operatorname{Span}}
\newcommand{\Spec}{\operatorname{Spec}}
\newcommand{\End}{\operatorname{End}}
\newcommand{\Irr}{\operatorname{Irr}}
\newcommand{\SSp}{*\text{-}\operatorname{Irr}}
\newcommand{\Prim}{\operatorname{Prim}}
\newcommand{\SPrim}{*\text{-}\operatorname{Prim}}
\newcommand{\Max}{\operatorname{Max}}
\newcommand{\Pol}{\operatorname{Pol}}
\newcommand{\ev}{\operatorname{ev}}

\newcommand{\Nbb}{\mathbb N}
\newcommand{\Cbb}{\mathbb C}
\newcommand{\Zbb}{\mathbb Z}
\newcommand{\Rbb}{\mathbb R}

\newcommand{\Tbb}{\mathbb{T}}

\newcommand{\Lbb}{\mathbb{L}}

\newcommand{\Acal}{\mathcal{A}}
\newcommand{\Bcal}{\mathcal{B}}

\newcommand{\Hcal}{\mathcal{H}}
\newcommand{\Ical}{\mathcal{I}}

\newcommand{\Lcal}{\mathcal{L}}

\newcommand{\Ocal}{\mathcal{O}}

\newcommand{\Scal}{\mathcal{S}}
\newcommand{\Rcal}{\mathcal{R}}

\newcommand{\Tcal}{\mathcal{T}}

\newcommand{\gfrak}{\mathfrak{g}}
\newcommand{\hfrak}{\mathfrak{h}}
\newcommand{\kfrak}{\mathfrak{k}}

\newcommand{\Cfrak}{\mathfrak{C}}
\newcommand{\tfrak}{\mathfrak{t}}
\newcommand{\nfrak}{\mathfrak{n}}
\newcommand{\bfrak}{\mathfrak{b}}
\newcommand{\Dfrak}{\mathfrak{D}}
\newcommand{\gf}{\mathfrak{g}}
\newcommand{\hf}{\mathfrak{h}}
\newcommand{\kf}{\mathfrak{k}}
\newcommand{\Cf}{\mathfrak{C}}
\newcommand{\tf}{\mathfrak{t}}
\newcommand{\Df}{\mathfrak{D}}

\newcommand{\Qbf}{\mathbf{Q}}
\newcommand{\Pbf}{\mathbf{P}}

\newcommand{\BDelta}{\boldsymbol{\Delta}}

\newcommand{\wrm}{\mathrm{w}}

\newcommand{\roots}{\mathbf{Q}}

\newcommand{\weights}{\mathbf{P}}

\begin{document}

\title[Complexification of $ q $-deformations]{On the complexification of $ q $-deformations of compact semisimple Lie groups}

\author{Heon Lee}
\address[1]{Institute for Advanced Study in Mathematics, Harbin Institute of Technology \\
Harbin 150001 \\
China
}
\address[2]{Research Institute of Mathematics, Seoul National University \\
Gwanak-ro 1, Gwanak-gu, Seoul 08826 \\
Republic of Korea
}
\email{heon.lee.math@gmail.com}

\author{Christian Voigt} 
\address{School of Mathematics and Statistics \\
         University of Glasgow \\
         University Place \\
         Glasgow G12 8QQ \\
         United Kingdom 
}
\email{christian.voigt@glasgow.ac.uk}


\begin{abstract}
For the quantised algebra of functions on a simply connected compact semisimple Lie group, we study the relationship between the primitive spectrum, topologically irreducible representations by bounded operators on Hilbert spaces, and quantum holomorphic evaluations induced via the Drinfeld double. In terms of the Joseph stratification of the primitive spectrum, we prove a sharp dichotomy: primitive ideals in homogeneous strata always arise as kernels of topologically irreducible representations by bounded operators on a Hilbert space, whereas those in inhomogeneous strata do not admit any such realisations. We further show that, unlike in the classical situation, quantum holomorphic evaluations typically fail to be topologically irreducible, and that their kernels do not determine the primitive spectrum. Finally, we describe the Iwasawa decomposition of the Drinfeld double in terms of polynomial functions, extending work by Podleś-Woronowicz for the quantum Lorentz group to all complex semisimple quantum groups, and clarifying the structure of quantum holomorphic evaluations.
\end{abstract}

\maketitle

\section{Introduction}

A fundamental feature underpinning classical representation theory is the fact that there is a bijective correspondence between compact semisimple Lie groups and complex semisimple Lie groups given by \emph{complexification}. The aim of this paper is to study quantum analogues of this correspondence for quantised function algebras on $ q $-deformations, in different incarnations.

To motivate this, recall first that there is an elegant approach to the complexification of a compact connected Lie group $ K $ in terms of the Hopf $ * $-algebra $ \Ocal^\mathbb{R}(K) $ of matrix coefficients of finite-dimensional unitary representations of \(K\), see \cite{Chevalley}. Indeed, forgetting the $ * $-structure allows one to view this Hopf algebra as the algebra of representative functions on a complex affine algebraic group, and hence a complex Lie group, and this group is precisely the complexification $ G = K_\mathbb{C} $ of $ K $. That is, the Hopf algebra $ \Ocal(G) = \Ocal^\mathbb{R}(K) $ encodes both $ K $ and $ G $ at the same time, and the passage from $ G $ to $ K $ amounts to incorporating the $ * $-structure on $ \Ocal(G) $ corresponding to the compact real form of $ G $. 
By construction, the evaluation homomorphisms at points of $ G $ correspond to the maximal ideals of $ \Ocal(G) $, or equivalently, the irreducible algebra representations of $ \Ocal(G) $, which are precisely its characters. The points of $ K $ correspond to $ * $-algebra characters, that is, characters compatible with the $ * $-structure of $\Ocal^\mathbb{R}(K) $. 

A similar picture remains valid in the quantum setting, and provides the link between Woronowicz's theory of compact quantum groups and the purely algebraic theory of quantised function algebras and 
quantised universal enveloping algebras. Namely, the standard $ q $-deformation $ K_q $ of a simply connected compact semisimple Lie group is encoded by a Hopf $ * $-algebra $ \Ocal^\mathbb{R}(K_q) $, whose underlying Hopf algebra $ \Ocal(G_q) $ can be viewed as a quantisation of $ \Ocal(G) $. By the work of Hodges-Levasseur \cite{Hodges1993}, Joseph \cite{Joseph} and others, the structure of the primitive spectrum of $ \Ocal(G_q) $ fits nicely with the Poisson geometry of $ G $ at the semiclassical level, see \cite{Yakimov2014} for a detailed account. Note that, in view of the analogy with point evaluations in the classical case, one may regard the primitive spectrum of $ \Ocal(G_q) $ as an incarnation of the complexification of $ K_q $.  

There is another approach to the complexification of $ K_q $ using operator algebras. This builds on the work of Podle\'s-Woronowicz on the quantum Lorentz group \cite{Podles1990}, and starts from the $ C^* $-completion $ C(K_q) $ of $ \Ocal^\mathbb{R}(K_q) $. In this picture, the complexification $ G_q $ is encoded by the Drinfeld double of $ K_q $.  Our main aim is to clarify the connection between the resulting analytic approach to complexification and its purely algebraic counterpart in terms of Hopf $ * $-algebras. At the level of quantised universal enveloping algebras and convolution algebras, there is indeed a close correspondence \cite{VoigtYuncken}, but the picture is less clear on the dual side, namely at the level of function algebras. 

We first study bounded Hilbert space representations of the quantised function algebra $ \Ocal(G_q) = \Ocal^\mathbb{R}(K_q) $. Vaksman-Soibelman and Soibelman \cite{VAKSMAN_SOIBELMAN_functionsquantumsu2}, \cite{SOIBELMAN_functionscompactquantumgroups} classified all topologically irreducible $ * $-representations, see also \cite{KorogodskiSoibelman}. We consider the more general case of algebra representations by bounded operators on a Hilbert space which are not necessarily $ * $-preserving. To formulate our result, recall that the Joseph stratification 
\[
\Prim(\Ocal(G_q))
=
\coprod_{\wrm\in W\times W}
\Prim_{\wrm}(\Ocal(G_q))
\]
of the set $ \Prim(\Ocal(G_q)) $ of primitive ideals of $ \Ocal(G_q) $ is parametrised by pairs of elements of the Weyl group $ W $. We shall say that $ \Prim_{\wrm}(\Ocal(G_q)) $ is \emph{homogeneous} if $ \wrm $ is of the form $ \wrm = (w,w) $ for some $ w \in W $, and \emph{inhomogeneous} otherwise. 
Using this terminology we establish the following dichotomy. 

\begin{thma} 
Primitive ideals in homogeneous Joseph strata can always be realised as kernels of topologically irreducible representations by bounded operators on some Hilbert space, whereas primitive ideals in inhomogeneous strata do not admit any such realisations. 
\end{thma}

For the details regarding Theorem A and its proof we refer to Theorem \ref{thm:primitive ideals in homogeneous strata} and Theorem \ref{thm:no bdd repn for inhomogeneous strata} below. In fact, we show that there are no bounded representations of $ \Ocal(G_q) $ on Hilbert spaces, topologically irreducible or not, whose kernels lie in \(\Prim_{(w_+,w_-)}(\Ocal(G_q))\) with \(w_+ \neq w_-\).

Next we study the link between the primitive spectrum of $ \Ocal(G_q) $ and representations of  function algebras associated to the complexification of $ K_q $ at the level of $ C^* $-algebras, given by the Drinfeld double of $ K_q $. Following \cite{VoigtYuncken}, we shall use the notation $ \Cf^\infty(K_q) = \Ocal^\mathbb{R}(K_q) $ for the Hopf $ * $-algebra underlying the compact quantum group $ K_q $ in this context. The $ C^* $-algebra $ C_0(G_q) $ of functions on the Drinfeld double $ G_q $ is then obtained as the $ C^* $-completion of
\[
\Cf_c^\infty(G_q)=\Cf^\infty(K_q)\otimes \Dfrak(K_q),
\]
where $ \Df(K_q) $ is the dual of $ \Cf^\infty(K_q) $ in the sense of algebraic quantum groups \cite{VanDaele}. By construction, nondegenerate $ * $-representations of $ C_0(G_q) $ are in bijective correspondence with nondegenerate $ * $-representations of $ \Cf^\infty_c(G_q) $. Moreover, one obtains a canonical identification 
\[
\SSp(\Cf_c^\infty(G_q)) \cong \SSp(\Cf^\infty(K_q)) \times \SSp(\Df(K_q)) 
\]
of the corresponding sets of irreducible $ * $-representations. In particular, every irreducible $ * $-representation of $ \Cf_c^\infty(G_q) $ is of the form $ \pi = \pi_C \otimes \pi_\lambda $ where $ \pi_C \in \SSp(\Cf^\infty(K_q)) $ and $ \lambda \in \weights^+ $. 

Writing $ \Cf^\infty(G_q) $ for the algebraic multiplier algebra of $ \Cf_c^\infty(G_q) $, there exists a unital algebra embedding 
\[
i_-: 
\Ocal(G_q) \rightarrow
\Cfrak^\infty(G_q),
\]
which can be interpreted as realising holomorphic polynomials on $ G_q $ as unbounded multipliers of $ C_0(G_q) $.  
We denote by $ \Pol^\Rbb(G_q) $ the $ * $-subalgebra generated by the image of $ i_- $. This is the quantum analogue of the $ * $-algebra of holomorphic and antiholomorphic polynomials on the complex Lie group $ G $.

The relation between the $ * $-representations of $ \Pol^\Rbb(G_q) $ and $ \Cfrak^\infty_c(G_q) $, and hence also of $ C_0(G_q) $, can be described as follows.  

\begin{thmb} 
Every unital $ * $-representation of \(\Pol^\Rbb(G_q)\) by bounded operators on a Hilbert space extends uniquely to a unital $ * $-representation of \(\Cf^\infty(G_q)\), which in turn restricts to a nondegenerate $ * $-representation of \(\Cf_c^\infty(G_q)\).
\end{thmb}

We also characterise precisely which nondegenerate $ * $-representations of $ \Cf_c^\infty(G_q) $ arise this way. 
Our result, see Theorem \ref{thm:generalisation of Podles and Woronowicz} below, generalises \cite[Theorem~5.4.4]{Podles1990} from the quantum Lorentz group $ SL_q(2, \Cbb) $ to arbitrary \(G_q \). In fact, in the case \( G_q = SL_q(2, \Cbb) \), the unital $ * $-representations of \(\Pol^\Rbb(SL_q(2, \Cbb)) \) correspond to \( SL_q(2, \Cbb) \)-matrices in the sense of \cite{Podles1990}, and this correspondence was the starting point of the work of Podle\'s-Woronowicz.

In analogy to the classical case, we define \emph{quantum holomorphic evaluations} of \(\Ocal(G_q)\) at ``points'' of \(G_q\) to be the bounded algebra representations of \(\Ocal(G_q)\) of the form 
\[
\theta_\pi: \Ocal(G_q) \xrightarrow{i_-} \Pol^\Rbb(G_q) \xrightarrow{\pi} \Lbb(\Hcal_\pi)
\]
for $ \pi \in \SSp(\Pol^\Rbb(G_q)) $, the set of all topologically irreducible $ * $-representations of $ \Pol^\Rbb(G_q) $. In contrast to classical point evaluations, it turns out that quantum holomorphic evaluations are typically reducible, and their kernels usually fail to be primitive ideals, see Proposition \ref{prop:kernels of q-holomorphic evaluations are not primitive} and Corollary \ref{cor:kernels of quantum holomorphic evaluations are not primitive-general case}.
At best, one may hope to connect them with \(\Prim(\Ocal(G_q))\) through the Zariski-closed subsets
\[
\Prim(\theta_\pi)
=
\{P \in \Prim(\Ocal(G_q)) \mid P \supseteq \Ker(\theta_\pi)\}
\]
associated to $ \pi \in \SSp(\Pol^\Rbb(G_q)) $. We show, however, that these subsets only detect a limited amount of information about $ \Prim(\Ocal(G_q)) $. 

\begin{thmc} 
Let $ \pi = \pi_C \otimes \pi_\lambda \in \SSp(\Pol^\Rbb(G_q)) $, where $ \pi_C $ is a topologically irreducible $ * $-representation of $ \Cf^\infty(K_q) $ and $ \lambda \in \weights^+ $. If $ \Ker(\pi_C) \in \Prim_{(w,w)}(\Ocal(G_q)) $ then 
\[
\Prim(\theta_\pi) \cap \Prim_{(w,w)}(\Ocal(G_q)) \subset \Prim_{(w,w)}(\Ocal(G_q))
\]
is finite. Moreover 
\[
\bigcup_{\pi \in \SSp(\Pol^\Rbb(G_q))} \Prim(\theta_\pi) \subsetneqq \Prim(\Ocal(G_q))
\]
is a strict inclusion. 
\end{thmc}

This result clarifies to what extent the primitive spectrum of \(\Ocal(G_q)\) can be accessed via the Drinfeld double $ G_q $. For the details see Theorem~\ref{thm:the best that Ker theta pi can do}, Corollary~\ref{cor:kernels of quantum holomorphic evaluations are not primitive-general case}, and Remark~\ref{rmk:embedding of SPrim(Cc(Gq)) into Prim(C(Kq))}. 
In summary, in the quantum setting, the link between complexification in the algebraic and the analytic approach is significantly weaker than in the classical case. 

Let us explain how the paper is organised. Section~\ref{sec:preliminaries} gathers some background material on \(q\)-deformations of compact semisimple Lie groups and on the primitive spectrum of $ \Ocal(G_q) $. Section~\ref{sec:Irreducible representations and primitive spectrum} is devoted to the relationship between topologically irreducible Hilbert space representations of $ \Ocal(G_q) $ and the Joseph stratification of the primitive spectrum. We use this to complement our results on Beurling-Fourier algebras in \cite{LeeVoigt_BeurlingFourier} from the point of view of primitive ideals. Finally, in Section~\ref{sec:Complex semisimple quantum groups and primitive spectrum} we treat the irreducible representations of function algebras on Drinfeld doubles and discuss the relationship between quantum holomorphic evaluations and the primitive spectrum. 

We conclude with a few remarks on notation. We write \( \Nbb = \{ 0, 1, 2, \ldots \} \) for the set of natural numbers. All vector spaces, linear maps, and algebraic tensor products in this paper are over the field of complex numbers \( \Cbb \). Inner products on Hilbert spaces are conjugate-linear in the first argument. 
We write \( \Lbb(\Hcal) \) for the algebra of bounded operators on a Hilbert space \( \Hcal \). The orthogonal complement of a subset $ X $ of a Hilbert space is denoted by $ X^\perp $. 
If $ X $ is a subset of a Banach space we write $ [X] $ for the closed linear span of $ X $. 

We thank Woo Young Lee for communicating to us the current state of the art concerning the invariant subspace problem for separable infinite-dimensional Hilbert spaces. We also thank Keita Goto and Kan Kitamura for helpful comments on a draft version of our work.  

\section{Preliminaries}\label{sec:preliminaries}

In this section, we discuss some preliminary material that will be used throughout the paper. We freely use results from \cite{Joseph, KorogodskiSoibelman, Yakimov2014, VoigtYuncken}, translated into the notation and conventions of \cite{VoigtYuncken}.

\subsection{Representations} 

Throughout this paper we work with algebras over the complex numbers, not necessarily unital. By definition, an algebra homomorphism between algebras $ \Acal $ and $ \Bcal $ is a linear map $ f: \Acal \to \Bcal $ such that $ f(ab) = f(a) f(b) $ for all \( a,b \in \Acal \). In particular, algebra homomorphisms need not preserve the units if the source and target are unital. A unital algebra homomorphism between unital algebras is an algebra homomorphism \( f: \Acal \to \Bcal \) such that \(f(1_\Acal) = 1_\Bcal \). 

By a \emph{representation} of an algebra \( \Acal \) we mean an algebra homomorphism \(\pi: \Acal \to \End(V) \), where \(\End(V) \) is the algebra of linear endomorphisms of a complex vector space \( V \). We sometimes also write $ a \cdot v = \pi(a)(v) $ for the action of $ a \in \Acal $ on $ v \in V $. For a subset \( X \subseteq V \) we denote by \( \pi(\Acal) X = \Acal \cdot X \) the linear span of all \( \pi(a) x \) with \( a \in \Acal, \, x \in X \). 
A representation $ \pi: \Acal \rightarrow \End(V) $ is called (algebraically) \emph{nondegenerate} if $ \pi(\Acal) V = V $. If $ \Acal $ is unital then $ \pi $ is nondegenerate if and only if it is unital. 
An \emph{intertwiner} between representations $ \pi_1: \Acal \rightarrow \End(V_1) $ and $ \pi_2: \Acal \rightarrow \End(V_2) $ is a linear map $ T: V_1 \rightarrow V_2 $ such that $ \pi_2(a) T = T \pi_1(a) $ for all $ a \in \Acal $. We say that $ \pi_1, \pi_2 $ are \emph{equivalent} if there exists an invertible intertwiner between them. 

A nonzero representation \(\pi: \Acal \to \End(V) \) on a vector space $ V $ is called \emph{irreducible} if $ V $ admits no nontrivial invariant linear subspaces. Note that every irreducible representation is automatically nondegenerate. We write $ \Irr(\Acal) $ for the set of equivalence classes of irreducible representations of $ \Acal $. A representation is called reducible if it is not irreducible. 
An ideal \(I \subseteq \Acal \) is called \emph{primitive} if it is the kernel of an irreducible representation of \(\Acal \). We write $ \Prim(\Acal) $ for the set of all primitive ideals of $ \Acal $. 

By a \emph{bounded representation} of an algebra \(\Acal \) we mean an algebra homomorphism \( \pi: \Acal \to \Lbb(\Hcal)\), where \( \Lbb(\Hcal) \) is the algebra of bounded operators on a Hilbert space \( \Hcal \). It is called (topologically) \emph{nondegenerate} if $ [\pi(\Acal) \Hcal] = \Hcal $. 
A nonzero bounded representation \( \pi: \Acal \to \Lbb(\Hcal) \) on a Hilbert space $ \Hcal $ is called \emph{topologically irreducible} if \( \Hcal \) admits no nontrivial closed invariant subspaces. Such $ * $-representations are automatically nondegenerate. Note that a topologically irreducible representation typically fails to be irreducible. In particular, any bounded representation of a finitely generated algebra on an infinite-dimensional Hilbert space is reducible for cardinality reasons. We shall call an ideal $ I \subseteq \Acal $ \emph{topologically primitive} if it is the kernel of a topologically irreducible bounded representation. Two bounded representations $ \pi_1: \Acal \rightarrow \Lbb(\Hcal_1) $ and $ \pi_2: \Acal \rightarrow \Lbb(\Hcal_2) $ are called \emph{boundedly equivalent} if there exists an invertible bounded linear operator $ T: \Hcal_1 \rightarrow \Hcal_2 $ such that $ \pi_2(a) T = T \pi_1(a) $ for all $ a \in \Acal $. 

If \( \Acal \) is a \( * \)-algebra, then a \( * \)-representation of \( \Acal \) is a \( * \)-homomorphism \( \pi : \Acal \rightarrow \Lbb(\Hcal) \) for some Hilbert space \( \Hcal \). Clearly, every $ * $-representation is a bounded representation. Two $ * $-representations are called \emph{unitarily equivalent} if there exists a unitary intertwiner between them. We write $ \SSp(\Acal) $ for the set of equivalence classes of topologically irreducible $ * $-representations of $ \Acal $ modulo unitary equivalence. In a similar way, we denote by $ \SPrim(\Acal) $ the set of kernels of topologically  irreducible $ * $-representations of $ \Acal $. 

Recall that an ideal $ P \subsetneqq \Acal $ is \textit{prime} if, for any two ideals $ I, J \subseteq \Acal $, the condition $ IJ \subseteq P $ implies $ I \subseteq P $ or $ J \subseteq P $. We write $\Spec(\Acal)$ for the set of all prime ideals of $ \Acal $. An ideal $ P $ in $ \Acal $ is called \emph{completely prime} if $ \Acal/P $ is a domain, that is, if $ \Acal/P $ has no nontrivial zero divisors. 
Every completely prime ideal is prime, 
and every primitive ideal is prime. 
Similarly, every topologically primitive ideal is prime.

\subsection{The \texorpdfstring{$q$}{q}-deformation of a compact semisimple Lie group}

Throughout, we consider a simply connected compact semisimple Lie group \( K \) and let \( G \) be its complexification, with Lie algebras denoted by \( \kfrak \) and \( \gfrak \), respectively. We fix a maximal torus \( T \subseteq K \) with Lie algebra \( \tfrak \), whose complexification is denoted by \( H \subseteq G \), and a set of positive roots \( \BDelta^+ \) inside the set of roots \( \BDelta \) associated with the Cartan subalgebra \( \hfrak = \Cbb \otimes_\Rbb \tfrak \) of \( \gfrak \). Let \( \Sigma = \{ \alpha_1, \dots, \alpha_N \} \subseteq \BDelta^+ \) be the resulting set of simple roots, and let \( \{ \varpi_1, \dots, \varpi_N \} \) be the corresponding set of fundamental weights, defined with respect to the bilinear form on \( \hfrak^* \) obtained by rescaling the Killing form so that in each simple factor the shortest root \( \alpha \) satisfies \( (\alpha, \alpha) = 2 \). The subgroups of \( (i \tfrak)^* \) generated by \( \{ \alpha_1, \dots, \alpha_N \} \) and \( \{ \varpi_1, \dots, \varpi_N \} \) are denoted by \( \Qbf \) and \( \Pbf \), respectively, and we write \( \Qbf^+ \) and \( \Pbf^+ \) for the subsets consisting of nonnegative integer linear combinations of the respective generators. We have
\[
\rho = \frac{1}{2} \sum_{\alpha \in \BDelta^+} \alpha = \sum_{1 \leq j \leq N} \varpi_j \in \weights^+.
\]
For each \( 1 \leq i \leq N \), let \( d_i = \frac{(\alpha_i, \alpha_i)}{2} \) and define
\[
\alpha_i^\vee = d_i^{-1} \alpha_i, \quad \varpi_i^\vee = d_i^{-1} \varpi_i.
\]
The Cartan matrix for $\gf$ is given by the matrix $(a_{ij})_{1 \leq i, j \leq N}$ with
\[
a_{ij} = (\alpha_i ^\vee , \alpha_j) , \qquad 1\leq i , j \leq N.
\]
We denote the subgroups of \( (i \tfrak)^* \) generated by \( \{ \alpha_1^\vee, \dots, \alpha_N^\vee \} \) and \( \{ \varpi_1^\vee, \dots, \varpi_N^\vee \} \) by \( \Qbf^\vee \) and \( \Pbf^\vee \), respectively. By construction, \( \Qbf^\vee \) and \( \Pbf^\vee \) are the \( \Zbb \)-duals of \( \Pbf \) and \( \Qbf \), respectively, via the bilinear form \( ( \cdot, \cdot ) \). As usual, we endow \( \Pbf \) with the partial ordering such that \( \lambda \leq \mu \) if and only if there exists \( \nu \in \Qbf^+ \) such that \( \mu = \lambda + \nu \).

Let \( W \) be the Weyl group associated with the root system \( \BDelta \). We denote by \( s_\alpha \) the reflection with respect to \( \alpha \in \BDelta^+ \). The simple reflections \( s_i = s_{\alpha_i} \) for \( 1 \leq i \leq N \) generate \( W \), that is, each \( w \in W \) can be expressed as a product of simple reflections. We denote by \( \ell(w) \in \mathbb{N} \) the length of any reduced expression of $ w \in W $ with respect to these generators. Moreover we write \( w_0 \in W \) for the element with the longest reduced expression. It is uniquely determined and satisfies \( w_0^2 = e \in W \), where $e \in W$ is the unit element.

For \( w, w' \in W \) and \( \alpha \in \BDelta^+ \), we define
\[
w \xleftarrow{\alpha} w' \Longleftrightarrow \left( w = s_\alpha w' \text{ and } \ell(w) = \ell(w') - 1 \right).
\]
The Bruhat order on \( W \) is defined by
\[
w \leq w' \Longleftrightarrow \left( w = w' \text{ or there exists a chain } w \xleftarrow{} \cdots \xleftarrow{} w' \right),
\]
and we note that $e \leq w \leq w_0$ for all $w \in W$.

Now fix a deformation parameter \( 0 < q < 1 \). Moreover let \( h \in \Rbb \) be such that \( e^h = q \), and denote \( \hbar = \frac{h}{2\pi} \). For \( n \in \mathbb{N} \), define the \( q \)-numbers, \( q \)-factorials, and \( q \)-binomial coefficients by
\[
[n]_q = \frac{q^{n} - q^{-n}}{q - q^{-1}}, 
\qquad 
[n]_q! = \prod_{j=1}^{n} [j]_q, \qquad \begin{bmatrix} n \\ k \end{bmatrix}_q = \frac{ [n]_q ! }{ [k]_q ! [ n- k ]_q !  } \quad  \text{($ 0 \leq k \leq n$)} .
\]
We also write \( q_i = q^{d_i} \) for \( 1 \leq i \leq N \). 

The quantised universal enveloping algebra \( U_q(\gfrak) \) is defined as the universal algebra with generators
\[
K_\lambda, \;\; E_i, \;\; F_i, 
\quad \lambda \in \weights, \, 1 \leq i \leq N,
\]
and relations 
\begin{gather*}
K_0 = 1 , \qquad K_\lambda K_{\mu} = K_{\lambda + \mu} , \\
K_\lambda E_i K_{-\lambda} = q^{( \lambda , \alpha_i )} E_i , 
\qquad 
K_{\lambda} F_i K_{-\lambda} = q^{- ( \lambda, \alpha_i)} F_i , \\
E_i F_j - F_j E_i = \delta_{ij} \frac{K_{i} - K_i ^{-1}}{q_i - q_i ^{-1}}, \\
\sum_{k=0} ^{1 - a_{ij}} (-1)^k \begin{bmatrix} 1-a_{ij} \\ k \end{bmatrix}_{q_i} E_i ^{1 - a_{ij} - k} E_j E_i ^k  = 0 , \quad i \neq j \\
\sum_{k=0} ^{1 - a_{ij}} (-1)^k \begin{bmatrix} 1 - a_{ij} \\ k \end{bmatrix}_{q_i} F_i ^{1 - a_{ij} - k} F_j F_i ^k = 0, \quad i \neq j,
\end{gather*}
where $ \lambda , \mu \in \weights$, $1 \leq i, j \leq N $, and $K_i = K_{\alpha_i}$. It becomes a Hopf algebra with respect to the coproduct given by
\begin{align*}
\hat{\Delta}(K_\mu) &= K_\mu \otimes K_\mu, \\
\hat{\Delta}(E_i) &= E_i \otimes K_i + 1 \otimes E_i, \\
\hat{\Delta}(F_i) &= F_i \otimes 1 + K_i^{-1} \otimes F_i,
\end{align*}
the counit given by
\begin{align*}
\hat{\epsilon}(K_\lambda) = 1, 
\qquad 
\hat{\epsilon}(E_j) = 0, 
\qquad 
\hat{\epsilon}(F_j) = 0,
\end{align*}
and the antipode given by
\begin{align*}
\hat{S}(K_\lambda) = K_{-\lambda}, 
\qquad 
\hat{S}(E_j) = -E_j K_j^{-1}, 
\qquad 
\hat{S}(F_j) = -K_j F_j
\end{align*}
on generators. In particular, we have
\[
\hat{S}^2 (X) = K_{2\rho} X K_{-2\rho} , \qquad X \in U_q(\gf).
\]

We equip \( U_q(\gfrak) \) with the involution
\[
K_\lambda ^* = K_\lambda, 
\qquad 
E_i ^* = K_i F_i , 
\qquad 
F_i^* = E_i K_i ^{-1},
\]
corresponding to the compact real form of $ G $,
and denote the resulting Hopf $ * $-algebra by $ U_q ^\Rbb (\kf)$. In turn, when we view the quantised universal enveloping algebra as an algebra without a $ * $-structure, we continue to denote it by $ U_q(\gf) $.

We denote by \( U_q(\nfrak_+) \) the unital subalgebra of \( U_q(\gfrak) \) generated by \( E_1, \dots, E_N \), and by \( U_q(\nfrak_-) \) the unital subalgebra generated by \( F_1, \dots, F_N \). We also let \( U_q(\hfrak) \) be the unital subalgebra of \( U_q(\gfrak) \) generated by all \( K_\lambda \) for \(\lambda \in \weights \). Then \( U_q(\bfrak_\pm) = U_q(\hfrak) U_q(\nfrak_\pm) = U_q(\nfrak_\pm) U_q(\hfrak) \) are Hopf subalgebras of \( U_q(\gfrak) \). 
Moreover there are vector space isomorphisms
\begin{equation*}
U_q(\nfrak_-) \otimes U_q(\hf) \otimes U_q( \nfrak_+) \xrightarrow[\cong]{} U_q(\gf), \quad U_q( \nfrak_\pm) \otimes U_q(\hf) \xrightarrow[\cong]{} U_q (\bfrak_\pm),
\end{equation*}
induced by multiplication in $ U_q(\gf) $.

A finite-dimensional representation $ \pi: U_q(\gf) \rightarrow \End(V) $ is called \textit{integrable} if there exists a direct sum decomposition
\[
V = \bigoplus_{\nu \in \weights} V_\nu 
\quad \text{where} \quad 
V_\nu = \{ v \in V \mid \pi(K_\lambda)(v) = q^{(\lambda, \nu)} v \text{ for all } \lambda \in \weights \}.
\]
The subspace $ V_\nu \subseteq V $ is called the \textit{weight space} of weight $ \nu \in \weights $, 
and any $ \nu $ with $ V_\nu \neq 0 $ is called a \textit{weight of $ V $}. By a weight vector of weight $ \nu $ we mean any nonzero element $ v \in V_\nu $. The generators \( E_i \) and \( F_i \) of $U_q(\gf)$ act as weight raising and lowering operators in the sense that
\begin{equation*}
\pi(E_i)(v) \in V_{\nu + \alpha_i}, \quad \pi(F_i)(v) \in V_{\nu - \alpha_i},
\end{equation*}
for \( 1 \leq i \leq N \) and \( v \in V_\nu \). 

The equivalence classes of irreducible integrable representations of \( U_q(\gfrak) \) are parametrised by the set \( \weights^+ \) of dominant integral weights. For each \( \mu \in \weights^+ \), we denote the corresponding representation by \( \pi_\mu:  U_q(\gfrak) \rightarrow \End(V(\mu)) \), and the set of its weights by $ \weights(\mu) $. The set $ \weights(\mu) $ is invariant under the action of $W$, and we have
\(w_0(\mu) \leq \nu \leq \mu \) for all \( \nu \in \weights(\mu). \)
For this reason, $ \mu $ and $w_0(\mu) $ are referred to as the \textit{highest} and \textit{lowest} weights of $V(\mu)$, respectively. One has $\dim V(\mu)_{w(\mu)} = 1 $ for all $ w \in W $ and
\[
(\nu, \nu) < (\mu, \mu) , \qquad \forall \nu \in \weights(\mu) \setminus \{w(\mu) \mid w \in W\}.
\]
Every irreducible integrable representation of \( U_q(\gfrak) \) becomes a $ * $-representation of \( U_q^{\Rbb}(\kfrak) \) with respect to an inner product which is uniquely determined up to a positive scalar. This means that integrable representations of \( U_q(\gfrak) \) and integrable $ * $-representations of \( U_q^{\Rbb}(\kfrak) \) are essentially the same thing. 

Let us write $ \Bcal_\gf $ for the braid group associated with $ \gf $, and $ \Tcal_1, \dots, \Tcal_N $ for the generators of $ \Bcal_\gf $ corresponding to the simple roots $ \alpha_1, \dots, \alpha_N $. For $ w \in W $ and a reduced expression $ w = s_{i_1} \cdots s_{i_l} $, we define
\(
\Tcal_w = \Tcal_{i_1} \cdots \Tcal_{i_l} \in \Bcal_\gf,
\)
and this is independent of the reduced expression of  $ w $.

Now let $ V $ be an integrable representation of $ U_q (\gf) $. Then $ \Bcal_\gf $ acts on $ V $ by
\[
\Tcal_i (v) = \sum_{ \substack{r,s,t \geq 0 \\ r - s + t = (\nu , \alpha_i ^\vee ) } } (-1)^s q_i ^{s - rt} F_i ^{(r)} E_i ^{(s)} F_i ^{(t)} \cdot v, 
\qquad v \in V_\nu,
\]
where we write 
\[
E_i^{(k)} = \frac{E_i ^k }{ [k]_{q_i} !}, 
\quad 
F_i ^{(k)} = \frac{F_i ^k}{ [k]_{q_i} !} 
\]
for $ 1 \leq i \leq N $. We have $ \Tcal_w(V_\nu) = V_{w(\nu)} $ for any weight $ \nu $ of $ V $.

The braid group $ \Bcal_\gf $ also acts on $U_q(\gf)$ via the algebra automorphisms determined by
\begin{gather*}
\Tcal_i (K_\mu )  = K_{s_i \mu} , \qquad 
\Tcal_i (E_i) = -K_i F_i , \qquad 
\Tcal_i (F_i) = - E_i K_i ^{-1}, \\
\Tcal_i (E_j) = \sum_{k = 0} ^{ - a_{ij} } (-1)^k q_i ^k E_i ^{(k)} E_j E_i ^{( -a_{ij} - k)}, \quad i \neq j, \\
\Tcal_i (F_j) = \sum_{k = 0} ^{ - a_{ij} } (-1)^k q_i ^{- k} F_i ^{( - a_{ij} - k )} F_j F_i ^{( k )}, \quad i \neq j.
\end{gather*}
This is compatible with the action on integrable representations in the sense that
\[
\Tcal_i(X \cdot v) = \Tcal_i (X) 
\cdot \Tcal_i(v), 
\qquad X \in U_q(\gf), \, v \in V,
\]
for any integrable representation $ V $.

We shall work with orthonormal bases for irreducible integrable representations consisting of weight vectors. Given $ \mu \in \weights^+ $ we write $ e_1^\mu, \dots, e_{n_\mu}^\mu $ for such a basis, and let \( \epsilon_j^\mu \in \weights(\mu) \) be the weight of $ e_j^\mu $ for \( 1 \leq j \leq n_\mu = \dim(V(\mu)) \). We order the basis vectors such that \( \epsilon_1^\mu = \mu \) and \( \epsilon_{n_\mu}^\mu = w_0(\mu) \). That is, $ e^\mu_1 $ is a highest weight vector, and $ e^\mu_{n_\mu} $ is a lowest weight vector of $ V(\mu) $. In fact, we shall arrange without loss of generality that the vectors $ e^\mu_i $ are ordered in a non-ascending order, so that $ \epsilon^\mu_i > \epsilon^\mu_j $ implies $ i < j $. 

For every $ w \in W $ there is a unique index $ 1 \leq i_{w(\mu)} \leq n_\mu $ such that $ e^\mu_{i_{w(\mu)}} $ has weight $ \epsilon^\mu_{i_{w(\mu)}} = w(\mu) $. These weight vectors, associated to weights in the $ W $-orbit of the highest weight, are all uniquely determined up to a scalar, and we use tensor products and the braid group action to fix them as follows. 
Starting from the fundamental weights, we choose the highest weight vectors inductively such that 
\begin{equation}\label{eq:identification of highest weight vectors}
e_1^{\lambda + \mu} = e_1^\lambda \otimes e_1^\mu 
\end{equation}
for all $ \lambda, \mu \in \weights^+ $, using the fact that the vector $ e_1^\lambda \otimes e_1^\mu $ in the tensor product representation $ V(\lambda) \otimes V(\mu) $ generates a subrepresentation isomorphic to $ V(\lambda + \mu) $. 
For $ w \in W $, we then fix $ e_{i_{w(\mu)}}^\mu $ by setting 
\begin{equation}\label{eq:w-weight vectors}
e_{i_{w(\mu)}}^\mu = \frac{\mathcal{T}_w(e_1^\mu)}{\lVert \mathcal{T}_w(e_1^\mu) \rVert}.
\end{equation}
Using
\(
\mathcal{T}_w \bigl(e_1^\lambda \otimes e_1^\mu\bigr)
= \mathcal{T}_w(e_1^\lambda) \otimes \mathcal{T}_w( e_1^\mu),
\)
we observe that
\begin{equation}\label{eq:identification of w-weight vectors}
\begin{aligned}
e_{i_{w(\lambda)}}^\lambda \otimes e_{i_{w(\mu)}}^\mu
= e_{i_{w(\lambda+\mu)}}^{\lambda+\mu}
\end{aligned}
\end{equation}
for all $ \lambda, \mu \in \weights^+ $ and $ w \in W $. 

By definition, the \textit{quantised coordinate ring of $ G $} is the vector space 
\[
\Ocal(G_q) = \bigoplus_{\mu \in \weights^+} \End\bigl(V(\mu)\bigr)^*, 
\]
equipped with the unique Hopf algebra structure $(\Delta, \epsilon, S)$ for which the canonical bilinear pairing
\begin{equation*}
U_q(\gfrak) \times \Ocal(G_q)
\rightarrow \Cbb, \qquad
(X,f) = \sum_{\mu \in \weights^+} f_\mu\bigl(\pi_\mu(X)\bigr)
\end{equation*}
is a nondegenerate skew-pairing of Hopf algebras. In other words, we stipulate 
\begin{gather*}
(XY, f) = (X, f_{(1)})(Y, f_{(2)}), \qquad
(X, fg) = (X_{(1)}, g)(X_{(2)}, f), \\
(1, f) = \epsilon(f), \qquad
(X, 1) = \hat{\epsilon}(X), \\
(\hat{S}(X), f) = (X, S^{-1}(f)), \qquad
(X, S(f)) = (\hat{S}^{-1}(X), f), 
\end{gather*}
for all $ X,Y \in U_q(\gfrak)$ and $ f,g \in \Ocal(G_q) $. Incorporating the $ * $-structure of $ U_q(\gfrak) = U_q^{\Rbb}(\kfrak) $ induces a $ * $-structure on $ \Ocal(G_q) $ given by  
\[
(X^*, f) = \overline{(X, S(f)^*)}, \qquad
(X, f^*) = \overline{(\hat{S}^{-1}(X)^*, f)}.
\]
This turns $ \Ocal(G_q) $ into a Hopf $ * $-algebra which we denote by $ \Cf^\infty(K_q) $ or $ \Ocal^\mathbb{R}(K_q) $, and refer to as the \textit{quantised algebra of functions on $ K $}. By construction, the canonical bilinear pairing between $ U_q^{\Rbb}(\kfrak) $ and $ \Cf^\infty(K_q) $ 
becomes a skew-pairing of Hopf $ * $-algebras.  

The underlying vector space of \( \Ocal(G_q) = \Cf^\infty(K_q) \)
is the linear span of all \emph{matrix coefficients} $ c_{\xi,\eta}^\mu \in \End\bigl(V(\mu)\bigr)^* = \Lbb\bigl(V(\mu)\bigr)^* $ for $\mu \in \weights^+$ and $ \xi, \eta \in V(\mu) $, given by 
\( c^\mu_{\xi, \eta}(x) = \langle \xi, x(\eta) \rangle \).
With our choice of orthonormal basis $ e_1^\mu, \dots, e_{n_\mu}^\mu $ of $ V(\mu) $, we shall abbreviate 
\(
u_{ij}^\mu = c_{e_i^\mu, e_j^\mu}^\mu. 
\)
Then the vectors $ u^\mu _{ij} $ for $ \mu \in \weights^+ $ and $ 1 \leq i,j \leq n_\mu $ form a linear basis of the quantised coordinate ring. Moreover, $ u^\mu = (u_{ij}^\mu)_{1 \leq i,j \leq n_\mu} $ is a unitary element in the $ * $-algebra $ M_{n_\mu}\bigl(\Cf^\infty(K_q)\bigr) $, and we have 
\[
\Delta(u_{ij}^\mu)
= \sum_{1 \leq k \leq n_\mu} u_{ik}^\mu \otimes u_{kj}^\mu,
\qquad
\epsilon(u_{ij}^\mu) = \delta_{ij},
\qquad
S(u_{ij}^\mu) = (u_{ji}^\mu)^*
\]
for all $ 1 \leq i,j \leq n_\mu $. 

The quantised algebra of functions \( \Cfrak^\infty(K_q) \) becomes a \( U_q^\Rbb(\kfrak) \)-bimodule via the left and right actions given by
\begin{equation*}
X \triangleright f = f_{(1)} (X, f_{(2)}), \quad f \triangleleft X = (X, f_{(1)}) f_{(2)}, \quad X \in U_q^\Rbb(\kfrak), \, f \in \Cfrak^\infty(K_q),
\end{equation*}
where we use the Sweedler notation $\Delta(f) = f_{(1)} \otimes f_{(2)}$.
This induces a \( \Pbf \times \Pbf \)-grading on \( \Cfrak^\infty(K_q) \), for which the associated direct summands $ \Cfrak^\infty(K_q)_{\lambda, \nu} $ of \( \Cfrak^\infty(K_q) \) for $ \lambda, \nu \in \weights $ are determined by 
\begin{equation}\label{eq:torus action}
K_\zeta \triangleright f = q^{(\zeta, \nu)} f, \quad f \triangleleft K_\zeta = q^{(\zeta, \lambda)} f, \quad f \in \Cfrak^\infty(K_q)_{\lambda, \nu}
\end{equation}
for $ K_\zeta $ with $ \zeta \in \weights$. The left and right \( U_q^\Rbb(\kfrak) \)-actions introduced above can be extended to turn \( \Cfrak^\infty(K_q) \) into a bimodule over the multiplier Hopf $ * $-algebra dual to $ \Cfrak^\infty(K_q) $. This allows one in particular to consider the left and right actions 
of the linear functionals \( K_\zeta \) for \( \zeta \in \hfrak^* \) on $ \Cfrak^\infty(K_q) $ defined by
\begin{equation*}
(K_\zeta, u^\mu_{ij}) = q^{(\zeta, \epsilon_j^\mu)} \delta_{ij}, \quad \mu \in \weights^+, \; 1 \leq i,j \leq n_\mu. 
\end{equation*}
Since \( K_\zeta = 1 \) if and only if \( \zeta \in i \hbar^{-1} \Qbf^\vee \) and $ K_{\zeta_1} K_{\zeta_2} = K_{\zeta_1 + \zeta_2} $, these elements generate an action of \( \hfrak^* / i \hbar^{-1} \Qbf^\vee \times \hfrak^* / i \hbar^{-1} \Qbf^\vee \cong H \times H \) on \( \Cfrak^\infty(K_q) \), determined by the same formulas as in \eqref{eq:torus action}. 
With respect to the dual multiplier Hopf algebra structure, the elements $ K_\zeta $ satisfy
\[ 
\hat{\Delta}(K_\zeta) = K_\zeta \otimes K_\zeta, \quad 
\hat{S}(K_\zeta) = K_{-\zeta}, \quad 
(K_\zeta)^* = K_{-\overline{\zeta}} 
\] 
for all \( \zeta, \zeta_1, \zeta_2 \in \hfrak^* \), where \( \overline{(\cdot)} : \hfrak^* \to \hfrak^* \) denotes conjugation with respect to the real form \( \tfrak^* \subseteq \hfrak^* \). 

Define the \textit{left adjoint action of \( U_q(\gfrak) \)} on itself by 
\[
X \rightarrow Y = X_{(1)} Y \hat{S}(X_{(2)}) 
\]
for \(X, Y \in U_q(\gfrak) \). Similarly, we have the \textit{right adjoint action of \( U_q(\gfrak) \)} on itself given by 
\[
Y \leftarrow X = \hat{S}(X_{(1)}) Y X_{(2)} 
\]
for \( X,Y \in U_q(\gfrak) \). Note that we have
\[
\hat{S} ( X \rightarrow Y ) = \hat{S} (Y) \leftarrow \hat{S}(X), \qquad X,Y \in U_q (\gfrak).
\]
The \emph{ left and right locally finite parts of \( U_q(\gfrak) \)} are defined by
\begin{align*}
&F_l U_q(\gfrak) = \big\{ X \in U_q(\gfrak) \mid \dim \big( U_q(\gfrak) \rightarrow X \big) < \infty \big\}, \\
&F_r U_q(\gfrak) = \big\{ X \in U_q(\gfrak) \mid \dim \big( X \leftarrow U_q(\gfrak) \big) < \infty \big\}, \nonumber
\end{align*}
respectively. Both are $*$-subalgebras of \( U_q(\gfrak) \), and we have
\begin{equation}\label{eq:locally finite part is big1}
F_l U_q(\gfrak) = \bigoplus_{\mu \in \weights^+} U_q(\gfrak) \rightarrow K_{2\mu}, \quad F_r U_q(\gfrak) = \hat{S} ( F_l U_q (\gf) ) = \bigoplus_{\mu \in \weights^+} K_{- 2\mu} \leftarrow U_q(\gfrak). 
\end{equation}
In particular, since 
\begin{align} \label{eq:ej}
K_{-2\rho} \leftarrow E_j 
= - E_j K_j^{-1} K_{-2\rho} K_j + K_{-2\rho} E_j = K_{-2\rho} \big(1 - q_j^2\big) E_j, 
\end{align}
and likewise,
\begin{align} \label{eq:fj}
K_{-2\rho} \leftarrow F_j = - K_j F_j K_{-2\rho} + K_j K_{-2\rho} F_j = K_{-2\rho} \big(1 - q_j^{-2} \big) K_j F_j, 
\end{align}
the elements $ K_{-2\rho} E_j $ and $ K_{-2\rho} K_j F_j $ are contained in $ F_r U_q(\gf) $. This allows one to show 
\begin{equation} \label{eq:locally finite part is big2}
U_q (\gf) = U_q (\hf) F_r U_q (\gf) = F_r U_q (\gf) U_q (\hf) = U_q (\hf) F_l U_q (\gf) = F_l U_q (\gf) U_q (\hf),
\end{equation}
and also 
\begin{equation} \label{eq:locally finite part is big in representation}
\pi_\lambda(F_r U_q(\gf)) = \End(V(\lambda))
\end{equation}
for all $\lambda \in \weights^+$. 

Let \( \Rcal \in \prod_{\lambda, \mu} \End(V(\lambda)) \otimes \End(V(\mu)) \) denote the universal \( R \)-matrix of \( U_q(\gfrak) \) as in \cite[Theorem~3.108]{VoigtYuncken}. Consider the \( l \)-functionals \( l^\pm : \Cfrak^\infty(K_q) \to U_q(\gfrak) \) defined by
\begin{equation*}
(l^+(f), g) = (\Rcal, g \otimes f), \quad (l^-(f), g) = (\Rcal^{-1}, f \otimes g), \quad f, g \in \Cfrak^\infty(K_q),
\end{equation*}
which are Hopf algebra homomorphisms satisfying 
\(
l^\pm(f)^* = l^\mp(f^*) \) for all \( f \in \Cfrak^\infty(K_q) \).
We also note that \( l^\pm(\Cfrak^\infty(K_q)) \subseteq U_q(\bfrak_\pm) \).

Let us define linear maps \( I, J : \Cfrak^\infty(K_q) \to U_q(\gfrak) \) by
\begin{equation*}
I(f) = l^-(f_{(1)}) \hat{S} l^+(f_{(2)}), \quad J(f) = \hat{S} l^+ (f_{(1)} ) l^- ( f_{(2)} ),
\end{equation*}
and note that $J(f) = \hat{S} I S^{-1} (f) = \hat{S}^{-1} I S (f)$. Then \( I , J \) are linear isomorphisms from \( \Cfrak^\infty(K_q) \) onto $ F_l U_q(\gf) $ and $ F_r U_q(\gf) $, respectively.
Also, since
\[
J(u^\mu _{ij}) = \sum_{1 \leq k \leq n_\mu} \hat{S} l^+ (u^\mu _{ik} ) l^- ( u^\mu _{kj}) = \sum_{1 \leq k \leq n_\mu} l^+ ( {u^\mu _{ki}} ^* )  l^- ( u^\mu_{kj} ) = \sum_{1 \leq k \leq n_\mu} l^- ( u^\mu _{ki} )^* l^- ( u^\mu _{kj} )
\]
for $\mu \in \weights^+$ and $ 1 \leq i, j \leq n_\mu$, we have
\begin{equation*}
J(u^\mu) = \big( J(u^\mu _{ij}) \big)_{1 \leq i,j \leq n_\mu} = l^- (u^\mu) ^* l^- (u^\mu) \in M_{n_\mu} (F_r U_q (\gf) ).
\end{equation*}
The variant $ I': \Cf^\infty(K_q) \rightarrow U_q(\gf) $ of the map $ I $ introduced above, defined by  
\begin{equation} \label{eq:I'-isomorphism}
I'(f) = l^+ ( f_{(1)} ) \hat{S} l^- ( f_{(2)} ) 
\end{equation}
is also a linear isomorphism from $ \Cf^\infty(K_q) $ onto $ F_l U_q (\gf) $, see \cite[Section~6.5]{Lee2025}. 

\subsection{The primitive spectrum of \texorpdfstring{$\Ocal(G_q)$}{O(Gq)}}

In this section we summarise some results on the structure of $ \Ocal(G_q) $ from \cite{Joseph}, \cite{Yakimov2014} that will be used in this paper. Recall that we write $ \Cf^\infty(K_q) $ in place of $ \Ocal(G_q) $ if we are concerned with statements or results involving the involution on the quantised coordinate ring.

Let $ \mu \in \weights^+ $ and $ \xi \in V(\mu)_\lambda, \eta \in V(\mu)_\nu $ for $ \lambda, \nu \in \weights(\mu) $. Then we shall sometimes abbreviate
\[
c^\mu_{\xi,\eta}
= c^\mu_{\xi,\nu}
= c^\mu_{\lambda,\eta}
= c^\mu_{\lambda,\nu}
= c_{\lambda ,\nu},
\]
depending on which one of these expressions for the matrix coefficient $ c^\mu_{\xi,\eta} \in \Ocal(G_q) $ is most convenient. By abuse of notation, we will even sometimes write $ c^\mu_{\lambda, \nu} = c_{\lambda, \nu} $ for matrix coefficients of unspecified weight vectors in $ V(\mu)_\lambda, V(\mu)_\nu $, when these vectors are not important. 

Let $ \mu, \nu \in \weights^+ $ and $ 1 \leq i,j \leq n_\mu, 1 \leq k,l \leq n_\nu $. With the notation above, there are scalars $ \alpha^{ikl}_{rs}, \beta^{ijk}_{uv} \in \mathbb{C} $ such that
\begin{align} \label{eq:multiplication rule for coefficients}
u^\nu_{kl} u^\mu_{i1} &= q^{(\mu, \epsilon_l) -(\epsilon_i, \epsilon_k)} u^\mu_{i1} u^\nu_{kl} + \sum_{r = i + 1}^{n_\mu} \sum_{s = 1}^{k - 1} \alpha^{ikl}_{rs} u^\nu_{sl} u^\mu_{r1}, \\
u^\mu_{ij} u^\nu_{k n_\nu} 
&= q^{-(\epsilon_j, w_0(\nu)) +(\epsilon_i, \epsilon_k)} u^\nu_{kn_\nu} u^\mu_{ij} + \sum_{u = i + 1}^{n_\mu} \sum_{v = 1}^{k - 1} \beta^{ijk}_{uv} u^\nu_{v n_\nu} u^\mu_{uj} \nonumber
\end{align}
in $ \Ocal(G_q) $, compare \cite[Proposition 3.135]{VoigtYuncken} \footnote{We point out that there is a sign error in \cite[Proposition 3.135]{VoigtYuncken}, which is related to a typo in formula 3.104 c). The analogous formula 3.105 d) is correct.}. Here we recall that the vectors $ e^\mu_i $ are ordered in a non-ascending order, so that $ \epsilon^\mu_i > \epsilon^\mu_j $ implies $ i < j $. 
We can also rewrite these relations informally as 
\begin{align}
\label{eq:multiplication rule for coefficients2}
c_{\lambda,\nu} \, c^\mu_{\mu',\mu}
&= q^{(\mu,\nu)-(\mu',\lambda)} \, c^\mu_{\mu',\mu} \, c_{\lambda,\nu} + \sum_{0 \neq \gamma \in \Qbf^+}  c_{\lambda+\gamma,\nu} \, c^\mu_{\mu'-\gamma,\mu}, \\
c_{\lambda,\nu} \, c^{-w_0(\mu)}_{-\mu',-\mu}
&= q^{(\mu,\nu)-(\mu',\lambda)} \, c^{-w_0(\mu)}_{-\mu',-\mu} \, c_{\lambda,\nu} + \sum_{0 \neq \gamma \in \Qbf^+} c^{-w_0(\mu)}_{-\mu'+\gamma,-\mu} \, c_{\lambda-\gamma,\nu}, 
\nonumber
\end{align}
only indicating the contributions in \eqref{eq:multiplication rule for coefficients} from different weight spaces.  

Define
\begin{align*}
\Ocal^+
&= \Span\{u^\mu_{i,1} \mid \mu \in \weights^+,\, 1 \leq i \leq n_\mu \} = \Span\{ c^\mu_{\xi,\mu} \mid \mu \in \weights^+,\, \xi \in V(\mu) \}, \\ 
\Ocal^- &= \Span\{u^\mu_{j,n_\mu} \mid \mu \in \weights^+,\, 1 \leq j \leq n_\mu \} = \Span\{ c^{-w_0(\mu)}_{\xi,-\mu} \mid \mu \in \weights^+,\, \xi \in V(-w_0(\mu)) \}. \nonumber
\end{align*}
These are subalgebras of $ \Ocal(G_q) $ satisfying
\(
\Ocal^+ \Ocal^- = \Ocal(G_q) = \Ocal^- \Ocal^+.
\)
We also have 
\begin{equation}\label{eq:the +- parts of O(Gq) conjugate}
(\Ocal^\pm)^* = \Ocal^\mp,
\end{equation}
since for every $ \xi \in V(\mu)_\nu $ there exists a unique $ \eta \in V(-w_0(\mu))_{-\nu} $ such that 
\begin{equation} \label{eq:coefficients under conjugation}
(c^\mu_{\xi, \mu})^* = c^{-w_0(\mu)}_{\eta,-\mu}.
\end{equation}
In fact, we have the following more precise result. Let $\vee : \Qbf \rightarrow \Qbf^\vee$ be the group isomorphism defined by $\alpha_i \mapsto \alpha_i^\vee$.

\begin{thm}\label{thm:coefficients under conjugation}
For $ \mu \in \weights^+ $ and $ \nu \in \weights(\mu) $, let $ n_{\mu,\nu} = \dim(V(\mu)_\nu) $. Then one can choose orthonormal bases
\[
\bigl\{
e_i^{\mu,\nu}
\mid
1 \leq i \leq n_{\mu,\nu}
\bigr\}
\subseteq V(\mu)_\nu
\]
of the weight spaces of $ V(\mu) $ such that
\[
\bigl(c^\mu_{e_i^{\mu,\lambda}, e_j^{\mu,\nu}} \bigr)^*
= (-1)^{(\rho,(\lambda-\nu)^\vee)} q^{(\rho,\lambda-\nu)}
c^{-w_0(\mu)}_{e_i^{-w_0(\mu),-\lambda}, e_j^{-w_0(\mu),-\nu}}
\]
for all $ \mu\in\weights^+, \lambda,\nu\in\weights(\mu), 1 \leq i \leq n_{\mu,\lambda} $, and $ 1 \leq j \leq n_{\mu,\nu} $.
\end{thm}

When specialised to $K_q=SU_q(2)$, Theorem~\ref{thm:coefficients under conjugation} recovers 
the well-known formulas
\[
(u^\varpi _{11})^* = u^\varpi _{22} , \qquad (u^\varpi _{21})^* = - q^{-1} u^\varpi _{12}.
\]
We also note that Theorem~\ref{thm:coefficients under conjugation}
provides a corrected version of \cite[Proposition~2.3.2]{KorogodskiSoibelman}. 
As the weaker form \eqref{eq:coefficients under conjugation} is all that is needed in this paper, we omit the proof of Theorem~\ref{thm:coefficients under conjugation} here; a complete proof is given in the accompanying note [Supplement to ``On the complexification of $ q $-deformations of compact semisimple Lie groups"] that can be found in \cite{Supplement}.

For any $\mu \in \weights^+$ and $ w \in W $, define \textit{the Demazure module} $ V_w^\pm(\mu) $ by
\begin{align*}
V_w^+(\mu) &= U_q(\bfrak_+) \cdot V(\mu)_{w(\mu)} \subseteq V(\mu), \\
V_w^-(\mu) &= U_q(\bfrak_-) \cdot V(-w_0(\mu))_{-w(\mu)} \subseteq V(-w_0(\mu)). \nonumber
\end{align*}
By \cite[Proposition~4.4.5]{Joseph} we have
\begin{equation} \label{eq:Bruhat order and Demazure module}
w \leq w'   \Longleftrightarrow V^+ _{w} (\rho) \subseteq V^+ _{w'} ( \rho )
\end{equation}
for any $ w, w' \in W $, where $ \leq $ denotes the Bruhat order. Moreover, the subspaces
\begin{align*}
I_w^+ &= \Span \{ c^\mu_{\xi , \mu} \mid \mu \in \weights^+,\, \xi \in V_w^+(\mu)^\perp \} \subseteq \Ocal^+, \\
I_w^- &= \Span \{ c^{-w_0(\mu)}_{\xi, -\mu} \mid \mu \in \weights^+,\, \xi \in V_w^-(\mu)^\perp \} \subseteq \Ocal^- \nonumber
\end{align*}
are completely prime ideals of $ \Ocal^+ $ and $ \Ocal^- $, respectively, see 
\cite[Proposition 10.1.8]{Joseph}, \cite[Theorem 2.1]{Yakimov2014}. We also note that \eqref{eq:coefficients under conjugation} implies
\begin{equation}\label{eq:Demazure ideals conjugate}
(I_w^\pm)^* = I_w^\mp.
\end{equation}
For each pair $\wrm = (w_+, w_-) \in W \times W $, one has $ I_{w_+} ^+ \Ocal^- = \Ocal^-I_{w_+} ^+ $ and $ \Ocal^+ I_{w_-}^- = I_{w_-}^- \Ocal^+ $ due to
\eqref{eq:multiplication rule for coefficients}, \eqref{eq:multiplication rule for coefficients2}. 
Then
\begin{equation*}
I_\wrm = I_{w_+} ^+ \Ocal^- + \Ocal^+ I_{w_-}^- = \Ocal^- I_{w_+}^+ + I_{w_-}^- \Ocal^+ \subseteq \Ocal(G_q)
\end{equation*}
is again a completely prime ideal \cite[Proposition 10.3.5]{Joseph}, \cite[Theorem 2.1]{Yakimov2014}. From \eqref{eq:the +- parts of O(Gq) conjugate} and \eqref{eq:Demazure ideals conjugate} it follows that with respect to the $ * $-structure of $ \Ocal^\mathbb{R}(K_q) $ we have 
\begin{equation}\label{eq:Joseph ideal conjugate}
(I_{(w_+, w_-)})^* =  (I_{w_+}^+)^* (\Ocal^-)^* + (\Ocal^+)^* (I_{w_-}^-)^* = I_{w_+}^- \Ocal^+ + \Ocal^- I_{w_-}^+ = I_{(w_-, w_+)}.
\end{equation}
Note also that the ideals $ I_w^\pm $ and $ I_\wrm $ are all invariant under the actions of $ H $ described in \eqref{eq:torus action}.

For each $ \mu \in \weights^+ $ and $w \in W$, define
\begin{equation} \label{eq:w-multiplicative elements}
c_{w, \mu}^+ = u^\mu _{ i_{w\mu} 1 } \in \Ocal^+, \qquad c_{w, \mu}^- = ( c^+ _{w, \mu} )^* \in \Ocal^-.
\end{equation}
Note that, in our convention, these elements may be written as $c_{w, \mu} ^+ = c^\mu _{w(\mu), \mu}$ and  $c_{w, \mu} ^- = c^{-w_0 (\mu)} _{-w(\mu) , -\mu}$, respectively.
By \eqref{eq:identification of highest weight vectors}--\eqref{eq:identification of w-weight vectors}, we have
\begin{equation} \label{eq:relation for w-multiplicative elements}
c_{w, \mu_1}^\pm c_{w, \mu_2}^\pm  = c_{w, \mu_1 + \mu_2}^\pm
\end{equation}
for any $ \mu_1, \mu_2 \in \weights^+ $. 
Therefore, the subsets $E_w^\pm = \{ c_{w, \mu}^\pm \mid \mu \in \weights^+ \} \subseteq \Ocal^\pm$ are multiplicative. Moreover, by \eqref{eq:multiplication rule for coefficients}, \eqref{eq:multiplication rule for coefficients2}, these elements satisfy the commutation relations
\begin{equation}\label{eq:commutation relations for w-multiplicative elements}
c_{\lambda, \nu} c_{w_\pm , \mu}^\pm  = q^{(\mu, \nu) - (w_\pm(\mu), \lambda)} c_{w_\pm , \mu}^\pm  c_{\lambda, \nu} \textup{ mod } I_{\wrm}
\end{equation}
for any $ \wrm = (w_+, w_-) \in W \times W$, $\mu \in \weights^+$, and $\lambda, \nu \in \weights $. 
This implies that 
\begin{equation*}
E_{\wrm} = \{ \zeta\, c_{w_+, \lambda}^+ c_{w_-, \mu}^- \mid \zeta \in \Cbb^\times,\, \lambda, \mu \in \weights^+ \},
\end{equation*}
considered as a subset of $\Ocal(G_q) / I_\wrm$, is a multiplicative set and satisfies both the left and right Ore conditions, see \cite[Section~A.2]{Joseph}. Since $ (c^{\pm} _{w, \mu})^* = c^{\mp} _{w, \mu} $ we have
\begin{equation}\label{eq:w-multiplicative subset conjugate}
(E_{(w_+, w_-)})^* = E_{(w_-, w_+)}
\end{equation}
for any $ w_+, w_- \in W $. Let us record the following result from \cite{Joseph}, 
compare \cite[Theorem 2.3]{Yakimov2014}. 

\begin{thm} \label{thm:strata of prime spectrum}
Let $J \subseteq \Ocal(G_q)$ be a completely prime ideal. Then there exists a unique $\wrm = (w_+, w_-) \in W \times W$ such that $I_\wrm \subseteq J$ and $(J/I_\wrm) \cap E_\wrm = \emptyset$.
\end{thm}

An ideal in $ \Ocal(G_q) $ is prime if and only if it is completely prime \cite[Theorem 11.4]{Joseph1994}. If we let $ \Spec_\wrm( \Ocal(G_q)) \subseteq \Spec(\Ocal(G_q)) $ denote the set of all completely prime ideals associated to $ \wrm $ via Theorem~\ref{thm:strata of prime spectrum} then we obtain the decomposition
\begin{equation} \label{eq:stratification of prime spectrum}
\Spec(\Ocal(G_q)) = \coprod_{\wrm \in W \times W} \Spec_\wrm( \Ocal(G_q))
\end{equation} 
of the prime spectrum $ \Spec(\Ocal(G_q)) $ of $ \Ocal(G_q) $. 
The set $ \Prim( \Ocal(G_q)) $ of primitive ideals of $ \Ocal(G_q) $ is contained in $ \Spec( \Ocal(G_q)) $, and one obtains a corresponding decomposition 
\begin{equation} \label{eq:stratification of primitive spectrum}
\Prim( \Ocal(G_q) ) = \coprod_{\wrm \in W \times W} \Prim_\wrm( \Ocal(G_q))
\end{equation}
where
\[
\Prim_\wrm(\Ocal(G_q)) = \Prim(\Ocal(G_q)) \cap \Spec_\wrm( \Ocal(G_q))
\]
for $ \wrm \in W \times W $.

Since $ E_\wrm $ is multiplicative and right Ore, we can consider the localisation
\begin{equation}\label{eq:w-localisation}
\Ocal_\wrm = \big( \Ocal(G_q) / I_\wrm \big) [E_\wrm^{-1}], 
\end{equation}
compare \cite[A.2.4]{Joseph}. For the following statement see \cite[Theorem 2.3]{Yakimov2014}.

\begin{thm}\label{thm:the Joseph stratum in terms of Zw}
Let $ \wrm \in W \times W $ and denote by $ Z_\wrm $ the centre of $ \Ocal_\wrm $.  Given $ J \in \Spec(Z_\wrm) $, define
\[
\iota_\wrm(J) = \{ f \in \Ocal(G_q) \mid f + I_\wrm \in \Ocal_\wrm J \}.
\]
Then, $ \iota_\wrm(J) \in \Spec_\wrm( \Ocal(G_q)) $, and the map
\[
\iota_\wrm: \Spec(Z_\wrm) \rightarrow \Spec_\wrm(\Ocal(G_q))
\]
is a bijection, which restricts to a bijection
\[
\iota_\wrm : \Max(Z_\wrm) \rightarrow \Prim_\wrm(\Ocal(G_q)),
\]
where $ \Max(Z_\wrm) $ denotes the set of all maximal ideals of $ Z_\wrm $.
\end{thm}

Therefore, to understand $ \Prim_\wrm(\Ocal(G_q)) $, one needs to study the algebra $ Z_\wrm $. Let $ \mu = \mu_+ - \mu_- \in \weights $ with $ \mu_+, \mu_- \in \weights^+ $ and $ \wrm = (w_+, w_-) \in W \times W $. Extending \eqref{eq:w-multiplicative elements} by
\begin{equation*}
c_{w_\pm, \mu}^\pm = c_{w_\pm, \mu_+}^\pm \big( c_{w_\pm, \mu_-}^\pm \big)^{-1} \in \Ocal_\wrm, 
\end{equation*}
which is well-defined independently of the choice $(\mu_+ , \mu_-) $, one has
\begin{equation*} \label{eq:relation for w-multiplicative elements in localisation}
c_{w_\pm, \mu_1}^\pm \, c_{w_\pm, \mu_2}^\pm = c_{w_\pm, \mu_1 + \mu_2}^\pm
\end{equation*}
for any $ \wrm \in W \times W $ and $ \mu_1, \mu_2 \in \weights $, compare \eqref{eq:relation for w-multiplicative elements}.

For a subset \( \Ical \subseteq \{1, \dots, N\} \), we denote
$
\Pbf_\Ical = \bigoplus_{i \in \Ical} \Zbb \varpi_i.
$
Moreover, for \( w \in W \) define
\[
\Ical(w) = \{1 \leq i \leq N \mid w(\varpi_i) = \varpi_i \}, \quad \Scal(w) = \{1, \dots, N \} \setminus \Ical(w), 
\]
and for \( \wrm = (w_+, w_-) \in W \times W \), let
\[
\Ical(\wrm) = \Ical(w_+) \cap \Ical (w_-), \quad \Scal(\wrm) = \Scal(w_+) \cup \Scal(w_-) = \{1, \dots, N \} \setminus \Ical(\wrm).
\]
Then
\begin{equation} \label{defltilde}
\tilde{\Lcal}(\wrm) = \Ker(w_+ - w_-) \cap \Pbf
\end{equation}
is a lattice of rank \( \dim(\Ker(w_+ - w_-)) \) because the Weyl group preserves \( \Pbf \). Also, since \( \Pbf_{\Ical(\wrm)} \subseteq \Ker(w_+ - w_-) \), the set 
\[
\tilde{\Lcal}_{\text{red}}(\wrm) = \Ker(w_+ - w_-) \cap \Pbf_{\Scal(\wrm)}
\]
is a lattice of rank \( k = \dim(\Ker(w_+ - w_-)) - |\Ical(\wrm)| \), and
\[
\tilde{\Lcal}(\wrm) = \tilde{\Lcal}_{\text{red}}(\wrm) \oplus \Pbf_{\Ical(\wrm)}.
\]
Let \( \{\lambda^{(1)}, \dots, \lambda^{(k)} \} \subseteq \weights \) be a basis of \( \tilde{\Lcal}_{\text{red}}(\wrm) \) and set
\begin{equation} \label{eq:basis from the reduced part}
a_j = c_{w_+, \lambda^{(j)}}^+ (c_{w_-, \lambda^{(j)}}^-)^{-1} \in \Ocal_\wrm.
\end{equation}
The following result is due to Yakimov \cite[Theorem 3.1]{Yakimov2014}. 

\begin{thm}\label{thm:centre of the localisation}
Fix $ \wrm = (w_+, w_-) \in W \times W $. The centre $ Z_\wrm $ of $ \Ocal_\wrm $ is isomorphic to the algebra of complex Laurent polynomials of rank $ \dim(\Ker(w_+ - w_-)) $ with generators
\begin{equation*}
\{c_{w_+, \varpi_i}^+ \mid i \in \Ical(\wrm) \} \coprod \{ a_j \mid 1 \leq j \leq k \}, 
\end{equation*}
with the elements $ c_{w_+, \varpi_i}^+ $ and $ a_j $ as defined above.
\end{thm}

By the weak Nullstellensatz,  
it follows that
\begin{multline*}
\Max(Z_\wrm) = 
\Big\{ \sum_{1 \leq j \leq k} Z_\wrm(a_j - \zeta_j) + \sum_{i \in \Ical(\wrm)} Z_\wrm(c_{w_+, \varpi_i}^+ - z_i) \;\Big|\, \zeta_j, z_i \in \Cbb^\times
\text{ for all } i,j \Big\}.
\end{multline*}
Combining this with Theorem~\ref{thm:the Joseph stratum in terms of Zw}, we arrive at the following statement. 

\begin{thm}\label{thm:the primitive Joseph stratum}
Fix $ \wrm = (w_+, w_-) \in W \times W$ and let $ k = \dim( \Ker(w_+ - w_-)) - |\Ical(\wrm)|$. For $\zeta = (\zeta_1, \dots, \zeta_k) \in (\Cbb^\times)^k$ and $ z = (z_i)_{i \in \Ical(\wrm)} \in (\Cbb^\times)^{\Ical(\wrm)} $, define
\begin{equation} \label{eq:primitive ideal of O(Gq)}
J_{\wrm, \zeta, z} = \Big\{ f \in \Ocal(G_q) \;\Big|\, f + I_\wrm \in \sum_{1 \leq j \leq k} \Ocal_\wrm(a_j - \zeta_j) + \sum_{i \in \Ical(\wrm)} \Ocal_\wrm(c_{w_+, \varpi_i}^+ - z_i) \Big\}.
\end{equation}
Then we have
\begin{equation*}
\Prim_\wrm(\Ocal(G_q)) = \big\{J_{\wrm, \zeta, z} \;\big|\, \zeta = (\zeta_1, \dots, \zeta_k) \in (\Cbb^\times)^k,\, z = (z_i)_{i \in \Ical(\wrm)} \in (\Cbb^\times)^{\Ical(\wrm)} \big\},
\end{equation*}
and $ J_{\wrm, \zeta, z} = J_{\wrm, \zeta', z'} $ if and only if $ \zeta = \zeta' $ and $ z = z' $.
\end{thm}

Due to \eqref{eq:stratification of primitive spectrum}, one thus obtains an explicit parametrisation 
\[
\Prim(\Ocal(G_q)) = \coprod_{\wrm \in W \times W} \Prim_\wrm(\Ocal(G_q)) \cong \coprod_{\wrm \in W \times W} (\Cbb^\times)^{\dim(\Ker(w_+ - w_-))}
\]
of the primitive ideals of $ \Ocal(G_q) $. We note that the right-hand side of this formula also provides a natural parametrisation of the space of symplectic leaves of the Poisson Lie group $ G $, see \cite[Theorem 4.6]{Yakimov2014}.

Fix $ \wrm = (w_+, w_-) \in W \times W $ and write $ \wrm^* = (w_- , w_+) \in W \times W $. Due to \eqref{eq:Joseph ideal conjugate} and \eqref{eq:w-multiplicative subset conjugate}, the involution of $ \Cf^\infty(K_q) $ descends to a conjugate-linear algebra antihomomorphism
\begin{equation*}
*: \Ocal_{\wrm} \longrightarrow \Ocal_{\wrm^*}.
\end{equation*}
For $ i \in \Ical(\wrm) $, we have $ u^{\varpi_i} _{j 1} \in I_{\wrm} $ for $ 1 < j \leq n_{\varpi_i} $ because $ \epsilon^{\varpi_i} _{j} < \varpi_i = w_+(\varpi_i) $. Therefore, 
\[
c^- _{w_-, \varpi_i} c^+ _{w_+, \varpi_i} = (u^{\varpi_i} _{11})^* u^{\varpi_i} _{11} = \sum_{1 \leq j \leq n_{\varpi_i}} ( u^{\varpi_i} _{j 1} )^* u^{\varpi_i} _{j1} = 1 \quad \mod I_{\wrm}, 
\]
which means
\[
(c^{+} _{w_+, \varpi_i})^{-1} = c^- _{w_-, \varpi_i} \quad \text{in } \Ocal_{\wrm}.
\]
Upon replacing $ \wrm $ by $\wrm^* $, it follows that, for $z_i \in \Cbb^\times$ and $ i \in \Ical(\wrm)$,
\begin{equation}\label{eq:I(wrm) part under involution}
( c^+ _{w_+ , \varpi_i} - z_i )^* = c^- _{w_+ , \varpi_i} - \overline{z_i} = (c^+ _{w_- , \varpi_i})^{-1} - \overline{z_i}, \quad \text{in } \Ocal_{\wrm^*} .
\end{equation}
Moreover, for $ \mu = \mu_+ - \mu_- \in \weights $ with $ \mu_\pm \in \weights^+ $ we have 
\[
(c^+ _{ w_+ , \mu })^* = \big( c^+ _{w_+ , \mu_+} (c^+ _{w_+ , \mu_-})^{-1} \big)^* = c^- _{w_+ , \mu_+} ( c^{-} _{w_+ , \mu_-} )^{-1} = c^{-} _{w_+ , \mu}.
\]
Thus, for $ \lambda \in \Ker (w_+-w_-) \cap \weights $ we obtain
\begin{align*}
\big( c^+ _{w_+ , \lambda} (c^- _{w_- , \lambda})^{-1 } \big)^* = (c^+ _{w_- , \lambda} )^{-1} c^- _{ w_+ , \lambda }
= q^{ - (\lambda , \lambda )+ (w_+ \lambda , w_- \lambda ) } c^- _{w_+ , \lambda} ( c^+ _{w_- , \lambda} )^{-1} \\
= c^- _{w_+ , \lambda} ( c^+ _{w_- , \lambda} )^{-1} = \big( c^+ _{w_- , \lambda} (c^- _{w_+ , \lambda})^{-1} \big)^{-1}, 
\end{align*}
which implies 
\begin{equation}\label{eq:S(wrm) part under involution}
\big( c^+ _{w_+ , \lambda} (c^- _{w_- , \lambda})^{-1 } - \zeta \big)^* = \big( c^+ _{w_- , \lambda} (c^- _{w_+ , \lambda})^{-1} \big)^{-1} - \overline{\zeta} , \qquad \zeta \in \Cbb^\times , \, \lambda \in \Ker (w_+ - w_-) \cap \weights .
\end{equation}
Combining \eqref{eq:primitive ideal of O(Gq)} with \eqref{eq:I(wrm) part under involution}, \eqref{eq:S(wrm) part under involution}, we obtain the following statement.

\begin{prop}\label{prop:primitive ideals under involution}
Let $ \wrm = (w_+, w_-) \in W \times W$, $k = \dim(\Ker (w_+ - w_-)) - | \Ical(\wrm)| $, $\zeta \in (\Cbb^\times)^k$, and $z \in (\Cbb^\times)^{\Ical(\wrm)} $ be as in Theorem~\ref{thm:the primitive Joseph stratum}. Then we have
\[
*( J_{\wrm, \zeta , z} ) = J_{ \wrm^* , \overline{\zeta} ^{-1}, \overline{z}^{-1} },
\]
where $\wrm^* = (w_-, w_+)$, $\overline{\zeta}^{-1} = ( \overline{\zeta_1}^{-1} , \dots , \overline{\zeta_k}^{-1} )$, and $\overline{z}^{-1} = (\overline{z_i}^{-1})_{i \in \Ical(\wrm)}$.

In particular, the primitive ideal $ J_{\wrm, \zeta , z} \subseteq \Ocal(G_q) $ is self-adjoint if and only if $ w _+ = w_-$, $\zeta \in \Tbb^k $, and $z \in \Tbb^{\Ical(\wrm)} $, where $ \Tbb \subseteq \Cbb^\times $ is the unit circle.
\end{prop}

We now give another description of the ideal $ J_{\wrm, \zeta, z} $ which avoids reference to the localisation \eqref{eq:w-localisation}. As before, let $\wrm = (w_+, w_-) \in W \times W$, $k = \dim(\Ker(w_+ - w_-)) - |\Ical(\wrm)|$, $ \zeta = (\zeta_1, \dots, \zeta_k) \in (\Cbb^\times)^k$, and $z = (z_i)_{i \in \Ical(\wrm)} \in (\Cbb^\times)^{\Ical(\wrm)} $.
For each $ 1 \leq j \leq k $, write the vectors $ \lambda^{(j)} $ of the basis for $ \tilde{\Lcal}_{\text{red}} $ used in \eqref{eq:basis from the reduced part} in the form $ \lambda^{(j)} = \lambda^{(j)}_+ - \lambda^{(j)}_- $ where $ \lambda^{(j)}_\pm \in \Pbf^+ $ have disjoint support with respect to the standard basis $ \{\varpi_1, \dots, \varpi_N \} $ of $ \weights $. If we define 
\[
b_j(\zeta_j) = c^+ _{w_+ , \lambda^{(j)} _{+}} c^- _{w_- , \lambda^{(j)}_-} - \zeta_j c^+ _{w_+, \lambda^{(j)}_-} c^- _{w_- , \lambda^{(j)} _+} \in \Ocal(G_q), 
\]
then we have 
\[
b_j(\zeta_j) = c_{w_+, \lambda^{(j)}_-}^+ (a_j - \zeta_j) c_{w_-, \lambda^{(j)}_+}^-\in \Ocal(G_q)/I_\wrm.
\]
Using these elements, we obtain
\begin{align} \label{eq:primitive ideal of O(Gq)-alternative}
J_{\wrm, \zeta, z} = \Big\{ &f \in \Ocal(G_q) \;\Big|\, \exists c \in E_\wrm \text{ such that} \\
&\; cf \in \sum_{1 \leq j \leq k} \Ocal(G_q) b_j(\zeta_j) + \sum_{i \in \Ical(\wrm)} \Ocal(G_q)(c_{w_+, \varpi_i}^+ - z_i) + I_\wrm \Big\}. \nonumber
\end{align}
The following fact is then an easy consequence of Theorem~\ref{thm:the primitive Joseph stratum}.

\begin{prop}\label{prop:primitive ideals are maximal in each stratum}
Let $ \wrm = (w_+, w_-) \in W \times W $. If $ J, J' \in \Prim_\wrm(\Ocal(G_q)) $ and $ J \subseteq J' $, then $ J = J' $.
\end{prop}

\begin{proof}
Let $ k = \dim(\Ker(w_+ - w_-)) - |\Ical(\wrm)| $.
By Theorem~\ref{thm:the primitive Joseph stratum}, there exist $ \zeta, \zeta' \in (\Cbb^\times)^k $ and $ z, z' \in (\Cbb^\times)^{\Ical(\wrm)} $ such that
\[
J = J_{\wrm,\zeta,z} \quad \text{and} \quad J' = J_{\wrm,\zeta',z'}.
\]
Let $ 1 \leq j \leq k $. Since $ J \subseteq J' $, we have
\[
-(\zeta_j - \zeta_j') \, c^{+}_{w_+,\lambda_-^{(j)}} \, c^{-}_{w_-,\lambda_+^{(j)}} = b_j(\zeta_j) - b_j(\zeta_j')
\in J'.
\]
If $\zeta_j - \zeta_j' \neq 0$, then this implies $c^{+}_{w_+,\lambda_-^{(j)}} \, c^{-}_{w_-,\lambda_+^{(j)}} \in J'/I_\wrm$. Consequently,
\[
c^{+}_{w_+,\lambda_-^{(j)}} \, c^{-}_{w_-,\lambda_+^{(j)}} \in (J'/I_\wrm) \cap E_\wrm,
\]
contradicting the fact that $(J'/I_\wrm) \cap E_\wrm = \emptyset$. Therefore, $\zeta_j = \zeta_j'$ for every $1 \leq j \leq k$, and hence $\zeta = \zeta'$. A similar argument shows that $z = z'$, from which we conclude that $J = J'$.
\end{proof}

Finally, note that if $ J $ is a primitive ideal, then so is $ K_{\omega} \triangleright J \triangleleft K_{\chi} $ for any $ \omega, \chi \in \hf^*/i \hbar^{-1} \roots^\vee \cong H $. Explicitly, combining \eqref{eq:torus action} with \eqref{eq:primitive ideal of O(Gq)-alternative}, we obtain
\begin{equation*}
K_{\omega} \triangleright J_{\wrm, \zeta, z} \triangleleft K_{\chi} = J_{\wrm, (\omega , \chi) \cdot \zeta , (\omega, \chi) \cdot z}, \quad \wrm = (w_+, w_-) \in W \times W,
\end{equation*}
where $ k = \dim \Ker(w_+ - w_-) - |\Ical(\wrm)| $, and for any $ \zeta = (\zeta_j)_{1 \leq j \leq k} \in (\Cbb^\times)^k $ and $ z = (z_i)_{i \in \Ical(\wrm)} \in (\Cbb^\times)^{\Ical(\wrm)} $, the expressions $ (\omega, \chi) \cdot \zeta $ and $ (\omega, \chi) \cdot z $ are given by
\begin{equation}\label{eq:torus action on the second parameter}
(\omega, \chi) \cdot \zeta = \big( q^{-2(\omega, \lambda^{(j)}) -2 (\chi, w_+(\lambda^{(j)}) )} \zeta_j \big)_{1 \leq j \leq k} \in (\Cbb^\times)^k
\end{equation}
and
\begin{equation}\label{eq:torus action on the third parameter}
(\omega, \chi) \cdot z = \big( q^{-(\omega + \chi, \varpi_i)} z_i \big)_{i \in \Ical(\wrm)} \in (\Cbb^\times)^{\Ical(\wrm)},
\end{equation}
respectively. Note that the factor $ 2 $ in the exponent of \eqref{eq:torus action on the second parameter} arises from the fact that the element $ a_j $ in \eqref{eq:basis from the reduced part} is quadratic in the matrix coefficients $ (c^\pm_{w_\pm, \lambda^{(j)}})^{\pm 1} $. 

In particular, the strata $ \Prim_\wrm(\Ocal(G_q)) $ for $ \wrm \in W \times W $ are closed under the canonical left and right actions of $ H $ on $ \Prim(\Ocal(G_q)) $, and these actions are transitive on each stratum.    

\section{Irreducible representations and the primitive spectrum} \label{sec:Irreducible representations and primitive spectrum}

Our aim in this section is to investigate the primitive spectrum \(\Prim(\Ocal(G_q)) \) with respect to Hilbert space representations. Specifically, we show that all primitive ideals in homogeneous Joseph strata can be obtained as kernels of topologically irreducible bounded representations of $ \Ocal(G_q) $ on Hilbert spaces, whereas primitive ideals in inhomogeneous Joseph strata do not admit any such realisation.

\subsection{Homogeneous strata}

We first discuss how to represent any primitive ideal contained in 
\[
\coprod_{w \in W} \Prim_{(w,w)}\big( \Ocal(G_q)\big)
\]
as the kernel of a topologically irreducible bounded representation of $ \Ocal(G_q) $ on some Hilbert space.

This builds on the foundational work of Vaksman-Soibelman and Soibelman \cite{VAKSMAN_SOIBELMAN_functionsquantumsu2}, \cite{SOIBELMAN_functionscompactquantumgroups} classifying all  topologically irreducible $ * $-representations of $ \Cf^\infty(K_q) $, compare \cite{KorogodskiSoibelman}, which we shall briefly review. Recall that $ \Cf^\infty(SU_q(2)) $ is the universal $ * $-algebra generated by two matrix elements $ u_{11}^\varpi $ and $ u_{21}^\varpi $ of the fundamental representation $ \pi_\varpi $ of $ U_q(\mathfrak{sl}(2,\Cbb)) $, subject to the relation
\[
u^\varpi =
\begin{pmatrix}
u_{11}^\varpi & -q {u_{21}^\varpi}^* \\
u_{21}^\varpi & {u_{11}^\varpi}^*
\end{pmatrix}
\text{ is a unitary matrix in }
M_2(\Cf^\infty(SU_q(2))).
\]
One readily verifies that the operators on $ \ell^2(\Nbb) $ defined by 
\begin{equation} \label{eq:infinite-dimensional representation of C( SUq(2))}
\varphi_q(u_{11}^\varpi) e_k = \sqrt{1 - q^{2k}}\, e_{k-1}, \quad \varphi_q(u_{21}^\varpi) e_k = q^k e_k, \quad k \in \Nbb,
\end{equation}
where $\{ e_0, e_1, e_2, \ldots \}$ is the standard orthonormal basis of $ \ell^2(\Nbb) $ and $e_{-1} = 0$, determine a unital $ * $-representation $ \varphi_q : \Cf^\infty(SU_q(2)) \rightarrow \Lbb(\ell^2(\Nbb)) $.

Fix $ 1 \leq i \leq N $ and let $ U_{q_i}^\Rbb(\kf_i) \subseteq U_q^\Rbb(\kf) $ denote the $*$-subalgebra generated by $E_i$, $F_i$, and $K_i^{\pm1}$. This gives rise to an embedding
\[
\tilde{U}_{q_i}^\Rbb(\mathfrak{su}(2)) \cong U_{q_i}^\Rbb(\kf_i) \hookrightarrow U_q^\Rbb(\kf)
\]
of Hopf $ * $-algebras, where $ \tilde{U}_{q_i}^\Rbb(\mathfrak{su}(2)) \subseteq U_{q_i}^\Rbb(\mathfrak{su}(2)) $ is the $ * $-subalgebra generated by $ E, F $, and $ K_\varpi^{\pm2} $. The transpose of this embedding defines a Hopf $ * $-algebra homomorphism $ \sigma_i : \Cf^\infty(K_q) \rightarrow \Cf^\infty(SU_{q_i}(2)) $. Consider the $ * $-representation
\begin{equation*}
\pi_{s_i} = \varphi_{q_i} \sigma_i : \Cf^\infty(K_q) \rightarrow \Lbb(\ell^2(\Nbb)).
\end{equation*}
Next fix $ w \in W $, and choose a reduced expression $ w = s_{i_1} \cdots s_{i_l} $. Then the $ * $-representation
\begin{equation*}
\pi_w = \pi_{s_{i_1}} \otimes \cdots \otimes \pi_{s_{i_l}} : \Cf^\infty(K_q) \rightarrow \Lbb\big(\ell^2(\Nbb) \otimes \cdots \otimes \ell^2(\Nbb)\big) \cong \Lbb(\ell^2(\Nbb^l)),
\end{equation*}
is independent of the choice of reduced expression up to unitary equivalence. 

For every $ \zeta \in \hf^*/i \hbar^{-1} \Qbf^\vee \cong H $, we then obtain a bounded representation $ \pi_w \otimes K_\zeta: \Ocal(G_q) \rightarrow \Lbb(\ell^2(\Nbb^l)) $ by taking the tensor product of $ \pi_w $ with the character  $ K_\zeta $.  
Note that the map $ K_\zeta: \Ocal(G_q) \rightarrow \Cbb $ is a $ * $-homomorphism if and only if $ \zeta \in \tfrak^*/i \hbar^{-1} \Qbf^\vee \cong T $, 
since
\begin{equation} \label{eq:when Kzeta preserves involution}
(K_\zeta, f^*) = \overline{(\hat{S}^{-1}(K_\zeta)^*, f)} = \overline{(K_{\overline{\zeta}}, f)}, \quad f \in \Cf^\infty(K_q).
\end{equation}
Let us now recall the following key result from  \cite[Section~3.4--3.6]{KorogodskiSoibelman}. 

\begin{thm} \label{thm:irreducible *-representations of C(Kq)}
Let $ \SSp(\Cf^\infty(K_q)) $ denote the set of equivalence classes of topologically irreducible $ * $-representations of $ \Cf^\infty(K_q) $. Then there are canonical bijections
\[
\SSp(\Cf^\infty(K_q)) \cong \{ \pi_w \otimes K_\zeta \mid w \in W,\, \zeta \in \tfrak^*/i \hbar^{-1} \Qbf^\vee \cong T \} \cong W \times T.
\]
Moreover, if  $ w = s_{i_1} \cdots s_{i_l} \in W $ is a reduced expression and $ \zeta \in \tfrak^*/i \hbar^{-1} \Qbf^\vee $, then the dense $(\pi_w \otimes K_\zeta)(\Cf^\infty(K_q))$-invariant subspace
\begin{equation} \label{eq:the dense simple module}
D_w = \{ x \in \ell^2(\Nbb^l) \mid x \textup{ has finite support} \} \subseteq \ell^2 (\Nbb^l) = \Hcal_{\pi_w}
\end{equation}
carries an irreducible representation of $ \Ocal(G_q) $ under the restricted action.
Hence $ \Ker(\pi_w \otimes K_\zeta) $ is primitive, and in fact
\[
\Ker(\pi_w \otimes K_\zeta) \in \Prim_{(w,w)} ( \Cf^\infty(K_q))
\]
for all $ w \in W $ and $ \zeta \in \tfrak^*/i \hbar^{-1} \Qbf^\vee $. 
\end{thm}

As a direct consequence of Theorem \ref{thm:irreducible *-representations of C(Kq)}, one obtains that the bounded representations $ \pi_w \otimes K_\omega: \Ocal(G_q) \rightarrow \Lbb(\Hcal_{\pi_w}) $ for any $ \omega \in \hf^*/i \hbar^{-1} \Qbf^\vee \cong H $ are topologically irreducible, and that they restrict to irreducible representations on $ D_w $. Therefore we arrive at the following description of homogeneous Joseph strata. 

\begin{thm} \label{thm:primitive ideals in homogeneous strata}
Let $ w \in W $. Then
\begin{equation}\label{eq:primitive ideals in homogeneous strata}
\Prim_{(w,w)}(\Ocal(G_q))
= \bigl\{\Ker(\pi_w \otimes K_\omega) \;\big|\; \omega \in \hf^*/i\hbar^{-1}\Qbf^\vee \cong H \bigr\}
\end{equation}
and
\begin{equation} \label{eq:degeneracy of primitive ideals}
\Ker(\pi_w \otimes K_\omega) = \Ker(\pi_w \otimes K_{\omega'}) \Longleftrightarrow \omega - \omega' \in \bigg( \frac{i}{2} \hbar^{-1} \sum_{j \in \Scal (w)} \Zbb \alpha_j ^\vee \bigg) \bigg/ i \hbar^{-1} \Qbf^\vee.
\end{equation}
In particular, if $0 \neq \omega - \omega' \in \weights $, then
\[
\Ker(\pi_w \otimes K_\omega ) \neq \Ker (\pi_w \otimes K_{\omega'}).
\]
\end{thm}

\begin{proof}
Observe that we have 
$ (\pi_w \otimes K_\omega)(f) = \pi_w \bigl( K_\omega \triangleright f \bigr) $ for all $ f \in \Ocal(G_q) $. 
Hence
\begin{equation}\label{eq:torus action on the kernel level}
\Ker(\pi_w \otimes K_\omega)
= K_{-\omega} \triangleright \Ker(\pi_w)
\in \Prim_{(w,w)}(\Ocal(G_q)).
\end{equation}
Equation~\eqref{eq:primitive ideals in homogeneous strata} now follows from the fact that the left $ H $-action on $ \Prim_{(w,w)}(\Ocal(G_q)) $ is transitive.

In order to prove \eqref{eq:degeneracy of primitive ideals}, note that since both $ \Ker(\pi_w \otimes K_\omega) $ and $ \Ker(\pi_w \otimes K_{\omega'})$ are contained in the homogeneous stratum $ \Prim_{(w,w)} (\Ocal(G_q)) $ by Theorem~\ref{thm:irreducible *-representations of C(Kq)}, we have
\[
\tilde{\Lcal}_{\operatorname{red}} (w,w) = \Ker (w -w) \cap \weights_{\Scal(w,w)} = \weights_{\Scal(w)}, 
\]
compare \eqref{defltilde}. In the setting of Theorem~\ref{thm:the primitive Joseph stratum}, we may therefore take the set $ \{\varpi_j \mid j \in \Scal(w) \} $ as basis of $ \tilde{\Lcal}_{\operatorname{red}}(w,w) $, and use $ (\Cbb^\times)^{\Scal(w)} $ in place of $ (\Cbb^\times)^k $ as parameter set, where $ k = \dim(\Ker(w-w)) - | \Ical(w)| $. Now, by \eqref{eq:torus action on the second parameter}--\eqref{eq:torus action on the third parameter} and \eqref{eq:torus action on the kernel level}, we obtain
\begin{align*}
\Ker(\pi_w \otimes K_\omega) = \Ker(\pi_w \otimes K_{\omega'}) \Longleftrightarrow &\;(\omega - \omega' , \varpi_j) \in \frac{i}{2} \hbar^{-1} \Zbb \text{ for all } j \in \Scal(w) \text{ and} \\ &(\omega - \omega', \varpi_j) \in i \hbar^{-1} \Zbb \text{ for all } j \in \Ical(w) \\
\Longleftrightarrow & \;\omega - \omega' \in \bigg( \frac{i}{2} \hbar^{-1} \sum_{j \in \Scal (w)} \Zbb \alpha_j ^\vee \bigg) \bigg/ i \hbar^{-1} \Qbf^\vee
\end{align*}
as claimed. 
\end{proof}

This completes our objective of associating a topologically  irreducible bounded representation of $ \Ocal(G_q) $ to each primitive ideal in the homogeneous strata. 

Let us record the following consequence of Theorem \ref{thm:irreducible *-representations of C(Kq)} at the level of representations. Recall the definition of $ D_w $ and $ \Hcal_{\pi_w} $ from \eqref{eq:the dense simple module}. 

\begin{prop} \label{prop:degeneracy of representations}
Let $ w \in W $. The representations $ \pi_w \otimes K_\omega $ for $ \omega \in \hf^*/i\hbar^{-1}\Qbf^\vee \cong H $
are mutually inequivalent, both when viewed as representations on the domain $ D_w $, and as bounded representations on its Hilbert space completion $ \Hcal_{\pi_w} $. 
\end{prop}

\begin{proof}
Let $ w = s_{i_1} \cdots s_{i_l} \in W $ be a reduced expression. By \cite[Section~3, (6.2.4)]{KorogodskiSoibelman}, the elements $ c^+_{w,\mu} $ from \eqref{eq:w-multiplicative elements} satisfy
\begin{equation} \label{eq:w-multiplicative elements under tensor product decomposition}
\pi_w\bigl(c^+_{w,\mu}\bigr)
= \varphi_{q_{i_1}}\bigl(u^\varpi_{21}\bigr)^{p_1^\mu}
\otimes \varphi_{q_{i_2}}\bigl(u^\varpi_{21}\bigr)^{p_2^\mu}
\otimes \cdots \otimes
\varphi_{q_{i_l}}\bigl(u^\varpi_{21}\bigr)^{p_l^\mu}
\end{equation}
for any $ \mu \in \weights^+ $, where 
$ p_j^\mu = \bigl(\mu,\, s_{i_l}s_{i_{l-1}}\cdots s_{i_{j+1}}(\alpha_{i_j}^\vee) \bigr) \in \mathbb{N} $ 
for $ j = 1, \ldots, l $. Therefore, it follows from \eqref{eq:infinite-dimensional representation of C( SUq(2))} that the operators 
\[
(\pi_w \otimes K_\omega)(c^+_{w,\mu}) 
= q^{(\omega,\mu)}\pi_w(c^+_{w,\mu}) 
\]
are diagonal in the standard basis of $ D_w $. The scalar $ q^{(\omega,\mu)} $ is the unique element of largest norm among the eigenvalues of $ (\pi_w \otimes K_\omega)(c^+_{w,\mu}) $, and therefore preserved by any invertible intertwiner. Since this holds for every $ \mu \in \weights^+ $ we see that $ (\pi_w \otimes K_\omega) $ and $ (\pi_w \otimes K_{\omega'}) $ are equivalent as representations on $ D_w $ if and only if $ \omega = \omega' \in \hf^*/i\hbar^{-1}\Qbf^\vee \cong H $. 

Essentially the same argument carries over to the case of bounded representations by considering the spectra of $ \pi_w\bigl(c^+_{w,\mu}\bigr) $, viewed as bounded operators on $ \Hcal_{\pi_w} $. It follows that there is a bounded invertible intertwiner between $ (\pi_w \otimes K_\omega) $ and $ (\pi_w \otimes K_{\omega'}) $ if and only if $ \omega = \omega' \in \hf^*/i\hbar^{-1}\Qbf^\vee \cong H $. 
\end{proof}

\begin{cor} \label{cor:degeneracy of primitive ideals}
For $ w \in W $, let
\[
\SPrim_w ( \Cf^\infty(K_q)) = \SPrim(\Cf^\infty(K_q)) \cap \Prim_{(w,w)} (\Ocal(G_q))
\]
and
\[
\SSp_{w} (\Cf^\infty(K_q)) = \big\{ \pi \in \SSp (\Cf^\infty(K_q)) \mid \Ker(\pi) \in \SPrim_w (\Cf^\infty(K_q)) \big\}.
\]
Then 
\[
\SSp_w(\Cf^\infty(K_q)) \ni \pi \longmapsto \Ker(\pi) \in \SPrim_w (\Cf^\infty(K_q))
\]
is a $ 2^{|\Scal(w)|} $-to-$ 1 $ map.
\end{cor}

\begin{proof}
This is a direct consequence of Theorems~\ref{thm:irreducible *-representations of C(Kq)}--\ref{thm:primitive ideals in homogeneous strata} and Proposition~\ref{prop:degeneracy of representations}.
\end{proof}

We close this section with an explicit description of $ \SPrim(\Cf^\infty(K_q)) $ 
inside $ \Prim(\Ocal(G_q)) $.

\begin{thm} \label{thm:SPrim(C(Kq))}
We have
\begin{align*}
\SPrim ( \Cf^\infty(K_q)) &= \{\Ker(\pi_w \otimes K_\omega ) \mid w \in W, \omega \in \tf^*/i\hbar^{-1}\Qbf^\vee \cong T \} \\ 
&= \{ J \in \Prim(\Ocal(G_q)) \mid *(J) = J \} \\
&= \{ J_{(w,w), \zeta, z} \mid w \in W,\, \zeta \in \Tbb^{\Scal(w)},\, z \in \Tbb^{ \Ical(w) } \}. 
\end{align*}
\end{thm}

\begin{proof}
The first equality follows from the definition of $ \SPrim ( \Cf^\infty(K_q)) $ and Theorem~\ref{thm:irreducible *-representations of C(Kq)}. Next note that $ \Ker(\pi_w \otimes K_\omega) = K_{-\omega} \triangleright \Ker(\pi_w) $ together with $ *\Ker(\pi_w) = \Ker(\pi_w) $ and \eqref{eq:when Kzeta preserves involution} implies $  * \Ker(\pi_w \otimes K_\omega) = \Ker(\pi_w \otimes K_{\overline{\omega}}) $ for any $ \omega \in \hf^*/i\hbar^{-1}\Qbf^\vee \cong H $. Combining this with \eqref{eq:degeneracy of primitive ideals} yields the second equality. 
The final equality then follows from Proposition~\ref{prop:primitive ideals under involution}.
\end{proof}

\begin{rmk} \label{rmk:note on parametrisation}
Note that the parametrisation of the primitive ideals in the final equality of Theorem~\ref{thm:SPrim(C(Kq))} depends on our choice of the elements \eqref{eq:w-multiplicative elements}, which, in turn, depends on our choice of the weight vectors \eqref{eq:w-weight vectors}. 
\end{rmk}

\subsection{Inhomogeneous strata}

Next we show that inhomogeneous strata do not admit any bounded representations on a Hilbert space. Let us say that an ideal $ J \subseteq \Ocal(G_q) $ is of type $ \wrm = (w_+, w_-) \in W \times W $ if
\[
I_\wrm \subseteq J \quad \text{and} \quad (J / I_\wrm) \cap E_\wrm = \emptyset.
\]
According to Theorem \ref{thm:strata of prime spectrum}, the stratum $ \Spec_\wrm(\Ocal(G_q)) $ is the set of all prime ideals of type $ \wrm $.

\begin{thm} \label{thm:no bdd repn for inhomogeneous strata}
Let $ \wrm = (w_+ , w_-) \in W \times W $ with $ w_+ \neq w_- $. Then there are no bounded representations of $ \Ocal(G_q) $ on Hilbert spaces whose kernel is of type $\wrm$.
\end{thm}

\begin{proof}
The following arguments are inspired by the proof of \cite[Proposition~2]{Lambert2004}.

Note first that, in order to prove the claim, it suffices to consider unital representations.
Hence, suppose $ \theta : \Ocal(G_q) \rightarrow \Lbb(\Hcal)$ is a unital bounded representation whose kernel is of type $ \wrm = (w_+, w_-) $ with $ w_+ \neq w_- $. Then either $ w_+ \nleq w_- $ or $ w_+ \ngeq w_- $ with respect to the Bruhat order. The map 
\[
\theta^* : \Ocal(G_q) \ni f \longmapsto \theta(f^*)^* \in \Lbb(\Hcal)
\]
is a bounded representation which satisfies $ \Ker(\theta^*) = *(\Ker(\theta)) $, and hence $ \Ker(\theta^*) $ is of type $ (w_-, w_+) $, see \eqref{eq:Joseph ideal conjugate} and \eqref{eq:w-multiplicative subset conjugate}. Therefore we may assume without loss of generality that $ w_+ \ngeq w_- $. 

Due to \eqref{eq:Bruhat order and Demazure module}, we then have
$
V^+_{w_+}(\rho) \nsupseteq V^+_{w_-}(\rho),
$
and this implies
\begin{equation}\label{eq:Demazure module for w+ rho avoids w- rho}
V^+_{w_+}(\rho) \cap V(\rho)_{w_- (\rho)} = 0. 
\end{equation}
Indeed, otherwise the one-dimensional weight space $ V(\rho)_{w_- (\rho)} $ would be contained in the $U_q(\bfrak_+)$-invariant subspace $V^+_{w_+}(\rho)$, and hence
\[
V^+ _{w_+}(\rho) \supseteq U_q(\bfrak_+) V(\rho)_{w_- (\rho)} = V^+_{w_-}(\rho),
\]
which is a contradiction.

Since weight spaces for different weights are orthogonal, \eqref{eq:Demazure module for w+ rho avoids w- rho} implies
\begin{equation*} 
V(\rho)_{w_- (\rho)} \subseteq V^+_{w_+}(\rho)^\perp.
\end{equation*}
Using that $ V(\rho)_{w_- (\rho)} = \Cbb e^\rho _{i_{w_- (\rho)}} $, where $ 1 \leq i_{w_- (\rho)} \leq n_\rho $ is the unique index such that $ \epsilon^{\rho} _{i_{w_- (\rho)}} = w_- (\rho) $, this shows in particular
\begin{equation}\label{eq:intersection with weight space and Demazure module}
e_{i_{w_- (\rho)}}^\rho \in V^+_{w_+}(\rho)^\perp.
\end{equation}
Now consider the matrix of operators
\[
\theta(u^\rho) = \big( \theta(u^\rho_{ij}) \big)_{1 \le i,j \le n_\rho} \in M_{n_\rho}\big(\Lbb(\Hcal)\big) \cong \Lbb(\Hcal^{n_\rho}).
\]
Since  $u^\rho \in M_{n_\rho}(\Ocal(G_q)) $ is invertible and $ \theta $ is an algebra homomorphism, it follows that $ \theta(u^\rho): \Hcal^{n_\rho} \rightarrow \Hcal^{n_\rho} $ is invertible. Hence, there exists a constant $ a > 0 $ such that 
\begin{equation} \label{estimate}
a \| \xi \| = a \| (\xi,0,\dots,0) \| \le \| \theta(u^\rho)(\xi,0,\dots,0) \| = \Big( \sum_{i=1}^{n_\rho} \| \theta(u^\rho_{i1}) \xi \|^2 \Big)^{1/2} 
\end{equation}
for all $ \xi \in \Hcal $. 
Let $ \Ical = \{ 1 \le i \le n_ \rho \mid e_i^\rho \notin V_{w_+}^+(\rho)^\perp \} $. Since $ \Ker(\theta) $ is of type $ \wrm $ we have $ I_{w_+}^+ \subseteq I_\wrm \subseteq \Ker(\theta) $, and therefore only indices $ i \in \Ical $ contribute to the sum in \eqref{estimate}. Thus, for any $ m \in \Nbb $ and $ \xi \in \Hcal $, the commutation relation \eqref{eq:commutation relations for w-multiplicative elements} yields
\begin{align*}
a \|\theta(c^-_{w_-,\, \rho})^m \xi\| 
&\le \Big( \sum_{i \in \Ical} \| \theta(u^\rho_{i1}) \theta(c^-_{w_-,\, \rho})^m \xi \|^2 \Big)^{1/2} \\
&= \Big(\sum_{i \in \Ical} q^{2m\big((\rho,\rho) - (w_-(\rho), \epsilon_i^\rho)\big)} \| \theta(c^-_{w_-, \, \rho})^m \theta(u^\rho_{i1}) \xi \|^2 \Big)^{1/2} \\
&\le q^{mr} \| \theta(c^-_{w_-,\, \rho})^m \| \Big( \sum_{i \in \Ical} \| \theta(u^\rho_{i1}) \xi \|^2 \Big)^{1/2} \\
&= q^{mr} \| \theta(c^-_{w_-,\, \rho})^m \| \| \theta(u^\rho)(\xi,0,\dots,0) \| \le q^{mr} \| \theta(c^-_{w_-,\, \rho})^m \| \| \theta(u^\rho) \| \|\xi\|,
\end{align*}
where
\begin{equation} \label{defr}
r = \min_{i \in \Ical} \big( (\rho,\rho) - (w_-(\rho), \epsilon_i^\rho) \big) = \min_{i \in \Ical} (\rho, \rho - w_-^{-1}(\epsilon_i^\rho)).
\end{equation}
Note that $ r \geq 0 $ since $ w_- ^{-1} (\epsilon_i ^\rho) \in \weights(\rho) $ and hence $ \rho - w_- ^{-1} (\epsilon_i ^\rho) \in \Qbf^+ $. Taking the supremum over all $ \xi \in \Hcal $ with $ \|\xi \| = 1 $ we therefore obtain 
\begin{equation} \label{inequality}
a \|\theta(c^-_{w_-,\, \rho})^m \| \le q^{mr} \| \theta(c^-_{w_-,\, \rho})^m \| \| \theta(u^\rho) \|. 
\end{equation}
Using again that $ \Ker(\theta) $ is of type $ \wrm $ and \eqref{eq:relation for w-multiplicative elements} we also know that $ \theta(c^-_{w_-, \rho})^m = \theta(c^-_{w_-, m \rho}) \neq 0 $ for all $ m $. Hence \eqref{inequality} implies
\[
a \le q^{mr} \| \theta(u^\rho) \|, \quad \forall m \in \Nbb, 
\]
and we conclude $ r = 0 $. 

However, note that $ i_{w_- (\rho)} \notin \Ical $ due to \eqref{eq:intersection with weight space and Demazure module}. Hence $ \epsilon_i^\rho \neq w_- (\rho)$ for all $ i \in \Ical $, and therefore $ r > 0 $ by \eqref{defr}. We have thus arrived at a contradiction, which finishes the proof.
\end{proof}

\begin{rmk}\label{rmk:Remark for RIKEN}
Keita Goto and Kan Kitamura pointed out to us that the proof of Theorem~\ref{thm:no bdd repn for inhomogeneous strata}  also applies to representations by bounded operators on Banach spaces. The only additional complication is that, under the assumption $ w_+ \neq w_- $, we no longer have the simplification of assuming $ w_+ \not\geq w_- $ without loss of generality. However, the case $ w_- \not\geq w_+ $ can be treated in an entirely analogous manner.
We therefore conclude that there are in fact no bounded representations of $ \Ocal(G_q) $ on Banach spaces whose kernels are of inhomogeneous type.
\end{rmk}

\begin{cor} \label{cor:no bdd repn for inhomogeneous strata}
Given a topologically irreducible bounded representation $ \theta: \Ocal(G_q) \to \Lbb(\Hcal) $ on a Hilbert space $ \Hcal $ we have 
\[
\Ker(\theta) \in \Spec_{(w,w)} \big( \Ocal(G_q) \big) 
\]
for some $ w \in W $.
\end{cor}

\begin{proof}
Every topologically primitive ideal is prime. Hence, according to \eqref{eq:stratification of prime spectrum}, we have
\[
\Ker(\theta) \in \Spec_{\wrm} \big( \Ocal(G_q) \big)  
\]
for some $ \wrm = (w_+, w_-) \in W \times W $. 
Now Theorem~\ref{thm:no bdd repn for inhomogeneous strata} implies $ w_+ = w_- $, and the result follows.
\end{proof}

Corollary \ref{cor:no bdd repn for inhomogeneous strata} shows that inhomogeneous strata cannot be accessed using bounded representations on Hilbert spaces. 

\begin{rmk}
There is no obstruction against realising primitive ideals in inhomogeneous strata using unbounded Hilbert space representations. In fact, one can equip the underlying vector space $ V $ of any representation $ \theta: \Ocal(G_q) \rightarrow \End(V) $ with an inner product, and view $ \theta $ as an unbounded representation of $ \Ocal(G_q) $ on the corresponding Hilbert space completion of $ V $, with common dense domain $ V $. 

This can be done in particular in the case $ \Ker(\theta) $ is a primitive ideal in an inhomogeneous stratum. However, for such an unbounded representation, the $ * $-structure of $ \Cf^\infty(K_q) $ cannot be compatible with taking formal adjoints. 
This follows from the fact that primitive ideals in inhomogeneous strata are not invariant under the $ * $-operation, compare Theorem \ref{thm:SPrim(C(Kq))}. 
\end{rmk}

\subsection{Beurling-Fourier algebras and primitive spectrum}

We shall now discuss a consequence of Corollary~\ref{cor:no bdd repn for inhomogeneous strata} for the spectra of Beurling-Fourier algebras of $ K_q $, which provides an interesting variation of the main result of \cite{LeeVoigt_BeurlingFourier}. To align with the notation used there, we denote $ \Ocal(G_q) $ by $ \Cf^\infty(K_q) $ throughout this section.

For $ \mu = \mu_1 \varpi_1 + \cdots + \mu_N \varpi_N \in \weights^+ $ we write $ |\mu| = (\mu, \mu)^{\frac{1}{2}} $. As all norms on a finite-dimensional vector space are equivalent, we can find a constant $ a > 0 $ such that
\begin{equation} \label{aestimate}
\mu_1 + \cdots + \mu_N \leq a |\mu|,
\end{equation}
independently of $ \mu $. 

Next, the $ * $-representation $ \pi_\mu: U_q^\mathbb{R}(\mathfrak{k}) \rightarrow \Lbb(V(\mu)) $ is a direct summand of the tensor product representation $\pi_{\varpi_1}^{\otimes \mu_1} \otimes \cdots \otimes \pi_{\varpi_N}^{\otimes \mu_N}$. Therefore, the matrix 
$ u^\mu = (u_{ij}^\mu)_{1 \leq i,j \leq n_\mu} \in  M_{n_\mu}\bigl(\Cf^\infty(K_q)\bigr) $
is unitarily equivalent to a subblock matrix of the matrix 
\[
(u^{\varpi_1})^{\otimes \mu_1} \otimes \cdots \otimes (u^{\varpi_N})^{\otimes \mu_N},
\] 
where we use the notation
\[
u \otimes v = ( u_{ij} v_{kl} )_{\substack{1 \leq i,j \leq n \\ 1 \leq k,l \leq m}} \quad \text{for } u = (u_{ij})_{1 \leq i,j \leq n} \text{ and } v = (v_{kl})_{1 \leq k,l \leq m}.
\]
Now let $ \theta : \Cf^\infty(K_q) \rightarrow \Lbb(\Hcal) $ be a unital bounded representation. Since $ \theta $ is an algebra homomorphism, the matrix of operators
\[
\theta(u^\mu) = \big( \theta(u^\mu_{ij}) \big)_{1 \leq i,j \leq n_\mu} \in M_{n_\mu}(\Lbb(\Hcal)) \cong \Lbb(\Hcal^{n_\mu})
\]
is unitarily equivalent to a corner
of the matrix of operators
\[
\theta(u^{\varpi_1})^{\otimes \mu_1} \otimes \cdots \otimes \theta(u^{\varpi_N})^{\otimes \mu_N} \in M_{n_{\varpi_1}^{\mu_1} \cdots n_{\varpi_N}^{\mu_N}}(\Lbb(\Hcal)).
\]
Hence we get
\[
\|\theta(u^\mu)\| \leq \|\theta(u^{\varpi_1})^{\otimes \mu_1} \otimes \cdots \otimes \theta(u^{\varpi_N})^{\otimes \mu_N}\| \leq \|\theta(u^{\varpi_1})\|^{\mu_1} \cdots \|\theta(u^{\varpi_N})\|^{\mu_N} \leq r^{\mu_1 + \cdots + \mu_N},
\]
where 
\[
r = \max \big\{ \|\theta(u^{\varpi_j})\| \mid 1 \leq j \leq N \big\}.
\]
We may choose $ \beta \geq \max\{ 1 , r^a \} $, where $ a $ is the constant in \eqref{aestimate}. Then we get 
\[
\|\theta(u^\mu)\| \leq r^{\mu_1 + \cdots + \mu_N} \leq \beta^{a^{-1} (\mu_1 + \cdots + \mu_N )} \leq \beta^{|\mu|}, \qquad \forall \mu \in \weights^+.
\]
By \cite[Proposition~2.5.1 and Proposition~4.3.6]{LeeVoigt_BeurlingFourier}, this implies
\[
\text{$\theta$ extends to a CB (completely bounded) representation } \tilde{\theta} : A(K_q, w_\beta) \rightarrow \Lbb(\Hcal),
\]
where $A(K_q, w_\beta)$ is the Beurling-Fourier algebra of $K_q$ associated with the central weight $w_\beta : \weights^+ \ni \mu \mapsto \beta^{|\mu|} \in [1, \infty)$ on the dual of $ K_q $; see \cite{LeeVoigt_BeurlingFourier} for further details.

Conversely, every CB representation $ A(K_q, w_\beta) \rightarrow \Lbb(\Hcal) $ obviously restricts to a representation $ \Cf^\infty(K_q) \rightarrow \Lbb(\Hcal) $.  Therefore, fixing a Hilbert space $\Hcal$, we have
\begin{align*}
\{&\theta: \Cf^\infty(K_q) \rightarrow \Lbb(\Hcal) \mid \theta \text{ is a unital representation} \} = \\
&\bigcup_{\beta \geq 1} \{ \theta : \Cf^\infty(K_q) \rightarrow \Lbb(\Hcal) \mid \text{$\theta$ is unital and admits a CB extension to $A(K_q, w_\beta)$} \}.
\end{align*}
Now let us define the \emph{CB-primitive spectrum} of $ \Cf^\infty(K_q) $ by
\begin{align*}
\Prim_{CB}(\Cf^\infty(K_q)) = 
\bigcup_{\beta \geq 1} \Big\{&\Ker(\theta)\in \Prim (\Cf^\infty(K_q)) \mid \theta: \Cf^\infty(K_q) \rightarrow \Lbb(\Hcal) \text{ is } \\ &\text{topologically irreducible and CB extends to } A(K_q, w_\beta) \Big\}.
\end{align*}
Combining our above observation with Theorems~\ref{thm:primitive ideals in homogeneous strata} and \ref{thm:no bdd repn for inhomogeneous strata}, we obtain the following complementary version of the main result of \cite{LeeVoigt_BeurlingFourier}. 

\begin{thm} \label{thm:Spectra of Beurling-Fourier covers Prim}
We have 
\begin{align*}
&\Prim_{CB}(\Cf^\infty(K_q)) = \coprod_{w \in W} \Prim_{(w,w)}(\Cf^\infty(K_q)). 
\end{align*}
That is, the CB-primitive spectrum of $ \Cf^\infty(K_q) $ identifies with the union of the homogeneous strata of $ \Prim(\Cf^\infty(K_q)) $. 
\end{thm}

\begin{rmk}\label{rmk:remark on invariant subspace problem}
One may wonder whether the condition
\[
\Ker(\theta) \in \Prim (\Cf^\infty(K_q))
\]
in the definition of the CB-primitive spectrum $ \Prim_{CB}(\Cf^\infty(K_q)) $ can be removed, or equivalently, whether $ \Prim_{CB}(\Cf^\infty(K_q)) $ in Theorem \ref{thm:Spectra of Beurling-Fourier covers Prim} can be replaced by the set of all kernels of topologically irreducible representations $ \Cf^\infty(K_q) \rightarrow \Lbb(\Hcal) $ which CB extend to $ A(K_q, w_\beta) $ for some $ \beta \geq 1 $. 

It turns out that this question is at least as difficult as the invariant subspace problem.
Recall that the invariant subspace problem for separable infinite-dimensional Hilbert spaces asks whether every bounded linear operator $ L $ on a separable infinite-dimensional Hilbert space $ \Hcal $ admits a nontrivial closed invariant subspace, that is, a closed subspace
\( 0 \neq M \subsetneqq \Hcal\)
such that
\(
L(M) \subseteq M.
\)
For recent surveys about the state of the art of the invariant subspace problem we refer to \cite{chalendar2025}, \cite{chalendar}.

Suppose that the invariant subspace problem has a negative solution. That is, there exists a bounded operator $ L $ on an infinite-dimensional Hilbert space $ \Hcal $ with no nontrivial closed invariant subspaces. Since, for any subspace $ M \subseteq \Hcal $,
\[
L(M) \subseteq M
\quad \Longleftrightarrow \quad
(L + \lambda I)(M) \subseteq M
\]
for any $ \lambda \in \Cbb $, 
we may assume without loss of generality that $ L $ is invertible. 

Let us now explain how this relates to our situation. Let $e \in W$ be the unit element. According to \cite[Lemma~8.8]{Yakimov2014}, the quotient $ \Cf^\infty(K_q)/I_{(e,e)} $ is isomorphic to $ \Cf^\infty(T) $, the algebra of matrix coefficients of the classical maximal torus  $T \subseteq K $, via
\[
\varphi: \Cf^\infty(K_q) / I_{{(e,e)}} \ni u^\lambda _{ i j} \longmapsto \delta_{ij} t^{\epsilon_j ^\lambda} \in \Cf^\infty (T), \qquad \lambda \in \weights^+, \, 1 \leq i,j \leq n_\lambda ,
\]
where $ t^\mu: T \rightarrow \Tbb $ is the character associated to $ \mu \in V(\lambda) $, defined by
\[
t^\mu : T \ni \exp(X) \longmapsto e^{ \mu(X) } \in \Tbb , \qquad X \in \tf.
\]
Note that $ \Cf^\infty(T) \cong \Cbb [ t^{\pm \varpi_1}, \dots, t^{\pm \varpi_N} ] $ is a Laurent polynomial algebra.

Consider the representation $ \theta: \Cf^\infty(K_q) \rightarrow \Lbb(\Hcal) $ obtained by composing the canonical projection $ \Cf^\infty(K_q) \rightarrow  \Cf^\infty(K_q)/I_{(e,e)} \cong \Cf^\infty(T) $ with the representation 
\[
\Cf^\infty(T) \ni t^{\varpi_j} \longmapsto L \in \Lbb(\Hcal), \qquad 1 \leq j \leq N.
\]
Since $ L $ has no nontrivial closed invariant subspaces it follows that $ \theta $ is topologically irreducible. Therefore $ \Ker(\theta) \subseteq \Cf^\infty(K_q) $ is a prime ideal, and in fact 
\[
\Ker(\theta) \in \Spec_{(e,e)}(\Cf^\infty(K_q)).
\]
However, 
the image of $ \theta $ must be infinite-dimensional, and therefore $ \Ker(\theta) $ cannot be primitive, compare \cite[Corollaries~8.11 and 8.12]{Yakimov2014}.

To summarise, the existence of a bounded operator $ L $ on an infinite-dimensional Hilbert space with no nontrivial closed invariant subspaces implies the existence of a topologically irreducible representation of $ \Cf^\infty(K_q) $ whose kernel is not primitive. 
\end{rmk}

Of course, in the above context one may still attempt to establish partial results for homogeneous strata of the form $ \Spec_{(w,w)}\big( \Cf^\infty(K_q) \big) $ with $ w \neq e $. We do not pursue this problem here. 

\section{Complex semisimple quantum groups and primitive spectrum}\label{sec:Complex semisimple quantum groups and primitive spectrum}

Recall the classical one-to-one correspondence
\[
\Irr(\Ocal^\mathbb{R}(K)) \ni \theta \longmapsto \Ker(\theta) \in \Prim(\Ocal^\mathbb{R}(K))
\]
for the algebra $ \Ocal^\mathbb{R}(K) $ of matrix coefficients of the compact group $ K $. The set $ \Irr(\Ocal^\mathbb{R}(K)) $ of equivalence classes of irreducible representations of $ \Ocal^\mathbb{R}(K)$ coincides with the set of holomorphic evaluations at points of $ G $. This is due to the fact that the elements of $ \Ocal^\mathbb{R}(K) $ extend canonically to holomorphic functions on the complexification $ G $.

On the quantum side, a natural candidate for the complexification of the standard quantisation $ K_q $ is the Drinfeld double $ G_q = D(K_q) $ of $ K_q $. There is an algebra embedding of $ \Cf^\infty(K_q) = \Ocal^\mathbb{R}(K_q) $ into the algebra of polynomial functions on $G_q$, which gives rise to the notion of \emph{quantum holomorphic evaluations at points of $G_q$}.

It is a natural task to investigate the relationship between quantum holomorphic evaluations and the primitive spectrum of $\Cf^\infty(K_q)$, and this is the focus of the present section.

\subsection{The \texorpdfstring{$q$}{q}-deformation of complex semisimple Lie groups and algebras of polynomials}

The underlying $*$-algebra of the algebraic quantum group dual to $\Cf^\infty(K_q)$ is
\begin{equation*}
\Df(K_q) = \bigoplus_{\mu \in \weights^+} \End(V(\mu)).
\end{equation*}
In analogy to the Iwasawa decomposition of a complex Lie group, the $ * $-algebra
\begin{equation}\label{eq:complex semisimple quantum group}
\Cfrak_c ^\infty(G_q) = \Cfrak^\infty(K_q) \otimes \Df(K_q) \cong \bigoplus_{\mu \in \weights^+} \Cfrak^\infty(K_q) \otimes \End(V(\mu))
\end{equation}
can be endowed with an algebraic quantum group structure, which is called \textit{the quantised algebra of functions on $G_q$}. It follows directly from the definition that we have a direct product decomposition 
\begin{equation} \label{eq:irreducible representations of Cc(G_q)}
\SSp(\Cfrak_c^\infty(G_q)) \cong \SSp(\Cfrak^\infty (K_q)) \times \SSp(\Df(K_q))
\end{equation}
of the set of topologically irreducible $ * $-representations of $ \Cfrak_c ^\infty(G_q) $ via the correspondence
\[
(\pi_C, \pi_\lambda) \longmapsto \pi_C \otimes \pi_\lambda,
\]
where $\pi_\lambda \in \SSp( \Df(K_q)) $ for $ \lambda \in \weights^+ $ is the projection onto the corresponding summand of $ \Df(K_q) $. The algebraic multiplier algebra of $\Cfrak_c ^\infty(G_q)$ is the $*$-algebra
\begin{equation*}
\Cfrak^\infty(G_q) = \prod_{\mu \in \weights^+} \Cfrak^\infty(K_q) \otimes \End(V(\mu)).
\end{equation*}
Every $\pi \in \SSp (\Cfrak_c ^\infty(G_q))$ extends uniquely to a unital $*$-representation of $\Cfrak^\infty (G_q)$, as it factors through one of the components in \eqref{eq:complex semisimple quantum group}, and we denote this extension by the same symbol $\pi$.

Let $ \Ocal^\Rbb (G_q) = \Cf^\infty(K_q) \bowtie \Cf^\infty(K_q)$ be the Drinfeld double associated to the universal $r$-form
\[
\Cf^\infty(K_q) \times \Cf^\infty(K_q) \ni (f,g) \longmapsto (\Rcal^{-1} , f \otimes g) \in \Cbb.
\]
More explicitly, we can describe $ \Ocal^\Rbb (G_q) $ as the tensor product $\Cf^\infty(K_q) \otimes \Cf^\infty(K_q)$ endowed with the Hopf algebra structure such that 
\begin{gather*}
(f \bowtie g) (h \bowtie k) = (\Rcal^{-1} , h_{(1)} \otimes g_{(1)}) f h_{(2)} \bowtie g_{(2)} k (\Rcal , h_{(3)} \otimes g_{(3)}), \\
\Delta(f \bowtie g) = ( f_{(1)} \bowtie g_{(1)} ) \otimes ( f_{(2)} \bowtie g_{(2)}), \quad \epsilon( f \bowtie g) = \epsilon(f) \epsilon(g), \\
S(f \bowtie g) = ( \Rcal , f_{(1)} \otimes g_{(1)} ) S(f_{(2)}) \bowtie S(g_{(2)}) ( \Rcal^{-1} , f_{(3)} \otimes g_{(3)} ),
\end{gather*}
for $f,g,h,k \in \Cf^\infty(K_q)$. Moreover, $ \Ocal^\Rbb (G_q) $ becomes a Hopf $ * $-algebra with respect to the involution determined by 
\[
(f \bowtie g)^* = g^* \bowtie f^*
\]
for $f,g \in \Cf^\infty(K_q)$. 

Recall that we have a canonical embedding $U_q(\gfrak) \subseteq \prod_{\mu \in \weights^+} \End(V(\mu))$ and consider the map
\begin{equation*}
i: \Ocal^\Rbb(G_q) \rightarrow \Cf^\infty (G_q), \quad i(f \bowtie g) = f_{(1)} g_{(1)} \otimes l^- (f_{(2)}) l^+ (g_{(2)}).
\end{equation*}
One checks that $ i $ is an injective Hopf $*$-algebra homomorphism, the image of which we will denote by
\[
\Pol^\Rbb (G_q) = i ( \Ocal^\Rbb (G_q) ) \subseteq \Cf^\infty(K_q) \otimes U_q ^\Rbb (\kf) \subseteq \Cf^\infty(G_q),
\]
and refer to as \textit{the algebra of polynomial functions on $G_q$}. The restrictions of $i$ to the first and second tensor components of $ \Ocal^\Rbb(G_q) = \Cf^\infty(K_q) \otimes \Cf^\infty(K_q) $ will be denoted by $ i_-, i_+ : \Cf^\infty(K_q) \rightarrow \Pol^\Rbb (G_q)$, respectively, and their images by
\[
\Pol(G_q) = i_- (\Cf^\infty(K_q)), \quad \overline{\Pol}( G_q ) = i_+(\Cf^\infty(K_q)).
\]
Note that $ i_- $ and $ i_+ $ are injective algebra homomorphisms such that $ i_\pm(f)^* = i_\mp(f^*) $ for all $ f \in \Cf^\infty(K_q) $, which implies $ *(\Pol(G_q)) = \overline{\Pol}(G_q) $.
We refer to the subalgebras
\(\Pol(G_q), \overline{\Pol}(G_q) \subseteq \Pol^\mathbb{R}(G_q) \)
as the \emph{algebras of holomorphic and antiholomorphic polynomials on \( G_q \)}, respectively. The embedding
\begin{equation}\label{eq:q-holomorphic embedding}
i_- : \Cf^\infty(K_q) \hookrightarrow \Pol^\Rbb(G_q)
\end{equation}
serves as a quantum analogue of the classical holomorphic extension map. 

\begin{lem}\label{lem:PolR(Gq) contains 1 otimes FrUq(g)}
We have
\(
1 \otimes F_r U_q (\gf) \subseteq \Pol^\Rbb (G_q).
\)
\end{lem}

\begin{proof}
For $f \in \Cf^\infty (K_q)$ we calculate
\begin{multline*}
\Pol^\Rbb (G_q) \ni i \big( S(f_{(1)} ) \bowtie f_{(2)} \big) = S(f_{(2)})f_{(3)} \otimes l^- \big(S( f_{(1)} ) \big) l^+ (f_{(4)} ) \\
= 1 \otimes \hat{S}^{-1} \Big( l^+ \big( S(f_{(2)} ) \big) \hat{S} l^- \big( S(f_{(1)} ) \big) \Big) = 1 \otimes \hat{S}^{-1} I' \big( S(f) \big).
\end{multline*}
Thus, the assertion follows from \eqref{eq:locally finite part is big1} and \eqref{eq:I'-isomorphism}.
\end{proof}

Recall that any topologically irreducible $ * $-representation of $ \Cf_c ^\infty(G_q) $ can be extended uniquely to a topologically irreducible unital $ * $-representation of $ \Cf^\infty(G_q) $. 

\begin{prop}\label{prop:images of irreps for Cc(Gq) and PolR(Gq) coincide}
For any $\pi \in \SSp ( \Cf_c ^\infty (G_q) )$ we have
\begin{equation*}
\pi(\Cf_c ^\infty(G_q)) = \pi ( \Pol^\Rbb (G_q) ).
\end{equation*}
\end{prop}

\begin{proof}
According to the direct product decomposition in \eqref{eq:irreducible representations of Cc(G_q)} we can write $ \pi = \pi_C \otimes \pi_\lambda $ for $ \pi_C \in \SSp(\Cfrak^\infty (K_q)), \pi_\lambda \in \SSp(\Df(K_q)) $. 
By Lemma~\ref{lem:PolR(Gq) contains 1 otimes FrUq(g)} and \eqref{eq:locally finite part is big in representation}, we then have
\begin{equation*}
1 \otimes \Lbb(V(\lambda)) = 1 \otimes \pi_\lambda ( F_r U_q( \gf ) ) \subseteq \pi ( \Pol^\Rbb (G_q) ).
\end{equation*}
Now let $ p_l \in \Lbb ( V(\lambda)) $ denote the projection onto the subspace $\Cbb e^\lambda _l \subseteq V(\lambda)$ for $ 1 \leq l \leq n_\lambda$. Inspecting the definition of $ i $, we then have 
\[
(1 \otimes p_l ) \pi i ( u^\mu _{jk} \bowtie 1 ) (1 \otimes p_l) = \pi_C ( u^\mu_{jk} ) \otimes p_l \pi_\lambda \big( l^- (u^\mu _{kk} ) \big) p_l = \pi_C (u^\mu _{jk} ) \otimes q^{ - ( \epsilon_k ^\mu , \epsilon_l ^\lambda ) } p_l,
\] 
which implies that $ \pi_C(u^\mu _{jk}) \otimes 1 = \sum_{l = 1} ^{n_\lambda} \pi_C (u^\mu _{jk} ) \otimes p_l $ is contained in $\pi( \Pol^\Rbb (G_q) )$ for any $\mu \in \weights^+$ and $1 \leq j, k \leq n_\mu$. That is, we have 
\[
\pi_C ( \Cf^\infty(K_q) ) \otimes 1 \subseteq \pi ( \Pol^\Rbb (G_q)).
\]
Combining these observations yields 
\[
\pi ( \Cf_c ^\infty (G_q) ) = \pi_C ( \Cf^\infty (K_q) ) \otimes \pi_\lambda ( \Dfrak (K_q) ) \subseteq \pi(\Pol^\Rbb (G_q)).
\]
Since $\pi( \Pol^\Rbb (G_q) ) \subseteq \pi ( \Cf^\infty (G_q) ) = \pi ( \Cf_c ^\infty(G_q) )$, we have in fact
\[
\pi(\Cf_c ^\infty (G_q) ) = \pi( \Pol^\Rbb (G_q) )
\]
as claimed. 
\end{proof}

\begin{cor}\label{cor:SSp(Cc(Gq)) restricts to SSP(Pol(Gq))}
There is a canonical inclusion \[
\SSp(\Cf_c ^\infty(G_q)) \subseteq \SSp(\Pol^\Rbb(G_q))
\]
obtained via extension of topologically irreducible $ * $-representations from $ \Cf_c ^\infty (G_q) $ to $ \Cf^\infty(G_q) $ followed by restriction to $\Pol^\Rbb (G_q)$. 
\end{cor}

We will show in the next subsection that the inclusion in Corollary \ref{cor:SSp(Cc(Gq)) restricts to SSP(Pol(Gq))} is in fact an equality.

\subsection{Iwasawa decomposition of complex semisimple quantum groups}

In the operator algebraic setting, complex quantum groups are defined as Drinfeld doubles of compact $ q $-deformations. The Iwasawa decomposition of a complex quantum group is basically built into the construction in this approach. 

However, the starting point for Podle\'s and Woronowicz in \cite{Podles1990} was different. They were considering coordinate functions and their conjugates, together with some representation-theoretic conditions, to derive a characterisation of the algebra of polynomial functions on the quantum Lorentz group $ SL_q(2, \Cbb) $. In this setting, the Iwasawa decomposition is not obvious from the outset. 

The goal of this section is to generalise the original approach of Podle\'s and Woronowicz in \cite{Podles1990} to the case of arbitrary complex semisimple quantum groups. This will allow us to give a complete description of the topologically irreducible $ * $-representations of algebras of polynomial functions on complex quantum groups by bounded operators on a Hilbert space. 

Let us say that a $ * $-representation $ \pi_D: U_q ^\Rbb (\kf) \rightarrow \Lbb(\Hcal) $ on a Hilbert space $ \Hcal $ is \emph{positive} if the operators $ \pi_D(K_\mu) $ are positive for all $ \mu \in \weights $. 

\begin{thm} \label{thm:Iwasawa decomposition}
Let $ \pi: \Pol^\Rbb(G_q) \rightarrow \Lbb(\Hcal) $ be a unital $ * $-representation. Then there exists a unital $ * $-representation $ \pi_C: \Cf^\infty(K_q) \rightarrow \Lbb(\Hcal) $ and a positive unital $ * $-representation $ \pi_D: U_q ^\Rbb (\kf) \rightarrow \Lbb(\Hcal) $ with commuting ranges such that 
\begin{equation}\label{eq:Iwasawa decomposition}
\pi = (\pi_C \otimes \pi_D)\big|_{\Pol^\Rbb (G_q)},
\end{equation}
where $ (\pi_C \otimes \pi_D)(f \otimes X) = \pi_C (f) \pi_D (X) $. Moreover, $ \pi_C $ and $ \pi_D $ are uniquely determined by $ \pi $. 
\end{thm}

\begin{proof}
Using Lemma~\ref{lem:PolR(Gq) contains 1 otimes FrUq(g)}, we define a unital $ * $-representation $ \pi_D : F_r U_q ^\Rbb (\kf)  \rightarrow \Lbb(\Hcal)$ by
\[
\pi_D(x) = \pi(1 \otimes x).
\]
Pick $ \mu \in \weights^+ $ and consider $ u^\mu \in M_{n_\mu}(\Cf^\infty(K_q)) $. Recall also that $ i_-: \Cf^\infty(K_q) \rightarrow \Pol^\Rbb(G_q) $ is the embedding of holomorphic polynomials, and note that
\[
i_-(u^\mu) = \Big( \sum_{1 \leq k \leq n_\mu} u^\mu_{ik} \otimes l^-(u^\mu_{kj}) \Big)_{1 \leq i, j \leq n_\mu} \in M_{n_\mu} \big(\Pol^\Rbb(G_q) \big)
\]
is invertible. It follows that
\begin{multline}\label{eq:square of pi i1(u mu)}
S = \pi i_-(u^\mu)^* \pi i_-(u^\mu) 
= \pi\big(1 \otimes l^-(u^\mu)^* l^-(u^\mu)\big) 
= \pi\big(1 \otimes J(u^\mu)\big) \\
= \pi_D J(u^\mu) \in M_{n_\mu}(\Lbb(\Hcal)) \cong \Lbb(\Hcal^{n_\mu})
\end{multline}
is an invertible positive operator.

Let $\iota : \Hcal \rightarrow \Hcal^{n_\mu} \cong \Hcal \otimes V(\mu) $ denote the inclusion corresponding to the highest weight vector of $ V(\mu) $ and let $ p = \iota^*: \Hcal^{n_\mu} \rightarrow \Hcal $. Then 
\[
\pi_D(K_{-2\mu}) = \pi_D J(u^\mu_{11}) = p S \iota \in \Lbb(\Hcal)
\]
is a positive self-adjoint operator. 

We claim that $ \pi_D(K_{-2\mu}) $ is in fact invertible. Indeed, since $ S^{\frac{1}{2}} $ is invertible, the operator $S^{\frac{1}{2}} \iota : \Hcal \rightarrow \Hcal^{n_\mu}$ is injective with closed range. In particular, if $ \xi \in \Hcal $ satisfies $ \pi_D(K_{-2\mu}) \xi = 0 $, then
\[
0 = \langle \xi , \pi_D(K_{-2\mu}) \xi \rangle = \langle \xi, p S \iota \xi \rangle = \langle \iota \xi , S \iota \xi \rangle = \| S^{\frac{1}{2}} \iota \xi \|^2,
\]
so $\xi = 0$. This shows that $\pi_D(K_{-2\mu})$ is injective.

Next, by the closed range theorem, 
the adjoint $ p S^{\frac{1}{2}}: \Hcal^{n_\mu} \rightarrow \Hcal$ of $ S^{\frac{1}{2}} \iota $ also has closed range. 
Therefore,
\begin{align*}
\Img( \pi_D(K_{-2\mu})) &= p S^{\frac{1}{2}} S^{\frac{1}{2}} \iota(\Hcal) = p S^{\frac{1}{2}} \big( \Img(S^{\frac{1}{2}} \iota) \big) \\
&= p S^{\frac{1}{2}} \big( \overline{\Img(S^{\frac{1}{2}} \iota)} \big) = p S^{\frac{1}{2}} \big( \Ker(p S^{\frac{1}{2}})^\perp \big) \\
&= p S^{\frac{1}{2}} \Big( \Ker(p S^{\frac{1}{2}}) \oplus \Ker(p S^{\frac{1}{2}})^\perp \Big) = p S^{\frac{1}{2}}(\Hcal^{n_\mu}) = \Img(p S^{\frac{1}{2}}) 
\end{align*}
is closed in $\Hcal$. Since $\pi_D(K_{-2 \mu})$ is injective and self-adjoint, we conclude
\[
\Img(\pi_D ( K_{-2\mu})) = \overline{\Img(\pi_D(K_{-2 \mu}))} = \Ker(\pi_D(K_{-2\mu}))^\perp = \Hcal,
\]
proving the surjectivity of $\pi_D (K_{-2\mu} )$.

Hence $ \pi_D(K_{-2\mu}): \Hcal \rightarrow \Hcal $ is bijective, and its inverse is bounded by the open mapping theorem, thus completing the proof of our claim.

Recall from \eqref{eq:ej}--\eqref{eq:fj} that $ K_{-2\rho} E_j $ and $ K_{-2\rho} K_j F_j $ are contained in $ F_r U_q(\gf) $. Using that $\{ \pi_D(K_{-2\mu}) \mid \mu \in \weights^+ \}$ is a family of commuting invertible positive operators it follows that
\begin{gather*}
\tilde{K}_\lambda = \pi_D(K_{-2\varpi_1})^{- \frac{\lambda_1}{2}} \cdots \pi_D(K_{-2\varpi_N})^{- \frac{\lambda_N}{2}}, \\
\tilde{E}_j = \pi_D(K_{-2\rho})^{-1} \pi_D(K_{-2\rho} E_j), \\
\tilde{F}_j = \pi_D ( K_{-2 \varpi_1} )^{\frac{a_{1j}}{2}} \cdots \pi_D ( K_{-2 \varpi_N} )^{\frac{a_{Nj}}{2}}  \pi_D(K_{-2\rho})^{- 1} \pi_D(K_{-2\rho} K_j F_j) 
\end{gather*}
for $ \lambda = \lambda_1 \varpi_1 + \cdots + \lambda_N \varpi_N \in \weights $ and $ 1 \leq j \leq N $ are well-defined operators in $ \Lbb(\Hcal) $. Here, $(a_{ij})_{1 \leq i,j \leq N}$ is the Cartan matrix, which satisfies $\alpha_j = \sum_{1 \leq i \leq N} a_{ij} \varpi_i$.

For $ \lambda \in \weights $ consider the operator-valued holomorphic function 
\[
f(z) = \tilde{K}_\lambda^z q^{(\lambda, \alpha_j)z}\pi_D(K_{-2 \rho} K_j F_j) \tilde{K}_\lambda^{-z} 
\]
on $ \mathbb{C} $. Since $ \pi_D $ is a $ * $-representation we have $ f(z + 2) = f(z) $ for all $ z \in \mathbb{C} $, and since $ \tilde{K}_\lambda^z $ is unitary for purely imaginary $ z $, it follows that $ f(z) $ is bounded, thus constant by Liouville's theorem. As a consequence, we obtain 
\[
\tilde{K}_\lambda^z \pi_D(K_{-2 \rho} K_j F_j) \tilde{K}_\lambda^{-z} = q^{-(\lambda, \alpha_j)z}\pi_D(K_{-2 \rho} K_j F_j )
\]
for all $ z \in \mathbb{C} $. In a similar way one shows 
\[
\tilde{K}_\lambda^z \pi_D(K_{-2\rho} E_j) \tilde{K}_\lambda^{-z} = q^{(\lambda, \alpha_j)z} \pi_D(K_{-2\rho} E_j)
\]
for all $ z \in \mathbb{C} $. Using these formulas it is straightforward to check that the operators $ \tilde{K}_\lambda, \tilde{E}_j, \tilde{F}_j $ satisfy the defining relations for $ U_q^\Rbb(\kf) $.  
By construction, the resulting unital $ * $-representation of $ U_q^\Rbb(\kf) $ is positive, and it agrees with $ \pi_D $ on $ F_r U_q^\Rbb(\kf) $. Hence, we conclude that $ \pi_D $ extends to a positive unital $ * $-representation of $ U_q^\Rbb(\kf) $, which we denote again by $ \pi_D $.

Recall that the matrix coefficients $ u^\mu_{ij} $ for $ \mu \in \weights^+ $ and $ 1 \leq i,j \leq n_\mu $ form a linear basis of $ \Cf^\infty(K_q) $. We define a linear map $\pi_C : \Cf^\infty(K_q) \rightarrow \Lbb(\Hcal)$ via
\begin{equation}\label{eq:definition of pic(umu)}
\pi_C(u^\mu) = \pi i_-(u^\mu) \pi_D l^-(u^\mu)^{-1} \in M_{n_\mu}(\Lbb(\Hcal)), \quad \mu \in \weights^+.
\end{equation}
Then $ \pi_C(u^\mu) $ is invertible and
\begin{align*}
\pi_C(u^\mu)^* \pi_C(u^\mu) &= \pi_D l^+((u^\mu)^{*})^{-1} \pi i_- (u^\mu)^* \pi i_-(u^\mu) \pi_D l^-(u^\mu)^{-1} \\
&= \pi_D l^+(u^\mu) \pi_D J(u^\mu) \pi_D l^-(u^\mu)^{-1} \\
&= \pi_D l^+(u^\mu) \pi_D \hat{S} l^+(u^\mu) \pi_D l^-(u^\mu) \pi_D l^-(u^\mu)^{-1}
=1,
\end{align*}
so $ \pi_C (u^\mu) $ is in fact unitary.

Pick $ \lambda, \mu \in \weights^+ $ and consider the tensor product
\[
u^\lambda \otimes u^\mu = \big( u^\lambda _{ij} u^\mu _{kl} \big) _ {\substack {1 \leq i,j \leq n_\lambda \\ 1 \leq k,l \leq n_\mu } }\in M_{n_\lambda n_\mu} (\Cf^\infty (K_q)).
\]
Then the matrix of operators
\[
\pi i_- ( u^\lambda \otimes u^\mu ) ^* \pi i_- (u^\lambda \otimes u^\mu ) \in M_{n_\lambda n_\mu} ( \Lbb (\Hcal) )
\]
satisfies, by the definition of $ i_- $, 
\begin{align*}
\pi i_- ( u^\lambda \otimes u^\mu ) ^* \pi i_- (u^\lambda \otimes u^\mu ) &= \pi_D J ( u^\lambda \otimes u^\mu ) \\
&= \bigg( \sum_{ \substack{ 1 \leq r \leq n_\lambda \\ 1 \leq s \leq n_\mu } } \pi_D \Big( \hat{S} l^+ ( u^\lambda _{ir} u^\mu _{ k s} ) l^- ( u^\lambda _{r j}  u^\mu _{s l} ) \Big) \bigg) _{ \substack{1 \leq i, j \leq n_\lambda \\ 1 \leq k, l \leq n_\mu} } \\
&= \bigg( \sum_{ \substack{ 1 \leq r \leq n_\lambda \\ 1 \leq s \leq n_\mu } } \pi_D l^- ( u^\mu _{sk} )^* \pi_D J( u^\lambda _{ij} ) \pi_D l^- ( u^\mu _{ s l } ) \bigg) _{ \substack{1 \leq i, j \leq n_\lambda \\ 1 \leq k, l \leq n_\mu} }
\end{align*}
as in \eqref{eq:square of pi i1(u mu)}. 
Moreover, using \eqref{eq:definition of pic(umu)} and the fact that $ \pi i_- $ is an algebra homomorphism, we also have
\begin{multline*}
\pi i_-(u^\lambda \otimes u^\mu) ^* \pi i_-(u^\lambda \otimes u^\mu) = \bigg( \sum_{ \substack{ 1 \leq r \leq n_\lambda \\ 1 \leq s \leq n_\mu } } \pi i_- ( u^\mu _{ s k } )^* \pi i_- ( u^\lambda _{ r i } ) ^* \pi i_- (u ^\lambda _{rj} ) \pi i_- ( u^\mu _{sl} ) \bigg)_{ \substack{1 \leq i,j \leq n_\lambda \\ 1 \leq k,l \leq n_\mu} } \\
= \bigg( \sum_{ 1 \leq s , p , t \leq n_\mu } \pi_D l^- ( u^\mu _{ p k } )^* \pi_C ( u^\mu _{ s p } )^*  \pi_D J ( u^\lambda _{ij} ) \pi_C ( u^\mu _{ s t } ) \pi_D l^- ( u^\mu _{t l} ) \bigg)_{ \substack{1 \leq i,j \leq n_\lambda \\ 1 \leq k,l \leq n_\mu} } .
\end{multline*}
Equating the last expressions of the two preceding equations and cancelling the conjugation by the invertible matrix $\pi_D l^-(u^\mu)$ from both sides, we obtain
\[
\delta_{kl} \, \pi_D J ( u^\lambda _{ij}) = \sum_{1 \leq s \leq n_\mu} \pi_C ( u^\mu _{s k } )^* \pi_D J( u^\lambda _{ij} ) \pi_C ( u^\mu _{sl}) , \quad 1 \leq i, j \leq n_\lambda, \, 1 \leq k, l \leq n_\mu,
\]
which is equivalent to
\[
\pi_C ( u^\mu _{kl} ) \pi_D J (u^\lambda _{ij} ) = \pi_D J ( u^\lambda _{ij} ) \pi_C ( u^\mu _{kl} ) , \quad 1 \leq i,j \leq n_\lambda, \, 1 \leq k,l \leq n_\mu.
\]
This shows that the elements in $ \pi_C(\Cf^\infty (K_q)) $ commute with those in $ \pi_D(F_rU_q(\gf)) $. Since $ K_{-2 \lambda} \in F_rU_q (\gf) $ for all $ \lambda \in \weights^+ $, it follows from the definition of $ \pi_D (K_\lambda) = \tilde{K} _\lambda $ and functional calculus that we have
\[
\pi_C(u^\mu _{kl}) \pi_D (K_{\lambda}) = \pi_D ( K_\lambda ) \pi_C(u^\mu _{kl}), \quad \lambda, \mu \in \weights^+, \, 1 \leq k,l \leq n_\mu.
\]
This, in turn, implies by \eqref{eq:locally finite part is big2} that $ \pi_C(\Cf^\infty(K_q)) $ and $ \pi_D(U_q (\gf)) $ commute.

Next we show that $ \pi_C $ is a unital $ * $-algebra representation. Let $ \lambda, \mu \in \weights^+ $ and consider the decomposition
\[
u^\lambda \otimes u^\mu \cong \bigoplus_{\nu \subseteq \lambda \otimes \mu} u^\nu
\]
of the tensor product 
$ u^\lambda \otimes u^\mu $ into irreducibles. Then
\begin{align*}
\pi_C(u^\lambda \otimes u^\mu) &\cong \bigoplus_{\nu \subseteq \lambda \otimes \mu} \pi_C(u^\nu) = \bigoplus_{\nu \subseteq \lambda \otimes \mu} \pi i_-(u^\nu) \bigoplus_{\nu \subseteq \lambda \otimes \mu} \pi_D l^-(u^\nu)^{-1} \\
&\cong \pi i_-(u^\lambda \otimes u^\mu) \pi_D l^-(u^\lambda \otimes u^\mu)^{-1} \\
&= \big( \pi i_-(u^\lambda) \otimes \pi i_-(u^\mu) \big) \big( 1 \otimes \pi_D l^-(u^\mu) \big)^{-1} \big( \pi_D l^-(u^\lambda) \otimes 1 \big)^{-1} \\
&= \big( \pi i_-(u^\lambda) \otimes 1 \big) \big( 1 \otimes \pi_C(u^\mu) \big) \big( \pi_D l^-(u^\lambda) \otimes 1 \big)^{-1},
\end{align*}
and since the images of $ \pi_C $ and $ \pi_D $ commute, this equals
\[
= \big( \pi i_-(u^\lambda) \otimes 1 \big) \big( \pi_D l^-(u^\lambda)^{-1} \otimes 1 \big) \big( 1 \otimes \pi_C(u^\mu) \big) = \pi_C(u^\lambda) \otimes \pi_C(u^\mu).
\]
We conclude that
\[
\pi_C(u^\lambda_{ij} u^\mu_{kl}) = \pi_C(u^\lambda_{ij}) \pi_C(u^\mu_{kl}) \quad \text{for } 1 \leq i,j \leq n_\lambda \text{ and } 1 \leq k,l \leq n_\mu,
\]
proving the multiplicativity of $ \pi_C $. Observe also that we have $ \pi_C(1) = \id $ by construction. 
Finally, since $\pi_C(u^\mu) \in M_{n_\mu}(\Lbb(\Hcal))$ is unitary for all $\mu \in \weights^+$, we have
\[
\pi_C(u^\mu)^* = \pi_C(u^\mu)^{-1} = \pi_C((u^\mu)^{-1}) = \pi_C((u^\mu)^*) \in M_{n_\mu}(\Lbb(\Hcal)),
\]
which shows that $ \pi_C $ is $ * $-preserving.

Now consider the unital $ * $-representation
\[
\pi_C \otimes \pi_D : \Cf^\infty(K_q) \otimes U_q^\Rbb(\kf) \ni f \otimes X \longmapsto \pi_C(f) \pi_D (X) \in \Lbb(\Hcal).
\]
According to \eqref{eq:definition of pic(umu)} we have
\[
(\pi_C \otimes \pi_D) i_- = \pi i_-,
\]
and since $ i_-(\Cf^\infty(K_q)) $ generates $ \Pol^\Rbb(G_q) $ as a $ * $-algebra, this implies
\[
(\pi_C \otimes \pi_D)|_{\Pol^\Rbb(G_q)} = \pi
\]
as required. 

Uniqueness of $ \pi_D: U_q(\gf) \rightarrow \Lbb(\Hcal) $ follows from the fact that this map is uniquely determined by $ \pi $ on $ F_r U_q(\gf) $, together with the requirement that $ \pi_D $ is positive. Uniqueness of $ \pi_C $ is then a consequence of $ \eqref{eq:definition of pic(umu)} $.  
\end{proof}

\begin{lem}\label{lem:bounded representation of Uq(g)}
Assume that $ \pi_D: U_q^\Rbb(\kf) \rightarrow \Lbb(\Hcal)$ is a topologically irreducible positive unital $ * $-representation. Then there exists a uniquely determined weight $ \lambda \in \weights^+ $ such that $ \pi_D $ is unitarily equivalent to the $ * $-representation $ 
\pi_\lambda: U_q ^\Rbb (\kf) \longrightarrow \Lbb(V(\lambda))
$ 
associated with $ \lambda $. 
\end{lem}

\begin{proof}
We claim that the positive self-adjoint operators $ \{\pi_D (K_{\mu}) \mid \mu \in \weights \}$ are simultaneously diagonalisable on $\Hcal$.

Let $ \Bcal(\Rbb_{+}^N) $ denote the Borel $ \sigma $-algebra of $ \Rbb_{+}^N  = (0,\infty)^N $, and let $ e: \Bcal(\Rbb_{+}^N) \rightarrow \Lbb(\Hcal) $ be the spectral measure associated with the commuting positive self-adjoint operators
\[
\{ \pi_D(K_{\varpi_1}), \dots, \pi_D(K_{\varpi_N}) \} \subseteq \Lbb(\Hcal).
\]
Moreover denote the support of $ e $ by $ \sigma $, the \textit{joint spectrum} of the commuting operators. The commutation relations
\[
\pi_D(E_i)\pi_D(K_{\varpi_j}) = q_j^{-\delta_{ij}} \pi_D(K_{\varpi_j}) \pi_D(E_i), \quad
\pi_D(F_i)\pi_D(K_{\varpi_j}) = q_j^{\delta_{ij}} \pi_D(K_{\varpi_j}) \pi_D(F_i)
\]
imply
\begin{equation}\label{eq:commutation of spectral measure with E F}
\pi_D(E_j) e(\omega) = e(Q_j \omega) \pi_D(E_j), \quad
\pi_D(F_j) e(\omega) = e(Q_j^{-1} \omega) \pi_D(F_j), \quad \omega \in \Bcal(\Rbb_{+}^N),
\end{equation}
where $ Q_j^{\pm 1}: \Rbb_{+}^N \rightarrow \Rbb_{+}^N $ are the maps
\[
Q_j^{\pm 1}(x_1, \dots, x_N) = (x_1, \dots, x_{j-1}, \, q_j^{\pm 1} x_j, \, x_{j+1}, \dots, x_N).
\]
This is because the map $\Bcal(\Rbb_{+}^N) \ni \omega \mapsto e(Q_j^{\pm 1} \omega) \in \Lbb(\Hcal)$ is the spectral measure corresponding to the commuting self-adjoint operators
\[
\{ \pi_D(K_{\varpi_1}), \dots, \pi_D(K_{\varpi_{j-1}}), \, q_j^{\mp 1} \pi_D(K_{\varpi_j}), \, \pi_D(K_{\varpi_{j+1}}), \dots, \pi_D(K_{\varpi_N}) \} \subseteq \Lbb(\Hcal).
\]
Let $ x = (x_1, \dots, x_N), \, y = (y_1, \dots, y_N) \in \sigma $. We claim that there exists $ l = (l_1, \dots, l_N) \in \Zbb^N$ such that
\[
(y_1, \dots, y_N) = (q_1^{l_1} x_1, \dots, q_N^{l_N} x_N).
\]
Suppose no such $l$ exists. This means in particular $ x \neq y $. For $l \in \Zbb^N$, define $ Q^l = Q_1^{l_1} \cdots Q_N^{l_N} : \Rbb_{+}^N \rightarrow \Rbb_{+}^N$. Since the action
\[
\Zbb^N \times \Rbb_{+}^N \ni (l, z) \mapsto Q^l z \in \Rbb_{+}^N
\]
is free and proper, there exist open subsets $ U_0, V_0 \subset \Rbb_{+}^N $ with $ x \in U_0 $ and $ y \in V_0 $ such that
\[
V_0 \cap Q^l U_0 = \emptyset \quad \text{for all } l \in \Zbb^N,
\]
see \cite[Lemma~21.11]{JohnLee}. Set
\[
U = \bigcup_{l \in \Zbb^N} Q^l U_0, \quad V = \bigcup_{l \in \Zbb^N} Q^l V_0 \subseteq \Rbb_{+}^N.
\]
By construction, $U \cap V = \emptyset$ and $e(U) \neq 0 \neq e(V)$. Moreover, \eqref{eq:commutation of spectral measure with E F} implies that $e(U)$ and $e(V)$ commute with $\pi_D(X)$ for all $X \in U_q^\Rbb(\kf)$. It follows that $e(U)\Hcal$ and $e(V)\Hcal$ are nontrivial mutually orthogonal $ \pi_D $-invariant closed subspaces of $ \Hcal $, contradicting the topological irreducibility of $ \pi_D $.

Therefore, we see that if $ x \in \sigma $, then
\[
\sigma \subseteq \big\{ (q_1^{l_1} x_1, \dots, q_N^{l_N} x_N) \mid l = (l_1, \dots, l_N) \in \Zbb^N \big\}.
\]
Since the latter set is discrete in $\Rbb_+ ^N$, the operators $\{ \pi_D(K_{\varpi_1}), \dots, \pi_D(K_{\varpi_N}) \}$ are simultaneously diagonalisable on $\Hcal$. Hence, the family $\{ \pi_D(K_\mu) \mid \mu \in \weights \}$ is also simultaneously diagonalisable. 

Choose a simultaneous eigenvector $\xi_0 \in \Hcal$. The commutation relations \eqref{eq:commutation of spectral measure with E F} imply that the action of any $ \pi_D(E_j) $ on $ \xi_0 $ produces simultaneous eigenvectors with strictly smaller eigenvalues, or zero. Since the joint spectrum $\sigma$ is a compact subset of $\Rbb_+^N$, there exist $ m \in \Nbb$ and $ 1 \leq i_1, \dots , i_m \leq N$ such that
\[
\xi = \pi_D(E_{i_1}) \cdots \pi_D(E_{i_m}) \xi_0
\]
is a simultaneous eigenvector for $\{ \pi_D(K_\mu) \mid \mu \in \weights \}$ satisfying
\begin{equation*}
\pi_D(E_j) \xi = 0, \quad 1 \leq j \leq N.
\end{equation*}
Thus, the linear subspace
\[
\pi_D(U_q^\Rbb(\kf)) \xi \subseteq \Hcal
\]
is a highest weight representation. Applying \eqref{eq:commutation of spectral measure with E F} with $\pi_D(F_j)$ and appealing to the compactness of $\sigma$ once more, we see that this subspace must be finite-dimensional since increasing the eigenvalues cannot be continued indefinitely. According to \cite[Corollary~3.44]{VoigtYuncken}, together with the positivity of the eigenvalues of $\pi_D(K_\mu)$, it follows that there exists $\lambda \in \weights^+$ such that
\begin{equation*}
\pi_D(K_\mu) \xi = q^{(\mu, \lambda)} \xi, \quad \forall \mu \in \weights.
\end{equation*}
Hence $ \pi_D(U_q^\Rbb(\kf)) \xi \subseteq \Hcal $ is an irreducible subrepresentation of  $\pi_D $ which is unitarily equivalent to $ \pi_\lambda $. Using again topological irreducibility of $ \pi_D $, we obtain $ \pi_D(U_q^\Rbb(\kf)) \xi = \Hcal $, and this finishes the proof. 
\end{proof}

Let us fix some $ M \geq 1 $. Denote by $ C^*_M(U_q^\Rbb(\kf)) $ the universal unital $ C^* $-algebra generated by bounded operators $ E_i, K_{\varpi_i}, F_i $ for $ 1 \leq i \leq N $ on some Hilbert space, satisfying the defining relations of $ U_q^\Rbb(\kf) $, with all $ K_{\varpi_i} $ positive,  
and norm bounds $ \|X \| \leq M $ for all $ X = E_i, K_{\varpi_i} ^{\pm 1}, F_i $. That is, $ C^*_M(U_q^\Rbb(\kf)) $ is the norm closure of the image of $ U_q^\Rbb(\kf) $ in the universal positive $ * $-representation of $ U_q^\Rbb(\kf) $ subject to the condition that the generators $ E_i, K_{\varpi_i} ^{\pm 1}, F_i $ have norm at most $ M $. We note that the canonical unital $ * $-homomorphism $ i_M: U_q^\Rbb(\kf) \rightarrow C^*_M(U_q^\Rbb(\kf)) $ will fail to be injective, but we have $ C^*_M(U_q^\Rbb(\kf))\neq 0 $ since the trivial representation of $ U_q^\Rbb(\kf) $ factorises through $ C^*_M(U_q^\Rbb(\kf)) $ for any $ M \geq 1 $.

\begin{lem} \label{lem: norm bounded positive reps}
The $ C^* $-algebra $ C^*_M(U_q^\Rbb(\kf)) $ is finite-dimensional for any $ M \geq 1 $. 
\end{lem}

\begin{proof}
According to Lemma \ref{lem:bounded representation of Uq(g)}, every topologically irreducible $ * $-representation of $ C^*_M(U_q^\Rbb(\kf)) $ must be unitarily equivalent to some $ \pi_\lambda $ for $ \lambda \in \weights^+ $. However, since the norms of the generators $ K_{\varpi_i} $ are bounded by $ M $, only finitely many weights $ \lambda $ are possible. We conclude that $ C^*_M(U_q^\Rbb(\kf)) $ is a unital $ C^* $-algebra with finitely many topologically irreducible $ * $-representations up to equivalence, which are all finite-dimensional. This means that $ C^*_M(U_q^\Rbb(\kf)) $ is finite-dimensional. 
\end{proof}

Now let $ \pi_D : U_q^\Rbb(\kf) \rightarrow \Lbb(\Hcal) $ be an arbitrary positive unital $ * $-representation. Then there exists some $ M \geq 1 $ such that $ \pi_D $ factorises through $ C^*_M(U_q^\Rbb(\kf)) $. 
Due to Lemma \ref{lem: norm bounded positive reps}, the representation $ \pi_D $ decomposes as a finite direct sum of multiples of the representations $ \pi_\lambda $. 

Let us say that a nondegenerate $ * $-representation $ \Pi_D: \Df(K_q) \rightarrow \Lbb(\Hcal) $ is of \emph{finite type} if there exists a finite set $ F \subseteq \weights^+ $ such that $ \Pi_D $ factorises through the quotient $ \bigoplus_{\lambda \in F} \End(V(\lambda)) $. Note that any nondegenerate $ * $-representation $ \Df(K_q) \rightarrow \Lbb(\Hcal) $ of finite type extends uniquely to a unital $ * $-representation of the algebraic multiplier algebra
\[
 M(\Df(K_q)) = \prod_{\lambda \in \weights^+} \End (V(\lambda))
\]
of $\Df (K_q)$.

Our above considerations imply the following characterisation of positive unital $ * $-representations of $ U_q^\Rbb(\kf) $. 

\begin{cor} \label{cor: irreducibility diagonalisability}
Let $ \pi_D: U_q^\Rbb(\kf) \rightarrow \Lbb(\Hcal) $ be a positive unital $ * $-representation. Then there exists a uniquely determined nondegenerate $ * $-representation of finite type $ \Pi_D: \Df(K_q) \rightarrow \Lbb(\Hcal) $ such that the restriction of the corresponding unital $ * $-homomorphism $ M(\Df(K_q)) \rightarrow \Lbb(\Hcal) $ to $ U_q^\Rbb(\kf) \subset M(\Df(K_q)) $ agrees with $ \pi_D $. Moreover, every nondegenerate $ * $-representation of $\Df(K_q)$ of finite type arises in this way.
\end{cor}

\begin{proof}
As explained above, we find a uniquely determined minimal finite subset $F \subseteq \weights^+ $ such that $ \pi_D $ factorises via unital $ * $-homomorphisms through $ \bigoplus_{\lambda \in F} \End(V(\lambda)) $. We can extend the resulting unital $ * $-homomorphism $ \bigoplus_{\lambda \in F} \End(V(\lambda)) \rightarrow \Lbb(\Hcal) $ in a unique way to a nondegenerate $ * $-homomorphism of finite type $ \Pi_D: \Df(K_q) \rightarrow \Lbb(\Hcal) $, which has the required properties by construction. Uniqueness of $ \Pi_D $ is straightforward to check.  

If $ \Pi: \Df(K_q) \rightarrow \Lbb(\Hcal) $ is a nondegenerate $ * $-representation of finite type then 
precomposition of the corresponding unital $ * $-homomorphism $ M(\Df(K_q)) \rightarrow \Lbb(\Hcal) $ with the canonical map $ U_q^\Rbb(\kf) \rightarrow M(\Df(K_q)) $ gives rise to a positive unital $ * $-representation $ \pi_D: U_q^\Rbb(\kf) \rightarrow \Lbb(\Hcal) $, and one checks that the above procedure applied to $ \pi_D $ recovers $ \Pi $. 
\end{proof}

The following theorem is a generalisation of \cite[Theorem~5.4.4]{Podles1990}, establishing the Iwasawa decomposition for complex semisimple quantum groups. In a similar way as above, we say that a nondegenerate $ * $-representation $ \pi: \Cf_c ^\infty(G_q) \rightarrow \Lbb(\Hcal) $ is of finite type if there exists a finite set $ F \subseteq \weights^+ $ such that $ \pi $ factorises through the quotient $ \bigoplus_{\lambda \in F} \Cf^\infty (K_q) \otimes \End(V(\lambda)) $.

\begin{thm} \label{thm:generalisation of Podles and Woronowicz}
Let $ G_q $ be a complex semisimple quantum group. Every unital $ * $-representation $ \Pol^\Rbb(G_q) \rightarrow \Lbb(\Hcal) $ on a Hilbert space $ \Hcal $ corresponds to a uniquely determined  nondegenerate $ * $-representation $ \Cf_c^\infty(G_q) \rightarrow \Lbb(\Hcal) $ of finite type, and every nondegenerate $ * $-representation of finite type of $ \Cf_c^\infty(G_q) $ arises in this way.
In particular, we have a canonical identification
\[
\SSp(\Pol^\Rbb(G_q)) \cong \SSp(\Cf_c^\infty(G_q))
\]
between the topologically irreducible $ * $-representations of $ \Pol^\Rbb(G_q) $ and $ \Cf_c^\infty(G_q) $. 
\end{thm}

\begin{proof}
According to Theorem~\ref{thm:Iwasawa decomposition}, we have a bijective correspondence between unital $ * $-representations $ \pi: \Pol^\Rbb(G_q) \rightarrow \Lbb(\Hcal) $ and pairs $ (\pi_C, \pi_D) $ consisting of a unital $ * $-representation $ \pi_C: \Cf^\infty(K_q) \rightarrow \Lbb(\Hcal) $ and a positive unital $ * $-representation $ \pi_D: U_q ^\Rbb (\kf) \rightarrow \Lbb(\Hcal) $ with commuting ranges. Moreover, by Corollary \ref{cor: irreducibility diagonalisability}, the map $ \pi_D $ corresponds to a uniquely determined nondegenerate $ * $-homomorphism $ \Pi_D: \Df(K_q) \rightarrow \Lbb(\Hcal) $ of finite type, and $ \pi_C $ and $ \Pi_D $ have commuting ranges. That is, they correspond to a uniquely determined nondegenerate $ * $-representation $ \Cf_c^\infty(G_q) \rightarrow \Lbb(\Hcal) $ of finite type as claimed. Conversely, by extension to the algebraic multiplier algebra, any nondegenerate $ * $-representation $ \Cf_c^\infty(G_q) \rightarrow \Lbb(\Hcal) $ of finite type induces a unital $ * $-representation $ \Pol^\Rbb(G_q) \rightarrow \Lbb(\Hcal) $. 

For the final part, note that a nondegenerate $ * $-representation $ \Cf_c^\infty(G_q) \rightarrow \Lbb(\Hcal) $ is topologically irreducible if and only if it is induced by a topologically irreducible $ * $-representation of $ \Cf^\infty(K_q) \otimes \End(V(\lambda)) $ for some $ \lambda \in \weights^+ $ via the canonical projection map \( \id \otimes \pi_\lambda : \Cf_c ^\infty(G_q) \rightarrow \Cf^\infty(K_q) \otimes \End(V(\lambda)) \). Clearly all such representations are of finite type.
\end{proof}

\begin{rmk}\label{rmk:Podles and Woronowicz}
Let us explain why Theorem~\ref{thm:generalisation of Podles and Woronowicz} indeed generalises \cite[Theorem~5.4.4]{Podles1990} from the case $ G_q = SL_q(2, \Cbb) $ to arbitrary $ G_q $. We note that the proof of \cite[Theorem~5.4.4]{Podles1990} given by Podle\'s and Woronowicz applies only to topologically irreducible $ * $-representations, as it relies on \cite[Theorem~5.1.3]{Podles1990}, which is established only for topologically irreducible $ * $-representations.

Define $ \alpha, \beta, \gamma, \delta \in \Pol(SL_q(2,\Cbb)) $ by
\begin{equation*}
    \begin{pmatrix}
        \alpha & \beta \\
        \gamma & \delta
    \end{pmatrix}
    = i_-(u^\varpi) =
    \begin{pmatrix}
        u^\varpi_{11} & u^\varpi_{12} \\
        u^\varpi_{21} & u^\varpi_{22}
    \end{pmatrix}
    \otimes
    \begin{pmatrix}
        l^-(u^\varpi_{11}) & l^-(u^\varpi_{12}) \\
        l^-(u^\varpi_{21}) & l^-(u^\varpi_{22})
    \end{pmatrix}
    \in M_2(\Pol(SL_q(2,\Cbb))).
\end{equation*}
By definition, these four elements generate $ \Pol(SL_q(2,\Cbb)) $ as an algebra, and the $ * $-subalgebra they generate is $ \Pol^\Rbb(SL_q(2,\Cbb)) $. In \cite[Theorem~5.4.1]{Podles1990}, it is shown that $ \alpha, \beta, \gamma, \delta $ satisfy the defining relations \cite[(1.9)--(1.25)]{Podles1990} of $ SL_q(2,\Cbb) $. Therefore, for any unital $*$-representation $\pi: \Pol^\Rbb(SL_q(2,\Cbb)) \rightarrow \Lbb(\Hcal)$, the entries of
\[
\begin{pmatrix}
\pi(\alpha) & \pi(\beta) \\
\pi(\gamma) & \pi(\delta)
\end{pmatrix}
\in M_2(\Lbb(\Hcal))
\]
also satisfy \cite[(1.9)--(1.25)]{Podles1990}. Conversely, any operators $\tilde{\alpha}, \tilde{\beta}, \tilde{\gamma}, \tilde{\delta} \in \Lbb(\Hcal)$ satisfying \cite[(1.9)--(1.25)]{Podles1990} give rise to a unital $*$-homomorphism $ \pi: \Pol^\Rbb (SL_q(2,\Cbb)) \rightarrow \Lbb(\Hcal) $ such that 
\[
\begin{pmatrix}
\pi(\alpha) & \pi(\beta) \\
\pi(\gamma) & \pi(\delta)
\end{pmatrix}
= 
\begin{pmatrix}
\tilde{\alpha} & \tilde{\beta} \\
\tilde{\gamma} & \tilde{\delta}
\end{pmatrix}.
\]
Note that if we write 
\[
u^\varpi = \begin{pmatrix}
a & b \\ c & d
\end{pmatrix} 
\]
then \eqref{eq:Iwasawa decomposition} becomes
\[
\begin{pmatrix}
\pi(\alpha) & \pi(\beta) \\ \pi(\gamma) & \pi(\delta)
\end{pmatrix}
=
\begin{pmatrix}
\pi_C(a) & \pi_C(b) \\ \pi_C(c) & \pi_C(d)
\end{pmatrix}
\otimes
\begin{pmatrix}
\pi_D(K)^{-1} & \pi_D l^-(b) \\ 0 & \pi_D(K)
\end{pmatrix}
\]
because $l^- (a) = K^{-1}$, $l^- (d) = K$, and $l^- (c) = 0 $. This recovers \cite[Theorem~1.3]{Podles1990} as the left-hand side defines an $ SL_q(2,\Cbb) $-matrix.

If $ \pi $ is topologically irreducible, then \cite[Theorem~5.1.3]{Podles1990} and \cite[Theorem~5.4.4]{Podles1990} imply that the operators $ \pi(\alpha), \pi(\beta), \pi(\gamma), \pi(\delta) $ determine a unital $*$-representation of $ \Cf^\infty(SL_q(2,\Cbb)) $ by bounded operators on $ \Hcal $, which restricts to a nondegenerate $ * $-representation $ \Cf_c^\infty(SL_q(2,\Cbb)) \rightarrow \Lbb(\Hcal) $. This implies the irreducible case of Theorem \ref{thm:generalisation of Podles and Woronowicz} in the case $ G_q = SL_q(2, \Cbb) $. 

Conversely, every nondegenerate $ * $-representation $ \Cf_c^\infty(SL_q(2, \Cbb)) \rightarrow \Lbb(\Hcal) $ extends uniquely to a nondegenerate $ * $-representation $ C_0(SL_q(2, \Cbb)) \rightarrow \Lbb(\Hcal) $, and it follows that Theorem \ref{thm:generalisation of Podles and Woronowicz} for $ G_q = SL_q(2, \Cbb) $ implies \cite[Theorem~5.4.4]{Podles1990} in the irreducible case.
\end{rmk}

\subsection{Quantum holomorphic evaluations and primitive spectrum}

Recall that the embedding $ i_-: \Cfrak^\infty(K_q) \rightarrow \Pol^\Rbb(G_q) $ from \eqref{eq:q-holomorphic embedding} serves as a quantum analogue of the holomorphic extension map. Motivated by this analogy, we associate to each \( \pi \in \SSp(\Pol^\Rbb(G_q)) \) the bounded representation
\begin{equation*}
\theta_\pi: \Ocal(G_q) = \Cfrak^\infty(K_q) \xrightarrow{i_-} \Pol^\Rbb(G_q) \xrightarrow{\pi} \Lbb(\Hcal_\pi),
\end{equation*}
and refer to it as the \emph{quantum holomorphic evaluation at \( \pi \)}. Note that, due to Theorem~\ref{thm:generalisation of Podles and Woronowicz}, we can replace $ \SSp(\Pol^\Rbb(G_q)) $ by $ \SSp(\Cf_c ^\infty(G_q)) $ here.

Classically, we have a canonical identification  
\begin{equation}\label{eq:classical relation between Prim(Cf(K)) and holomorphic evaluations}
\Max(\Ocal^\mathbb{R}(K)) = \Max(\Ocal(G)) = \{ \Ker(\ev_z) \mid z \in G \}
\end{equation}
of the maximal ideal space of the coordinate ring $ \Ocal^\mathbb{R}(K) = \Ocal(G) $ with the set of all kernels of evaluation homomorphisms at the points of $ G $. By way of analogy, it is natural to compare the primitive spectrum \(\Prim(\Ocal(G_q)) \) with the set 
\[
\{\Ker(\theta_\pi) \mid \pi \in \SSp(\Cf_c^\infty(G_q)) \}
\]
of all kernels of quantum holomorphic evaluations at the points of \(\SSp(\Cf_c^\infty(G_q))\). More specifically, one may ask to what extent the kernels of quantum holomorphic evaluations detect the primitive spectrum of $ \Ocal(G_q) $. 

The following result shows that the answer to this question is not as straightforward as in the classical case. 

\begin{prop} \label{prop:kernels of q-holomorphic evaluations are not primitive}
Let $ \lambda \in \weights^+ $ and consider $ \pi = \pi_C \otimes \pi_\lambda \in \SSp(\Cf_c^\infty(G_q)) $, where $ \pi_C: \Cf^\infty(K_q) \rightarrow \Lbb(\Hcal_C) $ is a topologically irreducible $ * $-representation. Then 
\[
\Hcal_{C} \otimes V(\lambda)_{w_0(\lambda)} \subseteq \Hcal_{C} \otimes V(\lambda)
\]
is a closed $ \theta_\pi(\Ocal(G_q)) $-invariant linear subspace. In particular, if $ \lambda \neq 0 $ then $ \theta_\pi $ is not topologically irreducible.
    
Moreover, if $ \lambda \neq 0$ and $ \pi_C = \epsilon $ is the counit of $ \Cf^\infty(K_q) $ then $ \Ker(\theta_\pi) $ is not primitive.
\end{prop}

\begin{proof}
The first claim follows from the fact that the lowest weight space $ V(\lambda)_{w_0(\lambda)} $ is preserved by $ l^-(\Cf^\infty(K_q)) \subseteq U_q(\bfrak_-) $. 
Note that
\begin{multline*}
\theta_\pi(u^\mu _{ij})(\xi \otimes e_{n_\lambda} ^\lambda) = \big( \pi_C(u^\mu _{ij}) \xi \big) \otimes \big(\pi_\lambda(K_{- \epsilon_{j}^\mu}) e_{n_\lambda}^\lambda \big) \\
= \big(\pi_C(u^\mu _{ij}) \xi \big) \otimes q^{-(\epsilon_j^\mu, w_0(\lambda))} e_{n_\lambda}^\lambda = \big((\pi_C \otimes K_{-w_0(\lambda)}) (u^\mu_{ij})\big)(\xi \otimes e_{n_\lambda}^\lambda )
\end{multline*}
for all $\mu \in \weights^+, 1 \leq i,j \leq n_\mu $, and $ \xi \in \Hcal_{C} $. Thus, the subrepresentation on the invariant subspace $\Hcal_C \otimes V(\lambda)_{w_0(\lambda)}$ is given by $ \pi_C \otimes K_{-w_0(\lambda)} $. 

For the final claim assume $ 0 \neq \lambda \in \weights^+$ and let $ \pi_C = \epsilon $. In this case the Hilbert space $ \Hcal_{C} \otimes V(\lambda) \cong V(\lambda) $ is finite-dimensional, and therefore $ \Ocal(G_q)/\Ker(\theta_\pi) $ is a finite-dimensional algebra. Moreover, by the above calculation, $ \epsilon \otimes K_{-w_0(\lambda)} \cong K_{-w_0(\lambda)} $ is a subrepresentation of $\theta_\pi$, and hence
\begin{equation*}
\Ker(\theta_\pi) \subseteq \Ker(K_{-w_0(\lambda)}).
\end{equation*}
Since $ \lambda \neq 0 $ there exists $ 1 \leq j \leq N $ such that $ v = F_j \cdot e^\lambda_1 \neq 0 $. Let us abbreviate $ c = c^\lambda_{\lambda, v} $ for the matrix coefficient associated with $ v $ and the highest weight vector $ e^\lambda_1 \in V(\lambda) $. By inspecting the definition of the $ l $-functionals and the universal $R$-matrix $\Rcal$, compare \cite[Theorem~3.108]{VoigtYuncken}, one checks that, up to a nonzero scalar, the operator $ \theta_\pi(c) = \pi_\lambda(l^-(c)) $ equals $ \pi_\lambda(K_{-\lambda} K_j F_j) $.  
Therefore, $\theta_\pi(c)(e_1^\lambda) $
is nonzero, whereas
\(
K_{-w_0(\lambda)}(c) = 0.
\)
As a consequence, we have a strict inclusion
\begin{equation}\label{eq:Ker theta pi is not primitive}
\Ker(\theta_\pi) \subsetneqq \Ker(K_{ - w_0(\lambda)}).
\end{equation}
Now assume that $ \Ker(\theta_\pi) $ is primitive. Then the finite-dimensional quotient algebra $ \Ocal(G_q)/\Ker(\theta_\pi) $ is isomorphic to a full matrix algebra by the Jacobson density theorem. This quotient must have dimension strictly greater than one by \eqref{eq:Ker theta pi is not primitive}. Therefore, the natural surjection
\[
\Ocal(G_q)/\Ker(\theta_\pi) \longrightarrow \Ocal(G_q)/ \Ker(K_{- w_0(\lambda)})
\]
becomes a nontrivial algebra homomorphism from a noncommutative matrix algebra onto $ \Cbb $, which is impossible.
\end{proof}

Recall that the closed subsets in the Zariski topology of $ \Prim(\Ocal(G_q)) $ are of the form $\{P \in \Prim(\Ocal(G_q)) \mid I \subseteq P \} $ for some ideal $ I \subseteq \Ocal(G_q) $. 
In view of Proposition \ref{prop:kernels of q-holomorphic evaluations are not primitive}, we consider the subsets
\[
\Prim(\theta_\pi) = \{P \in \Prim(\Ocal(G_q)) \mid \Ker(\theta_\pi) \subseteq P \} \subseteq \Prim (\Ocal(G_q))
\]
and examine to what extent they cover $ \Prim(\Ocal(G_q)) $ as $ \pi \in \SSp(\Cf_c^\infty(G_q)) $ varies.

\begin{thm} \label{thm:the best that Ker theta pi can do}
Let $ \lambda \in \weights^+ $ and consider $ \pi = \pi_C \otimes \pi_\lambda \in \SSp(\Cf_c^\infty(G_q)) $, where $ \pi_C: \Cf^\infty(K_q) \rightarrow \Lbb(\Hcal_C) $ is a topologically irreducible $ * $-representation. Then we have 
\begin{equation}\label{eq:the best that Ker theta pi can do}
\Prim(\theta_\pi) = \Big\{P \in \Prim(\Ocal(G_q)) \; \Big| \; \bigcap_{\nu \in \Pbf(\lambda)} \Ker(\pi_C \otimes K_{-\nu} ) \subseteq P \Big\}. 
\end{equation}
That is, the set $ \Prim(\theta_\pi) \subseteq \Prim (\Ocal(G_q)) $ is the Zariski closure of the finite subset $\big\{\Ker(\pi_C \otimes K_{-\nu}) \mid \nu \in \Pbf(\lambda)\big\} $. 
\end{thm}

\begin{proof}
Let $ \nu \in \Pbf(\lambda) $ and consider the subquotient of $ \Hcal_C \otimes V(\lambda) $ defined by 
\[
\big(\Hcal_C \otimes \bigoplus_{\nu' \leq \nu} V(\lambda)_{\nu'} \big) \big/ \big( \Hcal_C \otimes \bigoplus_{\nu' < \nu} V(\lambda)_{\nu'} \big) \cong \Hcal_C \otimes V(\lambda)_\nu. 
\]
Since $ \theta_\pi(\Ocal(G_q)) \subseteq \Lbb(\Hcal_{\pi_C}) \otimes \pi_\lambda(U_q(\bfrak_-)) $, 
this is naturally a representation of $ \Ocal(G_q) $, and the action is given by
\begin{align*}
u^\mu _{ij} \cdot (\xi \otimes v) &= \pi_C(u^\mu _{ij}) (\xi) \otimes \pi_\lambda(K_{-\epsilon_j^\mu})(v) \\
&= \pi_C(u^\mu _{ij})(\xi) \otimes q^{-(\epsilon_j^\mu, \nu)} v \\
&= \pi_C(K_{-\nu} \triangleright u^\mu_{ij})(\xi) \otimes v 
= (\pi_C \otimes K_{-\nu}) (u^\mu _{ij})(\xi) \otimes v
\end{align*}
for $ \xi \otimes v \in \Hcal_C \otimes V(\lambda)_\nu $ and any $ \mu \in \weights^+$ and $ 1 \leq i, j \leq n_\mu $. Therefore, the above subquotient is isomorphic to the representation $(\pi_C \otimes K_{-\nu})^{ \oplus \dim V(\lambda)_\nu}$. Since $\nu \in \Pbf(\lambda)$ was arbitrary, this implies
\begin{equation}\label{eq:Ker theta pi is contained in the intersection}
\Ker(\theta_\pi) \subseteq \bigcap_{\nu \in \Pbf(\lambda)} \Ker ( \pi_C \otimes K_{-\nu}),
\end{equation}
and hence
\begin{equation*} 
\Prim(\theta_\pi) \supseteq \Big\{ P \in \Prim ( \Ocal(G_q)) \; \Big| \; \bigcap_{\nu \in \Pbf(\lambda)} \Ker ( \pi_C \otimes K_{-\nu} ) \subseteq P \Big\}. 
\end{equation*}
To prove the reverse inclusion let $ f \in \bigcap_{\nu \in \Pbf(\lambda)} \Ker(\pi_C \otimes K_{-\nu}) $. If we arrange the basis $ e_1^\lambda, \dots, e_{n_\lambda}^\lambda $ of $ V(\lambda) $ such that $ \epsilon_i^\lambda > \epsilon_j^\lambda $ implies $ i < j $, then the matrix of operators
\[
\theta_\pi (f) \in M_{n_\lambda} ( \Lbb (\Hcal_{\pi_C}) )
\]
is lower triangular with vanishing diagonal entries. 
This implies
\[
\Big(\bigcap_{\nu \in \Pbf(\lambda)} \Ker(\pi_C \otimes K_{-\nu}) \Big)^{n_\lambda} \subseteq \Ker(\theta_\pi).
\]
Therefore, if $ P \in \Prim (\theta_\pi) $ then $\Big( \bigcap_{\nu \in \Pbf(\lambda)} \Ker ( \pi_C \otimes K_{-\nu} ) \Big)^{n_\lambda} \subseteq P$. Since any primitive ideal is prime, it follows that 
\[
\bigcap_{\nu \in \Pbf(\lambda)} \Ker ( \pi_C \otimes K_{-\nu} ) \subseteq P
\]
as required.
\end{proof}

\begin{rmk} \label{rmk:concluding discussion}
Theorem~\ref{thm:the best that Ker theta pi can do} shows that the information about $ \Prim(\Ocal(G_q)) $ which can be obtained via the kernels of holomorphic evaluations is captured by the action of the weight lattice $ \weights $ on $ \SPrim(\Cf^\infty(K_q)) \subseteq \Prim(\Ocal(G_q)) $. 
    
We do not know precisely which primitive ideals are contained in $ \Prim(\theta_\pi) $ in general, since the gluing of the Joseph strata with respect to the Zariski topology is not known. 

However, in the case $ G_q = SL_q(2,\Cbb) $, the primitive ideal space of $ \Ocal(G_q) = \Ocal(SL_q(2, \Cbb)) $ can be described explicitly, see \cite[Example~2.2]{Goodearl2022}, \cite[Example II.1.3]{BrownGoodearl}. 
Theorem~\ref{thm:the best that Ker theta pi can do} implies that  
\begin{multline*}
\bigcup_{\pi \in \SSp(\Cf^\infty_c(SL_q(2,\Cbb)))} \Prim(\theta_\pi) = \\
\Prim_{(e,e)}\big(\Ocal(SL_q(2, \Cbb))\big) 
\cup 
\big\{K_\lambda \triangleright P \mid P \in \SPrim_{w_0}\big( \Cf^\infty(SU_q(2))\big), \, \lambda \in \weights \big\}.
\end{multline*}
It follows that we have strict inclusions 
\begin{equation*}
\bigcup_{\pi \in \SSp(\Cf_c^\infty(G_q))} \Prim(\theta_\pi) 
\subsetneqq \coprod_{w \in W} \Prim_{(w,w)}(\Ocal(G_q)) \subsetneqq \Prim(\Ocal(G_q))
\end{equation*}
in this case. In particular, the sets $ \Prim(\theta_\pi) $ do not cover the primitive ideal space of $ \Ocal(G_q) $. 
At the same time, we note that 
$ \bigcup_{\pi\in\SSp(\Cf^\infty_c(G_q))}
\Prim(\theta_\pi) $ is Zariski dense in $ \Prim(\Ocal(G_q)) $ for general $ G_q $ since $ \SPrim(\Cf^\infty (K_q)) $ separates the elements of $ \Cf^\infty(K_q) = \Ocal(G_q) $. 
\end{rmk}

As a consequence of Theorem \ref{thm:the best that Ker theta pi can do}, we obtain a generalisation of the second part of Proposition~\ref{prop:kernels of q-holomorphic evaluations are not primitive}. Before stating this, we record the following result from \cite{Joseph}.

\begin{lem} \label{lem:Joseph's lemma}
Let $ \wrm, \wrm' \in W \times W $. If $ J \in \Prim_\wrm(\Ocal(G_q)) $ and $ \Ical \subseteq \Prim_{\wrm'}(\Ocal(G_q)) $ satisfy
\[
\bigcap_{P \in \Ical} P \subseteq J,
\]
then $ \wrm \leq \wrm' $ with respect to the product order.
\end{lem}

\begin{proof}
This is a direct consequence of \cite[Lemma~10.3.6]{Joseph},  applied to a subset of $ \Prim(\Ocal(G_q)) $ contained entirely in the stratum $ \Prim_{\wrm'}(\Ocal(G_q)) $.
\end{proof}

\begin{cor} \label{cor:kernels of quantum holomorphic evaluations are not primitive-general case}
Let $ \lambda \in \weights^+ $ and consider $ \pi = \pi_C \otimes \pi_\lambda \in \SSp ( \Cf_c ^\infty (G_q)) $, where $ \pi_C: \Cf^\infty(K_q) \rightarrow \Lbb(\Hcal_C) $ is a topologically irreducible $ * $-representation. Then $ \Ker(\theta_\pi) $ is  primitive if and only if $ \lambda = 0 $. 
\end{cor}

\begin{proof}
According to Corollary~\ref{cor:no bdd repn for inhomogeneous strata} and Theorem~\ref{thm:irreducible *-representations of C(Kq)}, we have
\[
\Ker(\pi_C) \in \Prim_{(w,w)} (\Ocal(G_q))
\]
for some $ w \in W $, and it follows from Proposition~\ref{prop:primitive ideals are maximal in each stratum} that finite subsets are closed in $ \Prim_{(w,w)}(\Ocal(G_q)) $. 
Therefore Theorem~\ref{thm:the best that Ker theta pi can do} implies
\begin{equation}\label{eq:kernel of q-evaluation at a stratum}
\Prim_{(w,w)}(\Ocal(G_q)) \cap \Prim(\theta_\pi)
= \big\{\Ker(\pi_C \otimes K_{-\nu}) \mid \nu \in \weights(\lambda) \big\}.
\end{equation}
Now suppose that $ \Ker(\theta_\pi) $ is primitive. According to  ~\eqref{eq:stratification of primitive spectrum} and Theorem \ref{thm:no bdd repn for inhomogeneous strata}, we then have 
\[
\Ker(\theta_\pi) \in \Prim_{(w',w')}(\Ocal(G_q))
\]
for some $ w' \in W $. Since
\[
\Ker (\theta_\pi) \subseteq \Ker(\pi_C \otimes K_{-\nu} ) \in \Prim_{(w,w)}(\Ocal(G_q))
\]
for $ \nu \in \weights(\lambda) $ by \eqref{eq:Ker theta pi is contained in the intersection}, applying Lemma~\ref{lem:Joseph's lemma} to $ J = \Ker(\pi_C \otimes K_{-\nu}) $ for some fixed $ \nu \in \weights(\lambda) $ and the singleton set $ \Ical = \{\Ker(\theta_\pi) \} $ gives $ w \leq w' $. Moreover we have \( \Ker (\theta_\pi) \in \Prim(\theta_\pi) \) by the definition of $ \Prim(\theta_\pi) $, so that \eqref{eq:the best that Ker theta pi can do} yields 
\[
\bigcap_{\nu \in \weights(\lambda)} \Ker(\pi_C \otimes K_{-\nu} ) \subseteq \Ker(\theta_\pi). 
\]
Applying Lemma~\ref{lem:Joseph's lemma} to $ J = \Ker(\theta_\pi) $ and $ \Ical = \{\Ker(\pi_C \otimes K_{-\nu}) \mid \nu \in \weights(\lambda)\} $ therefore gives $ w' \leq w $. As a consequence, we obtain $ w = w' $.

In view of \eqref{eq:kernel of q-evaluation at a stratum} we thus have
\[
\{ \Ker(\theta_\pi) \} = \big\{ \Ker(\pi_C \otimes K_{-\nu}) \mid \nu \in \weights(\lambda) \big\}.
\]
This implies $ \lambda = 0 $ since the set on the right-hand side has cardinality greater than one whenever $ \lambda \neq 0 $, see Theorem~\ref{thm:primitive ideals in homogeneous strata}. 

For the converse implication, observe that if $ \lambda = 0 $ then $ \Ker(\theta_\pi) = \Ker(\pi_C) $ is clearly primitive. 
\end{proof}

Theorem~\ref{thm:the best that Ker theta pi can do} allows us to establish the following partial quantum analogue of the fact that classical holomorphic evaluations induce a bijection between $ G $ and $ \Max(\Pol^\mathbb{R}(K)) $. 

\begin{prop} \label{prop:embedding of SPrim(Cc(Gq)) into Prim(C(Kq))}
Consider
\[
\SPrim(\Cf_c^\infty(G_q)) = \big\{\Ker(\pi) \mid \pi \in \SSp(\Cf_c^\infty(G_q)) \big\},
\]
and write $ \mathfrak{P}\Prim(\Ocal(G_q)) $ for the power set of $ \Prim(\Ocal(G_q)) $. 
Then the map 
\begin{equation} \label{injectivecorrespondence}
\SPrim(\Cf_c^\infty (G_q)) \ni \Ker(\pi) \mapsto \Prim(\theta_\pi ) \in \mathfrak{P}\Prim(\Ocal(G_q))
\end{equation}
is well-defined and injective. 
\end{prop}

\begin{proof}
Let $ \pi \cong \pi_C \otimes \pi_\lambda $ and $ \pi' \cong \pi_C ' \otimes \pi_{\lambda'} $ be two topologically irreducible $ * $-representations of $ \Cf_c^\infty(G_q) $, and write $ w, w' \in W $ for the elements such that $ \Ker(\pi_C) \in \SPrim_{w}(\Cf^\infty(K_q)) $ and $ \Ker(\pi_C') \in \SPrim _{w'}(\Cf^\infty(K_q)) $.

Assume that 
\(
\Prim(\theta_\pi) = \Prim(\theta_{\pi'}).
\)
Then, according to Theorem~\ref{thm:the best that Ker theta pi can do}, we have
\[
\bigcap_{\nu \in \weights(\lambda)} \Ker(\pi_C \otimes K_{-\nu} ) \subseteq \Ker(\pi_C' \otimes K_{-\nu'} ), \qquad \forall \nu' \in \weights(\lambda'),
\]
and hence
\[
\bigcap_{\nu \in \weights ( \lambda)} \Ker (\pi_C \otimes K_{-\nu} ) \subseteq \bigcap_{\nu' \in \weights ( \lambda' )} \Ker (\pi_C ' \otimes K_{-\nu'} ).
\]
By symmetry, we obtain in fact
\[
\bigcap_{\nu \in \weights(\lambda)} \Ker(\pi_C \otimes K_{-\nu}) = \bigcap_{\nu' \in \weights(\lambda')} \Ker(\pi_C ' \otimes K_{-\nu'}).
\]
Due to Lemma \ref{lem:Joseph's lemma} this yields $ w = w' $. 
Moreover, according to Proposition~\ref{prop:primitive ideals are maximal in each stratum}, finite subsets of $\Prim_{(w,w)}(\Ocal(G_q)) $ are Zariski-closed within the stratum. Therefore, the previous equality shows in fact  
\begin{equation}\label{eq:equality between actual primitive ideals}
\big\{ \Ker (\pi_C \otimes K_{-\nu} ) \mid \nu \in \weights(\lambda) \big\} = \big\{ \Ker (\pi_C ' \otimes K_{-\nu'} ) \mid \nu' \in \weights(\lambda') \big\}.
\end{equation}
By Theorem~\ref{thm:SPrim(C(Kq))}, there exist $ \zeta, \zeta' \in \Tbb^{\Scal(w)} $ and $ z, z' \in \Tbb^{\Ical(w)} $ such that
\[
\Ker(\pi_C) = J_{(w,w), \zeta, z} \quad \text{and} \quad \Ker(\pi_C') = J_{(w,w), \zeta', z'}.
\]
Now let $ \nu \in \weights(\lambda) $. Since $ \Ker(\pi_C \otimes K_{-\nu}) = K_\nu \triangleright \Ker(\pi_C) $ is contained in the right-hand side of \eqref{eq:equality between actual primitive ideals}, the formulas for the left torus action \eqref{eq:torus action on the second parameter}--\eqref{eq:torus action on the third parameter} imply that there exists $ \nu' \in \weights(\lambda') $ such that
\[
q^{-2 (\nu, \varpi_i)} \zeta_i = q^{-2 (\nu', \varpi_i)} \zeta_i ' , \;\; \forall i \in \Scal(w) \quad \text{and} \quad q^{-(\nu, \varpi_i)} z_i = q^{- (\nu', \varpi_i)} z_i ', \;\; \forall i \in \Ical(w).
\]
Taking absolute values forces $ \nu = \nu' $. This shows in particular $ \zeta = \zeta' $ and $ z = z' $. Moreover, applying this argument to all $ \nu \in \weights(\lambda) $ yields \(\weights(\lambda) \subseteq \weights(\lambda') \). 
By symmetry, it follows that \(\weights(\lambda) = \weights(\lambda') \), and hence $ \lambda = \lambda' $. 

We thus obtain
\[
\Ker(\pi_C) = J_{(w,w), \zeta, z } = J_{(w,w), \zeta' , z'} = \Ker(\pi_C'),
\]
so that 
\begin{align}\label{eq:kernel identity}
\Ker(\pi) &= \Ker(\pi_C) \otimes \Df(K_q) + \Cf^\infty(K_q) \otimes \Ker(\pi_\lambda) \\
&= \Ker(\pi_C') \otimes \Df(K_q) + \Cf^\infty(K_q) \otimes \Ker(\pi_{\lambda'}) = \Ker(\pi') \nonumber
\end{align}
as required.

To prove that the correspondence \eqref{injectivecorrespondence} is well-defined, let
\[
\pi = \pi_C \otimes \pi_\lambda,
\qquad
\pi' = \pi_C' \otimes \pi_{\lambda'}
\in \SSp\bigl(\Cf_c^\infty(G_q)\bigr),
\]
and suppose that $\Ker(\pi)=\Ker(\pi')$, so that \eqref{eq:kernel identity} holds.

Assume first that $ \lambda \neq \lambda' $. Applying $ \id \otimes \pi_\lambda$ to the two middle expressions in \eqref{eq:kernel identity}, we obtain
\[
\Ker(\pi_C)\otimes \End(V(\lambda))
=
\Cf^\infty(K_q)\otimes \End(V(\lambda)).
\]
It follows that $\Ker(\pi_C)=\Cf^\infty(K_q)$, contradicting the fact that $\pi_C\neq 0$. Hence $ \lambda = \lambda ' $.

Applying $ \id \otimes \pi_\lambda$ once more to the two middle expressions in \eqref{eq:kernel identity} yields
\[
\Ker(\pi_C)\otimes \End(V(\lambda)) = \Ker(\pi_C')\otimes \End(V(\lambda)).
\]
Therefore, we also have $ \Ker(\pi_C)=\Ker(\pi_C')$. As a consequence, Theorem~\ref{thm:the best that Ker theta pi can do} implies 
$
\Prim(\theta_\pi)=\Prim(\theta_{\pi'}), 
$
and this finishes the proof. 
\end{proof}

Note that according to Corollary~\ref{cor:degeneracy of primitive ideals} the correspondence
\[
\SSp(\Cf^\infty(K_q)) \ni \pi \longmapsto \Ker(\pi) \in \SPrim(\Cf^\infty(K_q))
\]
is not injective. Therefore the argument in the proof of Corollary ~\ref{prop:embedding of SPrim(Cc(Gq)) into Prim(C(Kq))} does not allow us to conclude that \( \pi_C \cong \pi_C' \). Due to this degeneracy, we cannot reformulate Proposition~\ref{prop:embedding of SPrim(Cc(Gq)) into Prim(C(Kq))}  with \( \SSp(\Cf_c^\infty(G_q)) \) instead of $ \SPrim(\Cf_c^\infty(G_q)) $.

\begin{rmk} \label{rmk:embedding of SPrim(Cc(Gq)) into Prim(C(Kq))}
As indicated above, Proposition~\ref{prop:embedding of SPrim(Cc(Gq)) into Prim(C(Kq))} can be interpreted as a quantum analogue of the inclusion \( \supseteq \) in~\eqref{eq:classical relation between Prim(Cf(K)) and holomorphic evaluations}.

In Remark \ref{rmk:concluding discussion} we have already observed in the special case $ G_q = SL_q(2, \mathbb{C}) $ that there is no quantum analogue of the reverse inclusion. This is true for arbitrary $ G_q $ as well. Specifically, 
let
\[
J = J_{(w_0,w_0), \zeta , z} \in \bigcup_{\pi \in \SSp(\Cf_c^\infty(G_q))} \Prim(\theta_\pi).
\]
Then, due to Theorem~\ref{thm:the best that Ker theta pi can do}, there exist \( \pi_C \in \SSp(\Cf^\infty(K_q)) \) and \( \lambda \in \weights^+ \) such that
\[
\bigcap_{\nu \in \weights(\lambda)} \Ker(\pi_C \otimes K_{-\nu}) \subseteq J,
\]
which implies according to Lemma~\ref{lem:Joseph's lemma} that
\[
\Ker(\pi_C) \in \SPrim_{w_0}(\Cf^\infty(K_q)).
\]
Since finite subsets of \( \Prim_{(w_0, w_0)}(\Ocal(G_q)) \) are Zariski-closed within the stratum, we conclude that
\[
J = \Ker(\pi_C \otimes K_{-\nu}) \qquad \text{for some } \nu \in \weights(\lambda).
\]
As a consequence, we obtain 
\begin{multline*}
\Prim_{(w_0, w_0)}(\Ocal(G_q)) \cap \bigcup_{\pi \in \SSp(\Cf_c^\infty(G_q))} \Prim(\theta_\pi) \\
= \left\{ \Ker(\pi_{w_0} \otimes K_{\zeta + \nu}) \mid \zeta \in \tf^*/i \hbar^{-1} \roots^\vee,\ \nu \in \weights \right\}
\neq \Prim_{(w_0, w_0)}(\Ocal(G_q)),
\end{multline*}
compare Theorem~\ref{thm:irreducible *-representations of C(Kq)} and Theorem~\ref{thm:primitive ideals in homogeneous strata}. 
In summary, the classical relation \eqref{eq:classical relation between Prim(Cf(K)) and holomorphic evaluations} does not carry over to the $ q $-deformed situation.
\end{rmk}

Instead of looking at primitive ideals, one may also study the behaviour of quantum holomorphic evaluations directly at the level of representations. In this setting we have the following variant of Proposition~\ref{prop:embedding of SPrim(Cc(Gq)) into Prim(C(Kq))}.

\begin{prop} \label{prop:holomorphic evaluations are distinct for different points}
Let $ \pi, \pi' \in \SSp(\Cf_c^\infty(G_q)) $. If the representations $ \theta_\pi $ and $ \theta_{\pi'} $ are boundedly equivalent then $ \pi $ and $ \pi' $ are unitarily equivalent.
\end{prop}

\begin{proof}
Let $ \pi = \pi_C \otimes \pi_\lambda $ and $ \pi' = \pi_{C} ' \otimes \pi_{\lambda'} \in \SSp(\Cf_c^\infty(G_q)) \cong \SSp(\Pol^\Rbb (G_q)) $. Let $ T: \Hcal_{\pi_C} \otimes V(\lambda) \rightarrow \Hcal_{\pi_{C} '} \otimes V(\lambda') $ be an invertible bounded operator such that
\begin{equation}\label{eq:unitary equivalence of two holomorphic evaluations}
T \theta_\pi (f) T^{-1} = \theta_{\pi'}(f), \quad \forall f \in \Cf^\infty(K_q).
\end{equation}
This implies $ \Ker(\theta_\pi) = \Ker(\theta_{\pi'}) $ and hence $ \Prim(\theta_\pi) = \Prim(\theta_{\pi'}) $.
From the proof of Proposition~\ref{prop:embedding of SPrim(Cc(Gq)) into Prim(C(Kq))} we thus get $ \lambda = \lambda' $. Therefore it suffices to show that $ \pi_C $ and $ \pi'_C $ are unitarily equivalent. 

By Theorem~\ref{thm:irreducible *-representations of C(Kq)}, there exist $ w, w' \in W $ and $ \zeta, \zeta' \in \tf^* $ such that $ \pi_C = \pi_w \otimes K_\zeta $, and $ \pi_C ' = \pi_{w'} \otimes K_{\zeta'} $. Moreover, the proof of Proposition~\ref{prop:embedding of SPrim(Cc(Gq)) into Prim(C(Kq))} yields $ w = w' $. Hence we only need to show that $ \zeta = \zeta' $ in $ \tf^*/i\hbar^{-1}\Qbf^\vee\cong T. $ 

In order to check this let $ \mu \in \weights^+ $. Applying \eqref{eq:unitary equivalence of two holomorphic evaluations} to $ f = c^+_{w,\mu}=u^\mu_{i_{w(\mu)} 1} $, compare \eqref{eq:w-multiplicative elements}, yields
\[
T
\Bigl(
q^{(\zeta,\mu)}
\pi_w(c^+_{w,\mu})
\otimes
\pi_\lambda(K_{-\mu})
\Bigr)
T^{-1}
=
q^{(\zeta',\mu)}
\pi_w(c^+_{w,\mu})
\otimes
\pi_\lambda(K_{-\mu}).
\]
By
\eqref{eq:w-multiplicative elements under tensor product decomposition}, the operator $ \pi_w(c^+_{w,\mu}) $ is nonzero, positive and self-adjoint. Moreover, $ \pi_\lambda(K_{-\mu}) $ is positive and invertible. Hence
\[
A_\mu = \pi_w(c^+_{w,\mu}) \otimes \pi_\lambda(K_{-\mu})
\]
is a nonzero positive operator. Since conjugation by the invertible operator $ T $ preserves the spectrum, the above identity implies
\[
q^{(\zeta,\mu)}\sigma(A_\mu) = q^{(\zeta',\mu)}\sigma(A_\mu).
\]
Noting that the spectrum of $ A_\mu $ satisfies $\{ 0 \} \subsetneqq \sigma(A_\mu) \subseteq \Rbb_+$, and $ q^{(\zeta - \zeta', \mu)} $ is a complex number of modulus $1$, this allows us to conclude $ q^{(\zeta - \zeta', \mu)} = 1 $. Since $ \mu \in \weights^+ $ was arbitrary we therefore obtain $ \zeta - \zeta' \in i \hbar^{-1} \roots^\vee $ as required.
\end{proof}

\begin{rmk}
Consider $ \pi = \pi_C \otimes \hat{\epsilon} \in \SSp(\Cf_c^\infty(G_q)) $ where $ \hat{\epsilon}: \Dfrak(K_q) \rightarrow \Cbb $ is the counit. Note that $ \hat{\epsilon} $ is the representation corresponding to the weight $ \lambda = 0 \in \weights^+ $ in our discussion above, and that $ \theta_\pi = \pi_C \in \SSp(\Cf^\infty(K_q)) $. Hence, in view of Proposition~\ref{prop:holomorphic evaluations are distinct for different points} we have 
\[
\SSp(\Cf^\infty(K_q))
\subseteq
\bigl\{ \theta_\pi \;\big|\; \pi \in \SSp(\Cf_c^\infty(G_q)) \bigr\} \big/ \sim
\;\cong\;
\SSp(\Cf_c^\infty(G_q)),
\]
where $ \sim $ denotes equivalence of bounded representations. This inclusion may be viewed as a quantum analogue of the classical relation $ K \subseteq G $ at the level of topologically irreducible $ * $-representations.
\end{rmk}

We have seen in Proposition~\ref{prop:kernels of q-holomorphic evaluations are not primitive} that quantum holomorphic evaluations typically fail to be topologically irreducible.   
One may wonder whether these representations are at least \emph{indecomposable} in a suitable sense. 

Let us say that a bounded representation $ \theta: A \rightarrow \Lbb(\Hcal) $ is \emph{orthogonally indecomposable} if there is no orthogonal direct sum decomposition $ \Hcal = \Hcal_0 \oplus \Hcal_1 $
into bounded subrepresentations on nontrivial closed subspaces $ \Hcal_0, \Hcal_1 \subseteq \Hcal $. 

\begin{prop} \label{orthogonalindecomposability}
For any $ \pi = \pi_C \otimes \pi_\lambda \in \SSp(\Cf_c^\infty(G_q)) $ the quantum holomorphic evaluation $ \theta_\pi $ is orthogonally indecomposable. 
\end{prop}

\begin{proof}
By Theorem~\ref{thm:generalisation of Podles and Woronowicz}, we may regard $\pi$ as a topologically irreducible unital $*$-representation of $\Pol^\Rbb(G_q)$. Suppose that
\[
\Hcal_\pi = \Hcal_0 \oplus \Hcal_1
\]
is an orthogonal direct sum of closed $\theta_\pi(\Ocal(G_q))$-invariant subspaces, and let $P \in \Lbb(\Hcal_\pi)$ be the orthogonal projection onto $\Hcal_0$. Since both summands are invariant, we have
\[
P \theta_\pi ( f ) = \theta_\pi ( f ) P, \qquad f \in \Ocal(G_q).
\]
Taking adjoints and using $P = P^*$, we also obtain
\[
P \theta_\pi ( f )^* = \theta_\pi ( f )^* P, \qquad f \in \Ocal(G_q).
\]
Since $\theta_\pi = \pi \circ i_-$ and $i_-(\Ocal(G_q))$ generates $\Pol^\Rbb(G_q)$ as a $*$-algebra, these imply
\[
P\pi(x) = \pi(x)P, \qquad x \in \Pol^\Rbb(G_q).
\]
In particular, $\Hcal_0 = P\Hcal_\pi$ is a closed $\pi(\Pol^\Rbb(G_q))$-invariant subspace. By topological irreducibility of $\pi$, it follows that $\Hcal_0 = 0 $ or $\Hcal_0 = \Hcal_\pi$.
\end{proof}

Note that for any $ \pi = \pi_C \otimes \pi_\lambda \in \SSp(\Cf_c^\infty(G_q)) $, the associated quantum holomorphic evaluation may also be viewed as a purely algebraic representation $ \theta_\pi: \Ocal(G_q) \rightarrow \End(D_C \otimes V(\lambda)) $, where $ D_C \subseteq \Hcal_C $ is the dense linear subspace described in Theorem \ref{thm:irreducible *-representations of C(Kq)}. 
In contrast to Proposition \ref{orthogonalindecomposability}, 
we have the following example showing that, at this level, indecomposability in general fails. 

\begin{prop}\label{prop:some quantum holomorphic evaluations are decomposable}
Let
\[
\pi = \pi_{w_0} \otimes \pi_\varpi \in \SSp\big(\Cf_c^\infty (SL_q(2,\Cbb)) \big)
\cong \SSp\big( \Cf^\infty(SU_q(2)) \big) \times \SSp\big( \Df(SU_q(2)) \big),
\]
and let $ D \subset \ell^2 (\Nbb) $ be the dense subspace spanned by the standard orthonormal basis.
Then the quantum holomorphic evaluation
\[
\theta_\pi: \Ocal(SL_q(2, \Cbb)) \longrightarrow \End\big(D \otimes V(\varpi) \big)
\]
decomposes into a direct sum 
\[
\theta_\pi \cong (\pi_{w_0} \otimes K_\varpi ) \oplus (\pi_{w_0} \otimes K_{-\varpi})
\]
as abstract representations. 
\end{prop}

\begin{proof}
Recall that we denote the fundamental representation matrix by
\[
u^\varpi = \begin{pmatrix} a & b \\ c & d \end{pmatrix}.
\]
We denote the standard orthonormal basis of $\ell^2(\Nbb)$ by $\{ e_n \mid n \in \Nbb \}$. Recalling that $l^- ( c^*) \in U_q (\mathfrak{sl} (2,\Cbb)) $ lowers weights by $ \alpha = 2 \varpi $, we may write $ V(\varpi) = \Cbb v_\varpi \oplus \Cbb v_{-\varpi} $, where $ v_{\pm \varpi} $ are weight vectors of weight $ \pm \varpi $ such that
\[
v_{-\varpi} = \pi_\varpi l^- (-qc^*) v_\varpi.
\]
Consider the cyclic submodules
\[
M_\pm = \theta_\pi \big(\Ocal(SL_q(2, \Cbb)) \big) ( e_0 \otimes v_{\pm\varpi} ) \subseteq D \otimes V(\varpi).
\]
From the definition of $ \theta_\pi $ we immediately see that $M_- = D \otimes V(\varpi)_{-\varpi} $ and that $ M_- $ is an irreducible representation isomorphic to $ \pi_{w_0} \otimes K_{\varpi} $, compare Proposition~\ref{prop:kernels of q-holomorphic evaluations are not primitive}.

We claim that $ M_+ $ is an irreducible representation equivalent to $ \pi_{w_0} \otimes K_{- \varpi} $. 
Note first that the generators of $ \Ocal(SL_q(2, \Cbb)) $ act on $ e_0 \otimes v_\varpi $ by
\begin{align}\label{eq:action of theta pi a star}
\theta_\pi(a) (e_0 \otimes v_\varpi) &= \pi_{w_0} (a) e_0 \otimes \pi_\varpi l^- ( a ) v_\varpi + \pi_{w_0} (-qc)^* e_0 \otimes \pi_\varpi l^- (c) v_\varpi = 0,  \\
\theta_\pi ( c) (e_0 \otimes v_\varpi ) &= \pi_{w_0} (c) e_0 \otimes \pi_\varpi l^-(a) v_\varpi + \pi_{w_0}(a)^* e_0 \otimes \pi_\varpi l^-(c) v_\varpi = q^{-\frac{1}{2}} e_0  \otimes v_\varpi, \nonumber \\
\theta_\pi ( c^* ) (e_0 \otimes v_\varpi ) &= \pi_{w_0} (a) e_0 \otimes \pi_\varpi l^- ( c^*) v_\varpi + \pi_{w_0} (c^*) e_0 \otimes \pi_\varpi l^- (a^*) v_\varpi = q^{\frac{1}{2}} e_0  \otimes v_\varpi, \nonumber \\
\theta_\pi (a^*) (e_0 \otimes v_\varpi ) &= \pi_{w_0} ( c) e_0 \otimes \pi_\varpi l^- ( - q c^* )v_\varpi + \pi_{w_0} (a^*) e_0 \otimes \pi_\varpi l^- ( a^*) v_\varpi \nonumber \\
= e_0 \otimes v_{-\varpi} &+ (1-q^2)^{\frac{1}{2}} q^{\frac{1}{2}} e_1 \otimes v_\varpi = e_0 \otimes v_{-\varpi} + \Big( (\pi_{w_0} \otimes K_{-\varpi} ) (a^*) e_0 \Big) \otimes v_{\varpi}. \nonumber
\end{align}
Since $ \Ocal(SL_q(2, \Cbb)) $ admits a basis
\[
\big\{ a^k (c^*)^m c^n  \mid k, m, n \in \Nbb \big\} \cup \big\{ (a^*)^k (c^*)^m c^n \mid k, m, n \in \Nbb, \, k \neq 0 \big\},
\]
where
\[
\begin{pmatrix}
a & -q c^* \\ 
c & a^*
\end{pmatrix}
= u^\varpi =
\begin{pmatrix}
u^\varpi_{11} & u^\varpi_{12} \\ 
u^\varpi_{21} & u^\varpi_{22}
\end{pmatrix},
\]
these formulas imply that
\begin{equation*}
M_+ = \theta_\pi(\Ocal(SL_q(2, \Cbb))) (e_0 \otimes v_{\varpi}) = \Span \big\{ \theta_\pi(a^*)^k (e_0 \otimes v_\varpi) \mid k \in \Nbb \}.
\end{equation*}
Consider the intertwiner $ T: M_+ \rightarrow \big(D \otimes V(\varpi) \big)/M_- $ obtained by composing the inclusion map $ M_+ \rightarrow D \otimes V(\varpi) $ with the canonical projection onto $ (D \otimes V(\varpi))/M_- $. The quotient $ (D \otimes V(\varpi))/M_- $ is equivalent to the irreducible representation $ \pi_{w_0} \otimes K_{-\varpi} $, compare the proof of Theorem~\ref{thm:the best that Ker theta pi can do}. 
Since $ T( e_0 \otimes v_{\varpi}) \neq 0 $ it follows that $ T $ is surjective.

Next let $ x = \sum_{k \in \Nbb} \lambda_k \theta_\pi(a^*)^k (e_0 \otimes v_\varpi) \in \Ker(T) \subseteq M_+ $. A repeated application of \eqref{eq:action of theta pi a star} shows that
\[
\theta_\pi( a^*)^k ( e_0 \otimes v_\varpi ) \in \Big( (\pi_{w_0} \otimes K_{-\varpi}) (a^*)^k e_0 \Big) \otimes v_{\varpi} + M_- , \qquad \forall k \in \Nbb.
\]
We therefore have
\[
0 + M_- = T(x) = \Big( \sum_{k \in \Nbb} \lambda_k (\pi_{w_0} \otimes K_{-\varpi}) (a^*)^k e_0 \Big) \otimes v_{\varpi} + M_-. 
\]
This implies $ \lambda_k = 0 $ for all $ k \in \Nbb $, since the vectors $ (\pi_{w_0} \otimes K_{-\varpi}) (a^*)^k e_0 \in D $ are linearly independent. We conclude that $ x = 0 $, and therefore $ \Ker(T) = 0 $. This means that $ T $ is an isomorphism, so that $ M_+ $ is equivalent to $ \pi_{w_0} \otimes K_{- \varpi} $. In particular $ M_+ $ is irreducible. 

Note that $M_+ \cap M_- = 0$ since $\Ker(T) = 0 $. 
Moreover, we observe that $ T(e_0 \otimes v_{\varpi}) $ generates $ D \otimes V(\varpi)/M_- $ and $ e_0 \otimes v_{-\varpi} $ generates $ M_- $ as a representation of $ \Ocal(SL_q(2, \Cbb)) $, so that $ M_+ + M_- = D \otimes V(\varpi) $. 

In summary, we obtain a direct sum decomposition 
\(
D \otimes V(\varpi) = M_+ \oplus M_-  
\)
and 
\[
\theta_\pi \cong \left( \pi_{w_0} \otimes K_{-\varpi} \right) \oplus \left( \pi_{w_0} \otimes K_{\varpi} \right)
\]
as claimed. 
\end{proof}

\section*{Funding}
H. Lee was supported by the National Natural Science Foundation of China (Grant No.~12371138 and W2441002) and by the National Research Foundation of Korea (NRF) grants funded by the Korea government (MSIT) (No.RS-2025-00561391 and No.RS-2022-NR069971).

\section*{Acknowledgments}
Most of this work was carried out while H. Lee was affiliated with the Institute for Advanced Study in Mathematics of Harbin Institute of Technology and was completed after he began a postdoctoral position at the Research Institute of Mathematics at Seoul National University. H. Lee would like to thank both institutes for their hospitality.

H. Lee is deeply grateful to Simeng Wang, Xiao Xiong, and Quanhua Xu for their generous support and encouragement during his postdoctoral stay in Harbin. He would also like to express his sincere gratitude to Hun Hee Lee and Sang-Gyun Youn for their invaluable support and kindness during his postdoctoral stay in Seoul. Their support has meant a great deal to him.

\section*{Declaration of Generative AI Use}

The authors used ChatGPT solely for proofreading and minor language editing of the completed manuscript. No AI tools were used in the development of the main content of the paper.

\printbibliography 

\end{document}